\documentclass[11pt]{amsart}
\usepackage[latin1]{inputenc}
\usepackage{amsmath}
\usepackage{amsfonts}
\usepackage{amssymb}
\usepackage{graphicx}
\usepackage{amsthm}
\usepackage{amssymb}
\usepackage{cite}
\usepackage{mathrsfs}
\usepackage{hyperref}
\usepackage{paralist}
\usepackage{a4wide}
\usepackage{enumitem}
\usepackage{tabularx}
\usepackage[all,cmtip]{xy}
\usepackage{tikz}
\usetikzlibrary{matrix,chains}
\begin{document}
	\newcommand{\Nb}{\mathbf{N}}
	\newcommand{\B}{\mathbf{B}}
	\newcommand{\U}{\mathbf{U}}
	\newcommand{\T}{\mathbf{T}}
	\newcommand{\G}{\mathbf{G}}
	\newcommand{\Para}{\mathbf{P}}
	\newcommand{\Levi}{\mathbf{L}}
	\newcommand{\Y}{\mathbf{Y}}
	
	\newcommand{\Gtilde}{\mathbf{\tilde{G}}}
	\newcommand{\Ttilde}{\mathbf{\tilde{T}}}
	\newcommand{\Btilde}{\mathbf{\tilde{B}}}
	
	\newcommand{\opp}{\operatorname{opp}}
	\newcommand{\N}{\operatorname{N}}
	\newcommand{\Z}{\operatorname{Z}}
	\newcommand{\Gal}{\operatorname{Gal}}
	\newcommand{\Kerel}{\operatorname{ker}}
	\newcommand{\Irr}{\operatorname{Irr}}
	\newcommand{\D}{\operatorname{D}}
	\newcommand{\I}{\operatorname{I}}
	\newcommand{\GL}{\operatorname{GL}}
	\newcommand{\SL}{\operatorname{SL}}
	\newcommand{\W}{\operatorname{W}}
	\newcommand{\R}{\operatorname{R}}
	\newcommand{\C}{\operatorname{C}}
	\newcommand{\Ind}{\operatorname{Ind}}
	\newcommand{\Res}{\operatorname{Res}}
	\newcommand{\Hom}{\operatorname{Hom}}
	\newcommand{\End}{\operatorname{End}}
	\newcommand{\Ho}{\operatorname{Ho}}
	\newcommand{\Br}{\operatorname{Br}}
	\newcommand{\br}{\operatorname{br}}
	\newcommand{\ad}{\operatorname{ad}}

	\theoremstyle{remark}

\theoremstyle{definition}
\newtheorem{definition}{Definition}[section]
\newtheorem{construction}[definition]{Construction}
\newtheorem{remark}[definition]{Remark}
\newtheorem{example}[definition]{Example}
\newtheorem{notation}[definition]{Notation}
\newtheorem{question}[definition]{Question}

\theoremstyle{plain}
\newtheorem{theorem}[definition]{Theorem}
\newtheorem{lemma}[definition]{Lemma}
\newtheorem{proposition}[definition]{Proposition}
\newtheorem{corollary}[definition]{Corollary}
\newtheorem{conjecture}[definition]{Conjecture}
\newtheorem{assumption}[definition]{Assumption}
\newtheorem{main theorem}[definition]{Main Theorem}
\newtheorem{hypothesis}[definition]{Hypothesis}
	
	\newtheorem*{theo*}{Theorem}
	\newtheorem*{conj*}{Conjecture}
	\newtheorem*{cor*}{Corollary}
	\newtheorem*{propo*}{Proposition}
	
	\newtheorem{theo}{Theorem}
	\newtheorem{conj}[theo]{Conjecture}
	\newtheorem{cor}[theo]{Corollary}
	\newtheorem{propo}[theo]{Proposition}
	
	\renewcommand{\thetheo}{\Alph{theo}}
	\renewcommand{\theconj}{\Alph{conj}}
	\renewcommand{\thecor}{\Alph{cor}}

	\title{Quasi-isolated blocks and the Alperin-McKay conjecture}
	
	\date{\today}
	\author{Lucas Ruhstorfer}
	\address{Fachbereich Mathematik, TU Kaiserslautern, 67653 Kaiserslautern, Germany}
	\email{ruhstorfer@mathematik.uni-kl.de}

	
	\begin{abstract}
		The Alperin--McKay conjecture is a longstanding open conjecture in the representation theory of finite groups.
		Späth showed that the Alperin--McKay conjecture holds if the so-called inductive Alperin--McKay (iAM) condition holds for all finite simple groups. In a previous paper, the author has proved that it is enough to verify the inductive condition for quasi-isolated blocks of groups of Lie type. In this paper we show that the verification of the iAM-condition can be further reduced in many cases to isolated blocks. As a consequence of this we obtain a proof of the Alperin--McKay conjecture for $2$-blocks of finite groups with abelian defect.
	\end{abstract}

\maketitle

\section*{Introduction}

\subsection*{Alperin--McKay conjecture}

In the representation theory of finite groups some of the most important conjectures predict a very strong relationship between the representations of a finite group $G$ and certain representations of its $\ell$-local subgroups, where $\ell$ is a prime dividing the order of $G$. One of these conjectures is the Alperin--McKay conjecture.
For an $\ell$-block $b$ of $G$ we denote by $\Irr_0(G,b)$ the set of height zero characters of $b$. Then the Alperin--McKay predicts the following:

	\begin{conj*}[Alperin--McKay]
	Let $b$ be an $\ell$-block of $G$ with defect group $D$ and $B$ its Brauer correspondent in $\mathrm{N}_G(D)$. Then $$|\Irr_0(G,b)|= |\Irr_0(\mathrm{N}_G(D),B)|.$$
\end{conj*}

Späth \cite[Theorem C]{IAM} showed that the Alperin--McKay conjecture holds if the so-called inductive Alperin--McKay condition holds for all finite groups. In a previous article, the author has reduced the verification of the inductive Alperin--McKay condition to so-called quasi-isolated blocks of groups of Lie type \cite{Jordan2}. The overall aim of this article is to further reduce the verification of this condition to isolated blocks of groups of Lie type. Using this, we are then able to verify the inductive Alperin--McKay condition for many important classes of blocks.

\subsection*{Equivariant Bonnafé--Dat--Rouquier equivalence}
One of our main ingredients towards such a reduction is the recent result by Bonnafé--Dat--Rouquier \cite{Dat}.
 They have constructed a Morita equivalence which can be seen as a modular analogue of Lusztig's Jordan decomposition for characters. Let $\G$ be a simple, simply connected algebraic group  with Frobenius endomorphism $F: \G \to \G$ defining an $\mathbb{F}_q$-structure where $q$ is a power of a prime $p$. Fix a prime $\ell$ different from $p$ and let $(\mathcal{O},K,k)$ be an $\ell$-modular system as in \ref{rep theory} below. In the following $\Lambda$ denotes either the discrete valuation ring $\mathcal{O}$ or its residue field $k$. Let $(\G^\ast,F^\ast)$ be a group in duality with $(\G,F)$ and $s \in (\G^\ast)^{F^\ast}$ a semisimple element of $\ell'$-order. Let $e_s^{\G^F} \in \mathrm{Z}(\Lambda \G^F)$ be the central idempotent associated to $s$ as in \cite[Theorem 9.12]{MarcBook}. We assume that $\Levi^\ast$ is the minimal Levi subgroup of $\G^\ast$ containing $\mathrm{C}_{\G^\ast}^\circ(s)$. Let $N$ be the common stabilizer in $\G^F$ of the idempotent $e_s^{\Levi^F}$ and $\Levi$ and suppose that $N/\Levi^F$ is cyclic. According to the main result of \cite{Dat} there exists a Morita equivalence between $\Lambda N e_s^{\Levi^F}$ and $\Lambda \G^F e_s^{\G^F}$. Using the methods developed in \cite{Jordan} we extend their result to incorporate automorphisms of $\G^F$.

\begin{theo}[see Theorem \ref{equiv}]\label{thm1}
	 Let $\G$ be a simple, simply connected algebraic group of type $B_n$, $C_n$ or $E_7$ such that either $n > 2$ or $q$ is odd. Let $\iota: \G \to \Gtilde$ be a regular embedding. Then there exists a Frobenius morphism $F_0: \G \to \G$ stabilizing $\Levi$ such that the image of $\Gtilde^F \rtimes \langle F_0 \rangle$ in the outer automorphism group of $\G^F$ is $\mathrm{Out}(\G^F)_{e_s^{\G^F}}$. There exists a Morita equivalence between $\Lambda N e_s^{\Levi^F}$ and $\Lambda \G e_s^{\G^F}$ which lifts to a Morita equivalence between $\Lambda \mathrm{N}_{\tilde{\G}^F \langle F_0 \rangle }(\Levi, e_s^{\Levi^F})  e_s^{\Levi^F}$ and $\Lambda \tilde{\G}^F \langle F_0 \rangle e_s^{\G^F}$.
\end{theo}

\subsection*{Quasi-isolated blocks of type $A$}
In the next part of our paper we consider the case of groups of type $A$. According to the main result of \cite{Jordan2} it is prove the inductive Alperin--McKay condition it suffices to consider strictly quasi-isolated block, i.e. blocks of $\Lambda \G^F e_s^{\G^F}$ such that $\C^\circ_{\G^\ast}(s) \C_{(\G^\ast)^{F^\ast}}(s)$ is not contained in a proper Levi subgroup of $\G^\ast$. Therefore we will from now on assume that $\G$ is of type $A$ and $s$ is a strictly quasi-isolated element. In this case the automorphism group of $\G^F$ is more complicated and the quotient group $N/\Levi^F$ can become arbitrary large with the rank of $\G$ increasing. Thus, a direct approach along the lines of Theorem \ref{thm1} does not seem possible.

Using the explicit description of quasi-isolated elements in groups of type $A$ by Bonnafé \cite{Bonnafe} we instead first construct a specific $F$-stable Levi subgroup $\Levi'$ of $\G$ containing $\Levi$. This Levi subgroup has the additional property that $N'/\Levi'$ is cyclic of prime order. Again, we denote by $N'$ be the common stabilizer of $\Levi'$ and $e_s^{\Levi'^F}$ in $\G^F$. The main result of \cite{Dat} is still applicable in this slightly more general situation and we obtain a Morita equivalence between $\Lambda N' e_s^{\Levi'^F}$ and $\Lambda \G^F e_s^{\G^F}$. This enables us to construct a certain abelian subgroup $\mathcal{A}$ of $\mathrm{Aut}(\Gtilde^F)$ (see Definition \ref{defA}) such that $\Gtilde^F \rtimes \mathcal{A}$ generates the stabilizer of $e_s^{\G^F}$ in $\mathrm{Out}(\G^F)$. This enables us the following result which can be seen as a version of Theorem \ref{thm1} for groups of type $A$.
%
%

\begin{theo}[see Corollary \ref{typeA}]\label{thm2}
	Assume that $\G$ is of type A and let $s \in (\G^\ast)^{F^\ast}$ be a strictly quasi-isolated element of $\ell'$-order. If $\ell \nmid |\mathcal{A}|$ then there exists a Morita equivalence between $\Lambda N' e_s^{\Levi'^F}$ and $\Lambda \G^F e_s^{\G^F}$ which lifts to a Morita equivalence between $\Lambda \mathrm{N}_{\G^F \mathcal{A}}(\Levi',e_s^{\Levi'^F}) e_s^{\Levi'^F}$ and $\Lambda \tilde{\G}^F \mathcal{A} e_s^{\G^F}$.
\end{theo}

\subsection*{Reduction to isolated blocks}
Using the methods developed in the proof of the main theorem of \cite{Jordan2} as a blueprint we use the result of Theorem \ref{thm2} to obtain a reduction of the verification of the iAM-condition to unipotent blocks of type $A$. By the work of \cite{CS14} and \cite{Brough} which together show that the inductive Alperin--McKay condition holds for unipotent blocks of type $A$ we obtain the following:

\begin{theo}[see Corollary \ref{corollaryA}]\label{thm3}
	The inductive Alperin--McKay condition holds for all $\ell$-blocks of quasi simple groups of type $A$, whenever $\ell \geq 5$.
\end{theo}

The statement of Theorem \ref{thm3} has been obtained in special cases by \cite{Brough}. As a biproduct of the reduction methods developed for the proof of Theorem \ref{thm3} and our equivariant Jordan decomposition from Theorem \ref{thm1} we obtain the following:

\begin{theo}[see Theorem \ref{maintheoremBC}]\label{thm4}
	Let $X$ be one of the symbols $B$ or $C$ and let $\ell \geq 5$. Assume that
all isolated $\ell$-blocks of quasi-simple group of type $X$ are AM-good relative to the Cabanes group (see \ref{centralizer}) of their defect group.
	Then all $\ell$-blocks of quasi-simple groups of type $X$ are AM-good.
\end{theo}

\subsection*{$2$-Blocks with abelian defect group}
Observe that Theorem \ref{thm3} and Theorem \ref{thm4} exclude the primes $\ell=2,3$. The first reason for this is that the structure of defect groups of $\ell$-blocks of groups of Lie type for $\ell < 5$ is vastly more complicated. The second reason is that the automorphism group $\mathcal{A}$ can have a non-cyclic Sylow $2$-subgroup. This makes it difficult to control the extension of characters and the block theory at the same time. By an explicit analysis of the $2$-blocks of groups of type $A$ we are able to overcome this difficulty for blocks with abelian defect group and obtain again in this case a reduction to unipotent blocks. According to the classification result of \cite{twoblocks} almost all quasi-isolated $2$-blocks of groups with abelian defect group arise as blocks of groups of type $A$. Using this we are able to prove the inductive Alperin--McKay condition for all quasi-isolated $2$-blocks with abelian defect group. In fact we are able to prove this condition for the slightly larger class of $2$-blocks of $\G^F$ whose defect group $D$ is almost abelian in $\G^F$, see Definition \ref{almost abelian}. The advantage of working with this larger class of blocks is that we obtain a reduction theorem to quasi-simple groups of the Alperin--McKay conjecture for blocks with almost abelian defect group. Following the proof of \cite{IAM} we are able to show the following:

\begin{propo}[see Proposition \ref{modifiedSpaeth}]\label{prop}
	Let $X$ be a finite group and $\ell$ a prime. Assume that for every non-abelian simple subquotient $S$ of $X$ with $\ell \mid |S|$ the following holds: Every $\ell$-block of the universal covering group $H$ of $S$ with almost abelian defect group satisfies the iAM-condition. Then the Alperin--McKay conjecture holds for any $\ell$-block of $X$ with almost abelian defect.
\end{propo}

 As a consequence of Proposition \ref{prop} and the verification of the inductive Alperin--McKay condition for the necessary $\ell$-blocks we obtain the following result.

\begin{theo}[see Theorem \ref{Alperin McKay}]\label{thm5}
	The Alperin--McKay conjecture holds for all $2$-blocks with almost abelian defect group.
\end{theo}

By \cite[Proposition 5.6]{Robinson} and \cite[Theorem 1.1]{KessarMalle} we obtain the following immediate consequence of Theorem \ref{thm5}.

\begin{theo}
	The Alperin weight conjecture holds for all $2$-blocks with abelian defect group.
\end{theo}

\section*{Acknowledgement}

The author would like to thank the Isaac Newton Institute for Mathematical Sciences for support and hospitality during the programme "Groups, representations and applications: new perspectives" when work on this paper was undertaken. This work was supported by: EPSRC grant number EP/R014604/1. The author would like to thank Julian Brough for insightful discussions and suggestions.

\section{General properties of the Bonnafé--Dat--Rouquier equivalence}

In the following we follow the notation of \cite{Jordan} and \cite{Jordan2}. For the convenience of the reader we recall the most important notions.

\subsection{Representation theory}\label{rep theory}

Let $\ell$ be a prime and $K$ be a finite field extension of $\mathbb{Q}_\ell$. We assume in the following that $K$ is large enough for the finite groups under consideration. Let $\mathcal{O}$ be the ring of integers of $K$ over $\mathbb{Z}_\ell$ and $k=\mathcal{O}/J(\mathcal{O})$ its residue field. We will use $\Lambda$ (respectively $A$) to interchangeably denote $\mathcal{O}$ or $k$ (respectively $K$ or $\mathcal{O}$).

\subsection{Groups of Lie type}
Let $\G$ be a connected reductive group with Frobenius endomorphism $F: \G \to \G$ defined over $\overline{\mathbb{F}}_p$ for some prime $p \neq \ell$.  Given such a group $\G$ it is often convenient to consider it as a closed subgroup of a group whose center is connected. Therefore, we fix a regular embedding $\iota: \G \hookrightarrow \Gtilde$ of $\G$ as in \cite[Section 15.1]{MarcBook} and we identify $\G$ with its image in $\Gtilde$. For any closed subgroup $\mathbf{M}$ of $\G$ we define $\tilde{\mathbf{M}}:=\mathbf{M} \mathrm{Z}(\tilde \G)$. Moreover, if $\mathbf{H}$ is any closed $F$-stable subgroup of $\Gtilde$ then we denote by $H$ its subset of $F$-stable points $\mathbf{H}^F$.

\subsection{Godement resolutions and $\ell$-adic cohomology}

Let $\mathbf{X}$ be a variety defined over an algebraic closure of $\mathbb{F}_p$ endowed with an action of a finite group $G$. By work of Rickard and Rouquier there exists an object $G\Gamma_c(\mathbf{X},\Lambda)$ in $\mathrm{Ho}^b(\Lambda G\text{-} \mathrm{perm})$, the bounded homotopy category of $\ell$-permutation $\Lambda G$-modules. Its $i$th cohomology groups are denoted $H_c^i(\mathbf{X},\Lambda)$ and we abbreviate $H_c^{\mathrm{dim}(X)}(\mathbf{X},\Lambda)$ by $H_c^{\mathrm{dim}}(\mathbf{X},\Lambda)$.

\subsection{Deligne--Lusztig induction}

 Suppose that $\Para$ is a parabolic subgroup of $\G$ with Levi decomposition $\Para= \Levi \ltimes \U$. We consider the Deligne--Lusztig variety $\Y_\U^\G$. Its cohomology groups $H_c^{i}(\Y_\U^\G,\Lambda)$ induce an additive map $R_{\Levi \subset \Para}^{\G}: \mathbb{Z} \Irr(\Levi^F)  \to \mathbb{Z} \Irr(\G^F)$ the so-called Lusztig induction. In the case under consideration the map $R_{\Levi \subset \Para}^{\G}$ will not depend on the choice of the parabolic subgroup $\Para$ and thus we will write $R_{\Levi}^{\G}$ for $R_{\Levi \subset \Para}^{\G}$ in the following.

\subsection{The Bonnafé--Dat--Rouquier equivalence}
 Let $\G^\ast$ be the Langlands dual of $\G$ with Frobenius endomorphism $F^\ast: \G^\ast \to \G^\ast$ dual to $F$. We fix a semisimple element $s \in (\G^\ast)^{F^\ast}$ of $\ell'$-order and as in \cite[Theorem 9.12]{MarcBook} let $e_s^{\G^F} \in \mathrm{Z}(\Lambda \G^F e_s^{\G^F})$ be the central idempotent associated to it. Assume that $\Levi^\ast$ is an $F^\ast$-stable Levi subgroup of $\G^\ast$ containing $\mathrm{C}_{\G^\ast}^\circ(s)$. Additionally, suppose that $\Levi^\ast \C_{\G^\ast}(s)^{F^\ast}=  \C_{\G^\ast}(s)^{F^\ast} \Levi^\ast$ and define $\mathbf{N}^\ast := \C_{\G^\ast}(s)^{F^\ast} \Levi^\ast$. Let $\Levi$ be a Levi subgroup of $\G$ in duality with $\Levi^\ast$ and denote by $\mathbf{N}$ the subgroup of $\mathrm{N}_{\G}(\Levi)$ which corresponds to the subgroup $\mathbf{N}^\ast$ of $\mathrm{N}_{\G^\ast}(\Levi^\ast)$ under the isomorphism $\mathrm{N}_{\G}(\Levi)/ \Levi  \cong \mathrm{N}_{\G^\ast}(\Levi^\ast)/ \Levi^\ast$ given by duality. Throughout this article we assume that $\Nb/\Levi \cong \Nb^F/\Levi^F$ is cyclic.
 
 Suppose that $\Para$ is a parabolic subgroup of $\G$ with Levi decomposition $\Para= \Levi \ltimes \U$ and consider the bimodule $H_c^{\mathrm{dim}}(\Y_\U^\G,\Lambda) e_s^{\Levi^F}$.
 Since $\Nb^F/\Levi^F$ is assumed to be cyclic (and of $\ell'$-order by \cite[Corollary 2.9]{Bonnafe}) and $H_c^{\mathrm{dim}}(\Y_\U^\G,\Lambda) e_s^{\Levi^F}$ is $\Nb^F$-stable by \cite[Theorem 7.2]{Dat} there exists a $\Lambda (\G^F \times (\Nb^F)^{\mathrm{opp}} \Delta \tilde{\Nb}^F)$-module $M'$ extending the $\Lambda (\G^F \times (\Levi^F)^{\mathrm{opp}} \Delta \tilde{\Levi}^F)$-module $H_c^{\mathrm{dim}}(\Y_\U^\G,\Lambda) e_s^{\Levi^F}$, see \cite[Lemma 10.2.13]{Rouquier3}.
 
 For the following theorem recall that for any complex $\mathcal C \in \mathrm{Comp}^b(A)$ there exists (see for instance \cite[2.A.]{Dat}) a complex $\mathcal C^{\mathrm{red}}$ with $\mathcal C \cong \mathcal C^{\mathrm{red}}$ in $\mathrm{Ho}^b(A)$ such that $\mathcal C^{\mathrm{red}}$ has no non-zero direct summand which is homotopy equivalent to $0$.
 
 \begin{theorem}[Bonnaf\'e--Dat--Rouquier]\label{Boro}
There exists a complex $\mathcal C'$ of $\mathcal{O} \G^F$-$\mathcal{O} \Nb^F $-bimodules extending $G\Gamma_c(\Y^\G_{\U}, \Lambda)^{\operatorname{red}}e_s^{\Levi^F}$ such that $H^{d}(\mathcal C')\cong M'$ where $d:= \mathrm{dim}(\Y_\U^\G)$. The complex $\mathcal C'$ induces a splendid Rickard equivalence between $\mathcal{O} \G^F e_s^{\G^F}$ and $\mathcal{O} \Nb^F e_s^{\Levi^F}$ and the bimodule $M'$ induces a Morita equivalence between $\mathcal{O} \G^F e_s^{\G^F}$ and $\mathcal{O} \Nb^F e_s^{\Levi^F}$.
 \end{theorem} 
 
 \begin{proof}
 	In the proof of \cite[Theorem 7.5]{Dat} use the fact that $M'$ extends $H_c^{\dim}(\Y_\U^\G,\Lambda) e_s^{\Levi^F}$ instead of \cite[Proposition 7.3]{Dat}. The rest of the proof of the theorem is as in \cite[Section 7]{Dat}. 
 \end{proof}
 
Note that the Morita equivalence in Theorem \ref{Boro} might depend on the particular choice of the extension $M'$. In particular, this is also the case for the character bijection $R:\Irr(\Nb^F,e_s^{\Levi^F}) \to \Irr(\G^F,e_s^{\G^F})$ induced by the Morita bimodule $M'$.

\subsection{Some local properties of the Bonnafé--Dat--Rouquier Morita equivalence}
Let $\mathcal C'$ be a complex of $\Lambda(\G^F \times (\Nb^F)^{\opp})$-modules inducing the splendid Rickard equivalence between $\Lambda \G^F e_s^{\G^F}$ and $\Lambda \Nb^F e_s^{\Levi^F}$ as in Theorem \ref{Boro} above. Then $\mathcal{C}'$ induces a bijection $c \mapsto b$ between the blocks of $\Lambda \Nb^F e_s^{\Levi^F}$ and $\Lambda \G^F e_s^{\G^F}$ where $b$ is defined as the unique block such that $b \mathcal C' c$ is not homotopy equivalent to $0$. We fix a block $c$ of $\Lambda \Nb^F e_s^{\Levi^F}$ and we let $(Q,c_Q)$ be a $c$-Brauer pair. There exists a unique $b$-Brauer pair $(Q,b_Q)$ such that the complex $\mathrm{Br}_{\Delta Q}(\mathcal C') c_Q$ induces a Rickard equivalence between $k \mathrm{C}_{\G^F}(Q) b_Q$ and $k \mathrm{C}_{\Nb^F}(Q) c_Q$, see \cite[Theorem 7.7]{Dat}.

\begin{lemma}\label{multiplicity free}
	Suppose that $Q$ is an $\ell$-subgroup of $\Levi^F$. Then the $\mathrm{C}_{\G^F}(Q) \times \mathrm{C}_{\Levi^F}(Q)^{\mathrm{opp}}$-bimodule $H^{\mathrm{dim}}_c(\Y_{\C_\U(Q)}^{\mathrm{C}_{\G}(Q)},\Lambda) \mathrm{br}_Q(e_s^{\Levi^F})$ is multiplicity-free.
\end{lemma}

\begin{proof}

According to the proof of \cite[Theorem 5.2]{Rickard} there exists a unique complex $\mathcal C_Q'$ of $\Lambda(\C_{\G^F}(Q) \times \C_{\Nb^F}(Q)^{\opp})$-modules lifting the complex $\Br_{\Delta Q}(\mathcal C')$ from $k$ to $\Lambda$. Note that 
	$$\Res_{\C_{\G^F}(Q) \times \C_{\Levi^F}(Q)^{\opp} }^{\C_{\G^F}(Q) \times \C_{\Nb^F}(Q)^{\opp}} (\mathrm{Br}_{\Delta Q}(\mathcal C)) \cong G \Gamma_c(\Y_{\C_\U(Q)}^{\mathrm{C}_{\G}(Q)},\Lambda) \mathrm{br}_Q(e_s^{\Levi^F})$$ in $\mathrm{Ho}^b(k(\mathrm{C}_{\G^F}(Q) \times \mathrm{C}_{\Levi^F}(Q)^{\mathrm{opp}}))$. Let $M'_Q:=H^{d_Q}(\mathrm{Br}_{\Delta Q}(\mathcal{C}_Q'))$, where $d_Q:=\mathrm{dim}( \Y_{\C_\U(Q)}^{\C_\G(Q)})$ is the unique degree in which $\mathrm{Br}_{\Delta Q}(\mathcal{C})$ has non-zero cohomology. Then $M'_Q c_Q$ induces a Morita equivalence between $\Lambda \mathrm{C}_{\G^F}(Q) b_Q$ and $\Lambda \mathrm{C}_{\Nb^F}(Q) c_Q$. In particular, $M'_Q c_Q$ is multiplicity free. Since this is true for all Brauer pairs of all blocks $c$ of $\Lambda \Nb^F e_s^{\Levi^F}$ it follows that $M'_Q$ is multiplicity free as well. The module $M'_Q$ is an extension of $H^{\mathrm{dim}}_c(\Y_{\C_\U(Q)}^{\mathrm{C}_{\G}(Q)},\Lambda) \mathrm{br}_Q(c_Q)$ and $\Nb^F/\Levi^F$ is cyclic of $\ell'$-order the result follows from Clifford theory.
\end{proof}

\subsection{Extending the action of complexes}

%
%
%
%

 We need to slightly strengthen \cite[Theorem 7.6]{Dat} by including the diagonal action of $\tilde N$.

\begin{lemma}\label{diagonal action}
	Assume that $\ell \nmid |H^1(F,\mathrm{Z}(\G))|$. Then there exists a complex $\mathcal C'$ of $\Lambda (\G^F \times (\Nb^F)^{\mathrm{opp}} \Delta \tilde{\Nb}^F)$-modules such that $H^d(\mathcal C') \cong M'$ and $\mathcal C'$ induces a splendid Rickard equivalence between $\Lambda \G^F e_s^{\G^F}$ and $\Lambda \Nb^F e_s^{\Levi^F}$.
\end{lemma}

\begin{proof}
	Let us first assume that $\Lambda =k$. We consider the complex of $\ell$-permutation $k(\G^F \times (\Levi^F)^{\opp} \Delta (\tilde{\Levi}^F))$-modules $ \mathcal{C}:=G\Gamma_c(\Y_\U,k)^{\mathrm{red}}e_s^{\Levi^F}$. Consider an $\ell$-subgroup $\tilde R$ of $\G^F \times (\Levi^F)^{\opp} \Delta(\tilde{\Levi}^F)$. Let $Z:=\mathrm{Z}(\tilde{\G}^F)_\ell$ and denote $R:=\tilde{R} \cap (\G^F \times (\Levi^F)^{\opp})$. By assumption $\ell \nmid |H^1(F,\mathrm{Z}(\G))|=|\Gtilde^F/\G^F\mathrm{Z}(\Gtilde^F)|$ and therefore $\tilde{R} \subset R \Delta(Z)$.
	The subgroup $\Delta(Z)$ centralizes the variety $\Y_\U^\G$. Hence, $\mathcal{C}$ can be regarded as complex of $(\Gtilde^F \times (\tilde{\Levi}^F)^{\opp})/ \Delta(Z)$-modules. Since  $\tilde{R} \subset R \Delta(Z)$ we deduce that $\Br_{R}(\mathcal{C}) \cong \Br_{\tilde R}(\mathcal{C})$. Using \cite[Corollary 3.8]{Dat} we deduce that $\Br_{R}(\mathcal{C})$ is acyclic unless $R$ is conjugate to a subgroup of $\Delta \Levi^F$. Thus, $\tilde R$ is conjugate to a subgroup of $\Delta \tilde{\Levi}^F$. By \cite[Lemma A.2]{Dat} we deduce that the indecomposable summands of $\mathcal{C}$ are all contained in $\Delta \tilde{\Levi}^F$.
	 
	 Consider now the complex $\mathrm{End}^\bullet_{k\G^F}(\mathcal{C})$ of $k(\Levi^F \times(\Levi^F)^{\opp} \Delta(\tilde{\Levi}^F))$-bimodules. The same argument as above show that $\Br_{\tilde R}(\mathrm{End}^\bullet_{k\G^F}(\mathcal{C})) \cong \Br_{R}(\mathrm{End}^\bullet_{k\G^F}(\mathcal{C}))$. The cohomology of the latter complex is concentrated in degree $0$ only by the proof of Step 1 of \cite[Theorem 7.5]{Dat}. Hence, by \cite[Theorem A.3]{Dat} we have $\End_{k \G^F}^\bullet(\mathcal{C}) \cong \End_{D^b(k \G^F)}(\mathcal{C})$ in $\Ho^b(k(\Levi^F \times (\Levi^F)^{\opp} \Delta \tilde{\Levi}^F)).$
	
	The rest of the proof is almost identical to the proof of \cite[Theorem 7.6]{Dat}. For completeness we provide most of the details here.
	Denote $\mathcal{C}'=\Ind_{\G^F\times(\Levi^F)^{\opp} \Delta \tilde{\Levi}^F }^{\G^F\times(\Nb^F)^{\opp} \Delta \tilde{\Levi}^F}(\mathcal{C})$. 
	Let $\mathcal P$ be a projective resolution of $k\Nb^F$, i.e., a
	complex of projective $k(\Nb^F\times(\Nb^F)^{\opp} \Delta \tilde{\Levi}^F)$-modules such that its terms $\mathcal P^i=0$ for $i>0$, together
	with a quasi-isomorphism $\mathcal P\to k\Nb^F$ of $k(\Nb^F\times(\Nb^F)^{\opp} \Delta \tilde{\Levi}^F)$-modules.
	
	Let $\mathcal X$ be a complex of $k(\G^F \times (\Nb^F)^{\mathrm{opp}} \Delta \tilde{\Levi}^F)$-modules. Observe that we can consider the complex $\End_{kG}^\bullet(\mathcal X)$ as complex of $k(\Nb^F \times (\Nb^F)^{\mathrm{opp}} \Delta \tilde{\Levi}^F)$-modules. We have a natural isomorphism $\End_{k(\G^F \times (\Nb^F)^{\mathrm{opp}} \Delta \tilde{\Levi}^F)}^\bullet(\mathcal X) \cong \Hom^\bullet_{k (\Nb^F \times (\Nb^F)^{\opp} \Delta \tilde{\Levi}^F)}(k \Nb^F, \End_{kG}^\bullet(\mathcal{X}))$. The terms of $\mathcal C'$ are projective $k\G^F$-modules. Therefore as in the proof of \cite[Theorem 7.6]{Dat} we can consider the following commutative diagram
	$$\xymatrix{
		\End_{\Ho^b(k(\G^F\times(\Nb^F)^{\opp} \Delta \tilde{\Levi}^F))}(\mathcal{C}') \ar[r] \ar[d]_\cong & 
		\End_{D^b(k(\G^F\times(\Nb^F)^{\opp} \Delta \tilde{\Levi}^F))}(\mathcal{C}') \ar[d]^\cong \\
		\Hom_{\Ho^b(k(\Nb^F\times (\Nb^F)^{\opp} \Delta \tilde{\Levi}^F))}(k\Nb^F,\End_{k\G^F}^\bullet(\mathcal{C}')) \ar[r] &
		\Hom_{\Ho^b(k(\Nb^F\times (\Nb^F)^{\opp} \Delta \tilde{\Levi}^F))}(\mathcal P,\End_{k\G^F}^\bullet(\mathcal{C}'))
	}$$
	
	Using the isomorphisms of complexes in $\Ho^b(k(\Nb^F\times
	(\Nb^F)^{\opp} \Delta \tilde{\Levi}^F))$
	$$\End_{k\G^F}^\bullet(\mathcal{C}') \cong  \Ind_{\Levi^F\times (\Levi^F)^{\opp} \Delta \tilde{\Levi}^F}^{\Nb^F\times (\Nb^F)^{\opp} \Delta \tilde{\Levi}^F}(
	\End_{k\G^F}^\bullet(C))$$
	and
	$$\End_{D^b(k\G^F)}(\mathcal{C}') \cong  \Ind_{\Levi^F\times (\Levi^F)^{\opp} \Delta \tilde{\Levi}^F}^{\Nb^F\times (\Nb^F)^{\opp} \Delta \tilde{\Levi}^F}(
	\End_{D^b(k\G^F)}(C)),$$
	we deduce that
	$$\End_{k\G^F}^\bullet(\mathcal{C}') \cong \End_{D^b(k\G^F)}(\mathcal{C}')\text{ in }\Ho^b(k(\Nb^F\times
	(\Nb^F)^{\opp} \Delta \tilde{\Levi}^F)).$$
	Now, the canonical map
	$$\Hom_{\Ho^b(k(\Nb^F\times (\Nb^F)^{\opp}))}(k\Nb^F,\End_{D^b(k\G^F)}(\mathcal{C}')) \to
	\Hom_{\Ho^b(k(\Nb^F\times (\Nb^F)^{\opp}))}(\mathcal P,\End_{D^b(k\G^F)}(\mathcal{C}'))$$
	is an isomorphism, since $\End_{D^b(k\G^F)}(\mathcal{C}')$ is a complex concentrated in degree $0$.
	It follows that the top horizontal map in the commutative diagram
	above is an isomorphism, hence we have canonical isomorphisms
	$$\End_{\Ho^b(k(\G^F\times(\Nb^F)^{\opp} \Delta(\tilde{\Levi}^F)))}(\mathcal{C}')\xrightarrow{\cong}
	\End_{D^b(k(\G^F\times(\Nb^F)^{\opp} \Delta(\tilde{\Levi}^F)))}(\mathcal{C}')\xrightarrow{\cong}
	$$
	$$
	\End_{k(\G^F\times(\Nb^F)^{\opp} \Delta \tilde{\Levi}^F)}(\Ind_{\G^F\times(\Levi^F)^{\opp} \Delta(\tilde{\Levi}^F)}^{\G^F\times(\Nb^F)^{\opp} \Delta (\tilde{\Levi^F})}
	H_c^d(\Y_\U,k) e_s^{\Levi^F})
	.$$
	Using the proof of Step 3 of \cite[Theorem 7.6]{Dat} we deduce that there exists a summand $\tilde{\mathcal C}$ of $\mathcal{C}'$ which is quasi-isomorphic to $M'$. As in Step 4 and 5 of the proof of \cite[Theorem 7.6]{Dat} we see that $\mathrm{Res}_{\G^F\times (\Levi^F)^{\opp} \Delta \tilde{\Levi}^F}^{\G^F\times (\Nb^F)^{\opp} \Delta \tilde{\Levi}^F} (\tilde{\mathcal{C}}) \cong \mathcal{C}$ in $\Ho^b(k({\G^F\times (\Levi^F)^{\opp} \Delta \tilde{\Levi}^F}))$. Furthermore, Step 5 of the proof of \cite[Theorem 7.6]{Dat} shows that $\tilde{\mathcal C}$ induces a splendid Rickard equivalence between $k \G^F e_s^{\G^F}$ and $k \Nb^F e_s^{\Levi^F}$.
	
	Finally, let us consider the case $\Lambda=\mathcal{O}$. Using \cite[Lemma 5.1]{Rickard} together with the arguments in the first paragraph of the proof of \cite[Theorem 5.2]{Rickard} we observe that there exists a unique complex of $\mathcal{O} (\G^F \times (\Levi^F)^{\opp} \Delta \tilde{\Levi}^F)$-modules lifting $\tilde{\mathcal C}$. Moreover by the proof of \cite[Theorem 5.2]{Rickard} this complex induces a splendid Rickard equivalence between $\mathcal{O} \G^F e_s^{\G^F}$ and $\mathcal{O} \Nb^F e_s^{\Levi^F}$.
\end{proof}

\subsection{The Bonnafé--Dat--Rouquier equivalence and character correspondences}\label{character bijection}

Let $M:=H^{\mathrm{dim}}_c(\Y_\U^\G,\Lambda) e_s^{\Levi^F}$ considered as $X:=\G^F \times (\Levi^F)^{\mathrm{opp}} \Delta \tilde{\Levi}^F$-module and let $M'$ be a $Y:=\G^F \times (\Nb^F)^{\mathrm{opp}} \Delta \tilde{\Nb}^F$-module extending $M$. Define $\tilde{M}=\mathrm{Ind}_X^{\tilde{X}}(M)$, where $\tilde{X}:=\tilde{\G}^F \times (\tilde{\Levi}^F)^{\mathrm{opp}}$. Observe that $\tilde{M} \cong H^{\mathrm{dim}}_c(\Y_\U^{\tilde{\G}},\Lambda) e_s^{\Levi^F}$.

Recall from the proof of \cite[Theorem 7.5]{Dat} that there exists a central idempotent $e \in \Z(\Lambda \tilde{\Levi}^F)$ such that $\sum_{n \in \Nb^F/\Levi^F} {}^n e=e_s^{\Levi^F}$. Moreover, the induction functor yields a Morita equivalence between $\Lambda \tilde{\Levi}^F e$ and $\Lambda \tilde{\Nb}^F e_s^{\Levi^F}$. From this it follows that the right action of $\tilde{\Levi}^F$ on $\tilde{M}$ extends to $\tilde{\Nb}^F$ and the extended bimodule $\tilde{M}':= \tilde{M} e \otimes_{\Lambda \tilde{\Levi}^F} \Lambda \tilde{\Nb}^F$ induces a Morita equivalence between $\Lambda \tilde{\Nb}^F e_s^{\Levi^F}$ and $\Lambda \tilde{\G}^F e_s^{\G^F}$. The proof of \cite[Theorem 7.5]{Dat} shows that $\mathrm{Ind}_Y^{\tilde{Y}}(M') \cong \tilde{M}'$.
Writing 
$$\tilde{R}: \Irr(\tilde{\Nb}^F,e_s^{\Levi^F}) \to \Irr(\tilde {\G}^F, e_s^{\G^F})$$ for the character bijection induced by the Morita bimodule $\tilde{M}'$ we therefore obtain: 

\begin{lemma}\label{BDR}
For every $\psi \in \Irr(\tilde{\Nb}^F,e_s^{\Levi^F})$ there exists a unique character $\lambda \in \Irr(\tilde{\Levi}^F,e)$ such that $\mathrm{Ind}_{\tilde{\Levi}^F}^{\tilde{\Nb}^F}(\lambda)=\psi$ and we have $R_{\tilde{\Levi}^F}^{\tilde{\G}^F}(\lambda)=\tilde{R} \mathrm{Ind}_{\tilde{\Levi}^F}^{\tilde{\Nb}^F}(\lambda)$.
\end{lemma}

We now prove a local version of this result. Let $\tilde{b}$ be a block of $\Lambda \tilde{\G}^F e_s^{\G^F}$ covering $b$ and let $\tilde{c}$ be the unique block of $\Lambda \tilde{\Nb}^F e_s^{\Levi^F}$ corresponding to it under the Morita equivalence given by $\tilde{M}'$. We let $\tilde f :=\tilde c e$ be the unique block of $\Lambda \tilde{\Levi}^F e$ below $\tilde{c}$.
We assume that $Q$ is a characteristic subgroup of $D$ and consider the $X_Q:=\N_{\G^F}(Q) \times \N_{\Levi^F}(Q)^{\mathrm{opp}} \Delta \N_{\tilde{\Levi}^F}(Q)$-module $M_Q:=H^{\mathrm{dim}}_c(\Y_{\C_\U(Q)}^{\N_\G(Q)},\Lambda) \br_Q(e_s^{\Levi^F})$. There exists a $Y_Q:=\N_{\G^F}(Q) \times (\N_{\Nb^F}(Q))^{\mathrm{opp}} \Delta \N_{\tilde{\Nb}^F}(Q)$-module $M_Q'$ extending $M_Q$.

 Denote $\tilde{M}_Q=\Ind_{X_Q}^{\tilde X_Q} M_Q$, where $\tilde{X}_Q:=\N_{\tilde{\G}^F}(Q) \times \N_{\tilde{\Levi}^F}(Q)^{\mathrm{opp}}$. Observe that $\tilde{M}_Q \cong H^{\mathrm{dim}}_c(\Y_{\C_\U(Q)}^{\N_{\tilde{\G}}(Q) },\Lambda) \br_Q(e_s^{\Levi^F})$. Let $\tilde{C}_Q:=\br_Q(\tilde c)$ and $\tilde{B}_Q:=\br_Q(\tilde b)$ be the Harris--Knörr correspondents (in the sense of \cite[Corollary 1.18]{Jordan}) of $ \tilde{c}$ and $\tilde{b}$ in $\N_{\tilde{\Nb}^F}(Q)$ and $\N_{\tilde{\G}^F}(Q)$ respectively. Since $\tilde{f}$ corresponds to the block $\tilde{b}$ under the Morita equivalence given by $\tilde{M}$ it follows that $\tilde{M}_Q \tilde{F}_Q$ induces a Morita equivalence between $\Lambda \N_{\tilde{\Levi}^F}(Q) \tilde{F}_Q$ and $\Lambda \N_{\tilde{\G}^F}(Q) \tilde{B}_Q$. Recall that the stabilizer of the idempotent $e$ in $\tilde{\Nb}^F$ is $\tilde{\Levi}^F$. Consequently, the stabilizer of $\br_Q(e)$ in $\N_{\tilde{\Nb}^F}(Q)$ is $\N_{\tilde{\Levi}^F}(Q)$. Therefore, the bimodule $\tilde{M}'_Q:= \tilde{M}_Q \br_Q(e) \otimes_{\Lambda \N_{\tilde{\Levi}^F}(Q)} \Lambda \N_{\tilde{\Nb}^F}(Q)$ is an extension of $\tilde{M}_Q$ and $\tilde{M}_Q' \tilde{C}_Q \cong \tilde{M}_Q \tilde{F}_Q \otimes_{\Lambda \N_{\tilde{\Levi}^F}(Q)} \Lambda \N_{\tilde{\Nb}^F}(Q)$ induces a Morita equivalence between $\Lambda \N_{\tilde{\Nb}^F}(Q) \tilde{C}_Q$ and $\Lambda \N_{\tilde{\G}^F}(Q) \tilde{B}_Q$. We write
 $$\tilde R_Q: \Irr(\mathrm{N}_{\tilde{\Nb}^F}(Q), \tilde C_Q) \to \Irr(\mathrm{N}_{\tilde{\G}^F}(Q), \tilde B_Q)$$
for the associated character bijection. 
 
 \begin{lemma}\label{character theoretic BDR}
 	For every $\psi \in \Irr(\N_{\tilde{\Nb}^F}(Q),\tilde{C}_Q)$ there exists a unique character $\lambda \in \Irr(\N_{\tilde{\Levi}^F}(Q),\tilde{F}_Q)$ such that $\mathrm{Ind}_{\N_{\tilde{\Levi}^F}(Q)}^{\N_{\tilde{\Nb}^F}(Q)}(\lambda)=\psi$ and we have $R_{\N_{\tilde{\Levi}^F}(Q)}^{\N_{\tilde{\G}^F}(Q)}(\lambda)=\tilde{R}_Q \mathrm{Ind}_{\N_{\tilde{\Levi}^F}(Q)}^{\N_{\tilde{\Nb}^F}(Q)}(\lambda)$.
 \end{lemma}
 
 \begin{proof}
As in the global case we need to show that $\mathrm{Ind}_{Y_Q}^{\tilde{Y}_Q}(M_Q') \cong \tilde{M}_Q'$. This follows as in the proof of \cite[Theorem 7.5]{Dat}: Observe that it suffices to show that $\Res_{\tilde X_Q}^{\tilde{Y}_Q}(\mathrm{Ind}_{Y_Q}^{\tilde{Y}_Q}(M_Q')) \br_Q(e) \cong \Res_{\tilde X_Q}^{\tilde{Y}_Q}(\tilde{M}_Q') \br_Q(e)$. By Mackey's formula the left hand side is isomorphic to $\tilde{M}_Q$. Additionally, we have $\Res_{\tilde X_Q}^{\tilde{Y}_Q}(\tilde{M}_Q') \br_Q(e) \cong \tilde{M}_Q \br_Q(e)$.
  \end{proof}

\section{Descent of scalars}\label{sec 2}

\subsection{Restriction of scalars for Deligne--Lusztig varieties}

We assume until Section \ref{sec 3} that $F_0: \G \to \G$ is a Frobenius endomorphism stabilizing $\Levi$ which satisfies $F_0^r=F$ for some integer $r$ and $\gamma: \G \to \G$ is an automorphism commuting with $F_0$. In what follows $\mathcal{A}$ will denote the subgroup of $\mathrm{Aut}(\Gtilde^F)$ generated by $F_0$ and $\gamma$.

Let us recall the setup from \cite[Section 5]{Jordan}.
We consider the reductive group $\underline{\G}= \G^r$ with Frobenius endomorphism $F_0 \times \dots \times F_0: \underline{\G} \to \underline{\G}$ which we also denote by $F_0$. More generally, whenever $\sigma: \G \to \G$ is a bijective morphism of $\G$ then we also denote by $\sigma$ the induced map
$\sigma \times \dots \times \sigma: \underline{\G} \to \underline{\G}$
on $\underline{\G}.$ We consider the automorphism 
$$\tau: \underline{\G} \to \underline{\G}$$
given by $\tau(g_1,\dots,g_r)=(g_2,\dots,g_r,g_1)$.
Consider the projection onto the first component
$$\mathrm{pr}:\underline{\G} \to \G, \, (g_1,\dots,g_r)\mapsto g_1.$$
The restriction of $\mathrm{pr}$ to $\underline{\G}^{F_0 \tau}$ induces an isomorphism
$\mathrm{pr}:\underline{\G}^{F_0 \tau} \to \G^F$
of finite groups. For any subset $\mathbf{H}$ of $\G$ we set
$$\underline{\mathbf{H}}:=\mathbf{H} \times F_0^{r-1}(\mathbf{H}) \times \dots \times F_0(\mathbf{H}).$$
Note that if $\mathbf{H}$ is $F$-stable then $\underline{\mathbf{H}}$ is $\tau F_0$-stable and the projection map $\mathrm{pr}: \underline{\mathbf{H}} \to \mathbf{H}$ induces an isomorphism $\underline{\mathbf{H}}^{\tau F_0} \cong \mathbf{H}^F$. Conversely, one easily sees that any $\tau F_0$-stable subset of $\underline{\G}$ is of the form $\underline{\mathbf{H}}$ for some $F$-stable subset $\mathbf{H}$ of $\G$. 

We consider the $r$-fold product $\underline{\G^*}:=(\G^*)^r$ of the dual group $\G^\ast$ endowed with the Frobenius endomorphism $F_0^\ast:=F_0^* \times \dots \times F_0^*: \underline{\G^*} \to \underline{\G^*}$. Moreover, let
$$\tau^\ast: \underline{\G^\ast} \to \underline{\G^\ast}, \, (g_1,\dots,g_r) \mapsto (g_r,g_1\dots,g_{r-1}).$$
We denote by $\mathrm{pr}: \underline{\G^\ast} \to \G^\ast$ the projection onto the first coordinate.

\subsection{Restriction of scalars and Jordan decomposition of characters}\label{SDJD}


We let $\Para$ be a parabolic subgroup of $\G$ with Levi decomposition $\Para=\Levi \ltimes \U$. Then $\underline{\Para}$ is a parabolic subgroup of $\underline{\G}$ with Levi decomposition $\underline{\Para}=\underline{\Levi} \ltimes \underline{\U}$ such that $\tau F_0(\underline{\Levi})=\underline{\Levi}$. We can therefore consider the Deligne--Lusztig variety $\Y_{\underline{\U}}^{\underline{\G},F_0 \tau}$ which is a $\underline{\G}^{F_0 \tau} \times (\underline{\Levi}^{F_0 \tau})^{\mathrm{opp}}$-variety. Under the isomorphism $\underline{\G}^{\tau F_0} \cong \G^F$ given by $\mathrm{pr}$ we can consider it as a $\G^F \times (\Levi^F)^{\mathrm{opp}}$-variety. Endowed with this structure the projection map induces an isomorphism $$\mathrm{pr}: \Y_{\underline{\U}}^{\underline{\G},F_0 \tau} \to \Y_{\U}^\G$$
which is $\G^F \times (\Levi^F)^{\mathrm{opp}}$-equivariant, see \cite[Proposition 5.3]{Jordan}.

We consider the unipotent radical $\underline{\U}':=\U^r$ of the parabolic subgroup $\underline{\Para}'=\Para^r$ of $\underline{\G}$. Note that we have a Levi decomposition $\underline{\Para}'= \underline{\Levi} \ltimes \underline{\U}'$ in $\underline{\G}$ and the parabolic subgroup $\underline{\Para}'$ is $\tau$-stable. In the following we will use the language of parabolic subgroups and Levi subgroups of disconnected reductive groups as in \cite[Section 2.1]{Jordan}.

\begin{lemma}\label{mainlemma}
	Assume that $e_s^{\Levi^F}$ is $\mathcal{A}$-stable. If $\Para$ is $\gamma$-stable then the module $H^{\mathrm{dim}}_c(\Y_\U^\G,\Lambda) e_s^{\Levi^F}$ extends to a $\G^F \times (\Levi^F)^{\mathrm{opp}} \Delta (\tilde{\Levi}^F\mathcal{A})$-module.
	If in addition $\Nb^F/\Levi^F$ is centralized by $\mathcal{A}$ then $H^{\mathrm{dim}}_c(\Y_\U^\G,\Lambda) e_s^{\Levi^F}$ extends to $\G^F \times (\Nb^F)^{\mathrm{opp}} \Delta( \tilde{\Levi}^F \mathcal{A})$.
\end{lemma}

\begin{proof}
The pair $(\underline{\Levi},\underline{\Para}')$ is $\langle \tau, \gamma \rangle$-stable. Therefore, we can consider $\Y_{\underline{\U}'}^{\underline{\G}}$ as a $\underline{\G}^{\tau F_0} \times (\underline{\Levi}^{\tau F_0})^{\opp} \Delta( \tilde{\underline{\Levi}}^{\tau F_0} \langle \tau,\gamma \rangle)$-variety. We have an isomorphism
$$H^{\mathrm{dim}}_c(\Y_{\underline{\U}}^{\underline{\G}, \tau F_0},\Lambda) e_{\underline{s}}^{\underline{\Levi}^{\tau F_0}} \cong H^{\mathrm{dim}}_c(\Y_{\underline{\U'}}^{\underline{\G}, \tau F_0},\Lambda) e_{\underline{s}}^{\underline{\Levi}^{\tau F_0}}$$
of $\Lambda ((\underline{\G}^{\tau F_0} \times (\underline{\Levi}^{\tau F_0})^{\mathrm{opp}}) \Delta( \underline{\tilde{\Levi}}^F ))$-modules.
The projection $\mathrm{pr}:\underline{\G} \to \G$ onto the first coordinate defines an isomorphism
$$H_c^{\mathrm{dim}}(\Y_{\underline{\U}}^{\underline{\G},\tau F_0},\Lambda) e_{\underline{s}}^{\Levi^{\tau F_0}} \cong  H_c^{\mathrm{dim}}(\Y_\U^{\G,F},\Lambda) e_s^{L} $$
of $\G^F \times (\Levi^F)^{\mathrm{opp}} \Delta \tilde{\Levi}^F$-modules. By transport of structure, we can endow  $H_c^{\mathrm{dim}}(\Y_\U^{\G,F},\Lambda) e_s^{\Levi^F}$ with a $\G^F \times (\Levi^F)^{\mathrm{opp}} \Delta \tilde{\Levi}^F \langle F_0,\gamma \rangle$-structure.

Assume now that $\mathbf{N}/\Levi$ is centralized by $\mathcal{A}$. Let $n \in \mathbf{N}^{F}$ be a generator of the quotient group $\mathbf{N}^F/ \Levi^F$. Then we have $\sigma(n)n^{-1},F_0(n)n^{-1} \in \Levi^F$. Thus, conjugation by $\underline{n}$ defines an automorphism of $\underline{\G} \langle \tau, \gamma \rangle$.
Thus, by \cite[Lemma 3.1]{Jordan} the $\Lambda ((\underline{\G}^{\tau F_0} \times (\underline{\Levi}^{\tau F_0})^{\mathrm{opp}}) \Delta( \underline{\tilde{\Levi}}^F \langle \tau, \gamma \rangle ))$-module $H^{\mathrm{dim}}_c(\Y_{\underline{\U'}}^{\underline{\G}, \tau F_0}) e_{\underline{s}}^{\underline{\Levi}^{\tau F_0}}$ is $\underline{n}$-stable. By transport of structure we deduce that $H_c^{\mathrm{dim}}(\Y_\U^{\G,F},\Lambda) e_s^{\Levi^F}$ is $\mathbf{N}^F$-stable as $\G^F \times (\mathbf{L}^F)^{\mathrm{opp}} \Delta (\tilde{\Levi}^F \mathcal{A})$-module. Thus, it extends to a $\G^F \times (\mathbf{N}^F)^{\opp} \Delta \tilde{\Levi}^F \langle F_0,\gamma \rangle$-module by \cite[Lemma 10.2.13]{Rouquier3}.
\end{proof}

%
In the following we abbreviate $\mathcal{N}:=\mathrm{N}_{\Gtilde^F \mathcal{A}}(\Levi,e_s^{\Levi^F})$.

\begin{lemma}\label{mainlemma:local}
	Suppose that we are in the situation of Lemma \ref{mainlemma}. Let $Q$ be an $\ell$-subgroup of $\Levi^F$. If $\Nb^F/\Levi^F$ is centralized by $\mathcal{A}$ then $H^{\mathrm{dim}}_c(\Y_{\mathrm{C}_\U(Q)}^{\mathrm{N}_{\G}(Q)},\Lambda) \mathrm{br}_Q(e_s^{\Levi^F})$ extends to an $\mathrm{N}_{\G^F}(Q) \times \mathrm{N}_{\Levi^F}(Q)^{\mathrm{opp}} \Delta \mathrm{N}_{\mathcal{N}}(Q)$-module.
\end{lemma}

\begin{proof}
Recall that the projection map $\mathrm{pr}:\underline{\G}^{\tau F_0} \to \G^F$ is an isomorphism of finite groups. If $H$ is a subgroup of $\G^F$ we let $\underline{H}:=\mathrm{pr}^{-1}(H)$ and if $x \in \Lambda H$ then we let $\underline{x}:= \mathrm{pr}^{-1}(x) \in \Lambda \underline{H}$.

The quotient group $\mathrm{N}_{\mathcal{N}}(Q)/ \mathrm{N}_{\tilde{\Levi}^F \mathcal{A} }(Q)$ is cyclic as it embeds into $\mathcal{N}/\tilde{\Levi}^F \mathcal{A} \cong \Nb^F/\Levi^F$. Let $x \in \Nb^F$ be a generator of said quotient. Since $\Nb^F/\Levi^F$ is centralized by $F_0$, we have $x F_0(x)^{-1} \in \mathrm{N}_{\Levi^F}(Q)$. Let $\underline{x}:=(x,F^{r-1}_0(x),\dots,F_0(x)) \in \underline{\G}^{\tau F_0}$ such that $\mathrm{pr}(\underline{x})=x$. Consider the bijective morphism
$\underline{\phi}:\underline{\tilde{\G}} \langle \tau \rangle \to \underline{\tilde{\G}} \langle \tau \rangle$ given by conjugation with $\underline{x}$. Note that $\underline{\phi}$ stabilizes $\underline{\tilde{\G}}$ and commutes with the Frobenius endomorphism $\tau F_0$ of $\underline{\tilde{\G}} \rtimes \langle \tau \rangle$. Moreover, $\underline{\phi}$ also stabilizes the Levi subgroup $\tilde{\underline{\Levi}} \langle \tau \rangle$ of $\underline{\tilde{\G}} \rtimes \langle \tau \rangle$. We denote $e:=\br_Q(e_s^{\Levi^F})$. By \cite[Lemma 2.23]{Jordan} we obtain an isomorphism
$${}^{\underline{\phi}} (H_c^{\mathrm{dim}}(\Y_{\C_{\underline{\U'}}(\underline{Q})}^{\mathrm{N}_{\underline{\G}}(\underline{Q}),\tau F_0},\Lambda ) \underline{e} )^{\underline{\phi}} \cong  H_c^{\mathrm{dim}}(\Y_{\C_{\underline{\phi}(\underline{\U'})}(\underline{Q})}^{\mathrm{N}_{\underline{\G}}(\underline{Q}),\tau F_0},\Lambda ) \underline{e}$$
of $\Lambda ((\mathrm{N}_{\underline{\G}^{\tau F_0}}(\underline{Q}) \times \mathrm{N}_{\underline{\Levi}^{\tau F_0}}(\underline{Q})^{\mathrm{opp}}) \Delta (\mathrm{N}_{\underline{\tilde{\Levi}}^{\tau F_0} \langle \tau \rangle}(\underline{Q}))$-modules.
We have two Levi decompositions 
$$\underline{\tilde{\Para}} \langle \tau \rangle= \underline{\tilde{\Levi}} \langle \tau \rangle \ltimes \underline{\U} \text{ and } \underline{\phi}(\underline{\tilde{\Para}} \langle \tau \rangle)= \underline{\tilde{\Levi}} \langle \tau \rangle \ltimes \underline{\phi}(\underline{\U})$$
with the same Levi subgroup  $\underline{\tilde{\Levi}} \langle \tau \rangle $ of $\tilde{\G} \langle \tau \rangle$. Therefore, \cite[Theorem 5.2]{Jordan} yields
$$H_c^{\mathrm{dim}}(\Y_{\C_{\underline{\phi}(\underline{\U'})}(\underline{Q})}^{\mathrm{N}_{\underline{\G}}(\underline{Q}),\tau F_0},\Lambda ) \underline{e}
\cong  H_c^{\mathrm{dim}}(\Y_{\C_{\underline{\U'}}(\underline{Q})}^{\mathrm{N}_{\underline{\G}}(\underline{Q}),\tau F_0},\Lambda ) \underline{e}.$$
It follows from this that $H_c^{\mathrm{dim}}(\Y_{\C_{\underline{\U'}}(\underline{Q})}^{\mathrm{N}_{\underline{\G}}(\underline{Q}),\tau F_0},\Lambda ) \underline{e}$ is $( \underline{\phi},\underline{\phi}^{-1})$-invariant. Hence, the bimodule $H^{\mathrm{dim}}_c(\Y^{\mathrm{N}_\G(Q)}_{\C_{\U}(Q)},\Lambda )e$ is by transport of structure $x$-invariant as $\Lambda ((\mathrm{N}_{\G^F}(Q) \times \mathrm{N}_{\Levi^F}(Q)^{\operatorname{opp}}) \Delta (\mathrm{N}_{\tilde{\Levi}^F \langle F_0 \rangle  }(Q)))$-module. The claim now follows from \cite[Lemma 10.2.13]{Rouquier3}.
\end{proof}

\section{Construction for twisted groups}\label{sec 3}

We generalize the construction of the previous section. This is essentially necessary for working with automorphisms of twisted groups. We suppose now that $F_0$ is a Frobenius endomorphism with $F_0^r \rho=F$ for some integer $r$ and $\rho: \G \to \G$ a graph automorphism of order $l$ which commutes with $F_0$. We denote by $\mathcal{A}$ the subgroup of $\mathrm{Aut}(\Gtilde^F)$ generated by $F_0$.

The construction in this section will be done in two seperate steps. We first consider the connected reductive group
$$\G_\rho:=\{(g,\rho(g),\dots,\rho^{l-1}(g) ) \mid g \in \G \}.$$
We note that the projection onto the first coordinate defines an isomorphism
$\mathrm{pr}_1:\G_\rho \to \G$. Denote by $\tau: \G_\rho \to \G_\rho$ the automorphism given by $$\tau(g,\rho(g),\dots,\rho^{l-1}(g)):=(\rho(g),\rho^2(g),\dots,\rho^{l-1}(g),g).$$
With this notation the projection onto the first coordinate defines a $(\tau F_0^r,F)$-equivariant isomorphim $\G_\rho \to \G$.

Define now the group $$\underline{\G}_\rho:=\G_\rho \times \dots \times \G_\rho$$
as the $r$-fold product of the group $\G_\rho$.
 We define the permutation $\sigma_0$ on the set $\{1, \dots, rl \}$ as follows:
 \begin{equation*}
 \sigma_0(i)=
 \begin{cases}
 i-l & \text{if } i > l,\\
rl-l+i        & \text{if } 1 \leq i \leq l.
 \end{cases}
 \end{equation*}
Consider the automorphism $\tau_0: \underline{\G}_\rho \to \underline{\G}_\rho$ which for an element $(g_1,\dots,g_{lr}) \in \underline{\G}_{\rho}$ is given by $\tau_0((g_1,\dots,g_{lr})):=(g_{\sigma_0(1)},\dots,g_{\sigma_0(rl)})$.
 Observe that
  \begin{equation*}
 \sigma^r_0(i)=
  \begin{cases}
 	i-1 & \text{if } l \nmid (i-1), \\
 i+l-1       & \text{otherwise}.
 \end{cases}
\end{equation*}
It follows that $(\tau_0)^r=\tau$ (here $\tau$ is understood as the permutation $\tau \times \dots \times \tau$ on $\underline{\G}_\rho$). In particular, the morphism $\tau_0 F_0$ satisfies $(\tau_0 F_0)^r= \tau F_0^r$. Moreover, $\tau_0 F_0$ cyclically permutes the $r$ copies $\G_\rho$ of the group $\underline{\G}_\rho$. We deduce that the projection map $\mathrm{pr}_2: \underline{\G}_{\rho} \to \G_\rho$ onto the first factor $\G_\rho$ induces an isomorphism $\underline{\G}_{\rho}^{\tau F_0} \cong \G_{\rho}^{\tau_0 F_0^r}$.

\begin{notation}
Assume that $\mathbf{H}$ is a closed subgroup of $\G$. Then we define $\mathbf{H}_\rho:=\mathrm{pr}_1^{-1}(\mathbf{H})$. We then define $$\underline{\mathbf{H}}_\rho:=\mathbf{H}_\rho \times (\tau F_0)^{r-1}(\mathbf{H}_\rho) \times \dots \times (\tau F_0)(\mathbf{H}_\rho)$$
considered as a subgroup of $\underline{\G}_\rho$.
Furthermore, we denote $\underline{\mathbf{H}}'_\rho:=(\mathbf{H}_\rho)^r$.
\end{notation}

\begin{remark}
	If the subgroup $\mathbf{H}$ is $\rho F_0$-stable, then $\mathbf{H}_\rho$ is $\tau F_0$-stable. Thus, $\underline{\mathbf{H}}_\rho=\underline{\mathbf{H}}_\rho'$ in this case.
	We remark that the automorphism $F_0$ of the finite group $\G^F$ corresponds to the automorphism $\tau_0^{-1}$ of $\underline{\G}_{\rho}^{\tau_0 F_0}$.
\end{remark}

	Let $\Para$ be a parabolic subgroup of $\G$ with Levi decomposition $\Para=\Levi \ltimes \U$. Then we observe that $\underline{\Para}_\rho$ is a parabolic subgroup of $\underline{\G}_\rho$ with Levi decomposition $$\underline{\Para}_\rho=\underline{\Levi}_\rho \ltimes \underline{\U}_\rho.$$
	 Assume now that $\Levi$ is $F$-stable. It follows that $\Levi_\rho$ is $\tau F_0^r$-stable and consequently the Levi subgroup $\underline{\Levi}_\rho$ is $\tau_0 F_0$-stable. We can therefore consider the Deligne--Lusztig variety $\Y_{\underline{\U}_\rho}^{\underline{\G}_\rho,F_0 \tau_0}$ which is a $\underline{\G}_\rho^{F_0 \tau_0} \times (\underline{\Levi}_\rho^{F_0 \tau_0})^{\mathrm{opp}}$-variety. Under the isomorphism $\G^F\cong \underline{\G}_\rho^{F_0 \tau_0}$ induced by the projection map $$\mathrm{pr}:=\mathrm{pr}_1 \circ \mathrm{pr}_2: \underline{\G}_\rho \to \G$$ we will in the following regard it as a $\G^F \times (\Levi^F)^{\mathrm{opp}}$-variety.

\begin{proposition}\label{ShintaniDLV2}
With the assumptions as above, the projection $\mathrm{pr}:\underline{\G}_\rho \to \G$ onto the first coordinate defines an isomorphism
	$$ \Y_{\underline{\U}_\rho}^{\underline{\G}_\rho,\tau_0 F_0} \cong  \Y_\U^{\G,F} $$
	of varieties which is $\G^F \times (\Levi^F)^{\mathrm{opp}}$-equivariant.
\end{proposition}

\begin{proof}
The projection onto the first coordinate defines a $(\tau F_0^r,F)$-equivariant isomorphim $\G_\rho \to \G$. Thus, we obtain a $\G^F \times (\Levi^F)^{\mathrm{opp}}$-equivariant isomorphism $ \Y_{\U_\rho}^{\G_\rho,\tau F_0^r} \to  \Y_\U^{\G,F}.$

The morphism $\tau_0 F_0$ satisfies $(\tau_0 F_0)^r= \tau F_0^r$ and cyclically permutes the $r$ copies of $\G_\rho$ of the group $\underline{\G}_\rho$. Therefore, by \cite[Theorem 5.2]{Jordan} projection onto the first coordinate yields a $\G^F \times (\Levi^F)^{\mathrm{opp}}$-equivariant isomorphism
$ \Y_{\underline{\U}_\rho}^{\underline{\G}_\rho,\tau_0 F_0}  \to \Y_{\U_\rho}^{\G_\rho,\tau F_0^r}.$
The composition of these two isomorphisms yields the required isomorphism.
\end{proof}

We will now explain how we can explicitly construct the dual group of $(\underline{\G}_\rho,\tau_0 F_0)$.
For this consider the connected reductive group
$\G^\ast_\rho:=\{(g,\rho^\ast(g),\dots,(\rho^\ast)^{l-1}(g) )\}$. Define now the group $\underline{\G}^\ast_\rho:=\G^\ast_\rho \times \dots \times \G^\ast_\rho$ as the $r$-fold product of the group $\G^\ast_\rho$. On this group we consider the automorphism $\tau^\ast_0: \underline{\G}^\ast_\rho \to \underline{\G}^\ast_\rho$ which for an element $(g_1,\dots,g_{lr}) \in  \underline{\G}^\ast_\rho$ is given by $$\tau^\ast_0((g_1,\dots,g_{lr})):=(g_{(\sigma_0)^{-1}(1)},\dots,g_{(\sigma_0)^{-1}(lr)}).$$
We then define $\tau^\ast:=(\tau^\ast_0)^r$. Again we denote by $\mathrm{pr}_1: \G_\rho^\ast \to \G^\ast$ and $\mathrm{pr}_2: \underline{\G}^\ast_{\rho} \to \G_\rho^\ast$ the projections onto the first coordinate and $\mathrm{pr}=\mathrm{pr}_1 \circ \mathrm{pr}_2$.

\begin{lemma}\label{twistedLusztig}
	The group $(\underline{\G}_\rho^\ast,\tau_0^\ast F_0^\ast)$ is dual to $(\underline{\G}_\rho, \tau_0 F_0)$ and for every semisimple element $\underline{s} \in (\underline{\G}_\rho^\ast)^{\tau_0^\ast F_0^\ast}$ the Lusztig series $\mathcal{E}(\underline{\G}^{\tau_0 F_0},\underline{s})$ corresponds to $\mathcal{E}(\G^F,\mathrm{pr}(\underline{s}))$ under the isomorphism $\mathrm{pr}: \underline{\G}_\rho^{\tau_0 F_0} \to \G^F$ of finite groups.
\end{lemma}

\begin{proof}
Recall that we have a $(\tau F_0^r,F)$-equivariant isomorphism $\G_\rho \to \G$. From this we deduce that $(\G^\ast_\rho,\tau^\ast (F_0^\ast)^r)$ is dual to  $(\G_\rho,\tau F_0^r)$. The Lusztig series $\mathcal{E}(\G_\rho^{\tau F_0^r},\mathrm{pr}_2(\underline{s}))$ then corresponds to $\mathcal{E}(\G^F,\mathrm{pr}(\underline{s}))$ via the isomorphism $\mathrm{pr}_1: \G_\rho^{\tau F_0^r} \to \G^F$. By applying \cite[Corollary 8.8]{Taylor2} we then observe that $(\underline{\G}_\rho^\ast,\tau_0^\ast F_0^\ast)$ is dual to $(\underline{\G}_\rho, \tau_0 F_0)$ and moreover that the Lusztig series $\mathcal{E}(\underline{\G}^{\tau_0 F_0},\underline{s})$ corresponds to $\mathcal{E}(\G_\rho^{\tau F_0^r},\mathrm{pr}_2(\underline{s}))$ via $\mathrm{pr}_2: \underline{\G}_\rho^{\tau_0 F_0} \to  \G_\rho^{\tau F_0^r}$. The claim follows from this.
\end{proof}

\begin{lemma}\label{mainlemma2}
	Assume that $\Levi$ and $e_s^{\Levi^F}$ are $F_0$-stable. Then the module $H^{\mathrm{dim}}_c(\Y_\U^\G,\Lambda) e_s^{\Levi^F}$ extends to a $\G^F \times (\Levi^F)^{\mathrm{opp}} \Delta \langle F_0 \rangle$-module.
	If in addition $\Nb^F/\Levi^F$ is centralized by $F_0$ then $H^{\mathrm{dim}}_c(\Y_\U^\G,\Lambda) e_s^{\Levi^F}$ extends to $\G^F \times (\Nb^F)^{\mathrm{opp}} \Delta(\tilde{\Levi}^F \langle F_0 \rangle)$.
\end{lemma}

\begin{proof}	
Since $\Levi$ is $F_0$-stable we have $\underline{\Levi}_\rho=\underline{\Levi}'_\rho$ and therefore $\underline{\Para}_\rho'$ is a parabolic subgroup with Levi decomposition
	$$\underline{\Para}'_\rho=\underline{\Levi}_\rho \ltimes \underline{\U}'_\rho.$$
We observe that $\mathrm{C}^\circ_{\underline{\G}_\rho^\ast}(\underline{s}) \subseteq \underline{\Levi}^\ast_\rho$. Therefore, \cite[Theorem 7.2]{Dat} shows that we have an isomorphism
	$$H^{\mathrm{dim}}_c(\Y_{\underline{\U}_\rho}^{\underline{\G}_\rho, \tau_0 F_0}) e_{\underline{s}}^{\underline{\Levi}^{\tau F_0}} \cong H^{\mathrm{dim}}_c(\Y_{\underline{\U'}_\rho}^{\underline{\G}_\rho, \tau_0 F_0}) e_{\underline{s}}^{\underline{\Levi}^{\tau F_0}}$$
	of $\Lambda [(\underline{\G}_\rho^{\tau F_0} \times (\underline{\Levi}_\rho^{\tau F_0})^{\mathrm{opp}}) \Delta( \underline{\tilde{\Levi}}_\rho^{\tau F_0} )]$-modules.
	According to Proposition \ref{ShintaniDLV2} and Lemma \ref{twistedLusztig} the projection $\mathrm{pr}:\underline{\G}_\rho \to \G$ onto the first coordinate induces an isomorphism
	$$H_c^{\mathrm{dim}}(\Y_{\underline{\U}_\rho}^{\underline{\G}_\rho,\tau_0 F_0},\Lambda) e_{\underline{s}}^{\underline{\Levi}_\rho^{\tau_0 F_0}} \cong  H_c^{\mathrm{dim}}(\Y_\U^{\G,F},\Lambda) e_s^{\Levi^F} $$
	of $\G^F \times (\Levi^F)^{\mathrm{opp}} \Delta (\tilde{\Levi}^F)$-modules.
	
	Note that since $\U$ is $\rho$-stable the automorphism $\tau_0$ stabilizes the unipotent radical $\underline{\U}'_\rho$. Therefore, we can consider $H^{\mathrm{dim}}_c(\Y_{\underline{\U'}_\rho}^{\underline{\G}_\rho, \tau_0 F_0}) e_{\underline{s}}^{\underline{\Levi}^{\tau_0 F_0}}$ as $\Lambda ((\underline{\G}_\rho^{\tau_0 F_0} \times (\underline{\Levi}_\rho^{\tau_0 F_0})^{\mathrm{opp}}) \Delta( \underline{\tilde{\Levi}}_\rho^{\tau_0 F_0} \langle \tau_0 \rangle ))$-module. By transport of structure, we can endow $H_c^{\mathrm{dim}}(\Y_\U^{\G,F},\Lambda) e_s^{\Levi^F}$ with a $\G^F \times (\Levi^F)^{\mathrm{opp}} \Delta \tilde{\Levi}^F \langle F_0 \rangle$-structure.

	Assume now that $\mathbf{N}/\Levi$ is centralized by $F_0$. Let $n \in \mathbf{N}^{F}$ be a generator of the quotient group $\mathbf{N}^F/ \Levi^F$. Since $F_0(n)n^{-1} \in \Levi^F$ we conclude that conjugation by $\underline{n}$ defines an automorphism of $\underline{\G}_\rho \langle \tau_0 \rangle$ which stabilizes the Levi subgroup $\underline{\Levi}_\rho \langle \tau \rangle$. Thus, conjugation by $\underline{n}$ yields an isomorphism
	$$(H^{\mathrm{dim}}_c(\Y_{\underline{\U'}}^{\underline{\G}, \tau F_0}) e_{\underline{s}}^{\underline{\Levi}^{\tau F_0}})^{\underline{n}} \cong H^{\mathrm{dim}}_c(\Y_{{}^{\underline{n}} \underline{\U'}}^{\underline{\G}, \tau F_0}) e_{\underline{s}}^{\underline{\Levi}^{\tau F_0}}$$
	of $\Lambda [(\underline{\G}^{\tau F_0} \times (\underline{\Levi}^{\tau F_0})^{\mathrm{opp}}) \Delta( \underline{\tilde{\Levi}}^F \langle \tau, \gamma \rangle )]$-modules. We conclude that $H^{\mathrm{dim}}_c(\Y_{\underline{\U'}}^{\underline{\G}_\rho, \tau_0 F_0}) e_{\underline{s}}^{\underline{\Levi}^{\tau_0 F_0}}$ is $\underline{n}$-stable. By transport of structure and \cite[Lemma 10.2.13]{Rouquier3} we deduce that $H^{\mathrm{dim}}_c(\Y_\U^{\G,F},\Lambda) e_s^{L}$ extends to a $\G^F \times (\mathbf{N}^F)^{\mathrm{opp}} \Delta( \tilde{\Levi}^F \langle F_0 \rangle)$-module.
\end{proof}

\begin{lemma}\label{mainlemma:local2}
	In the situation of Lemma \ref{mainlemma2}, let $Q$ be an $\ell$-subgroup of $\Levi^F$. If $\Nb^F/\Levi^F$ is centralized by $\mathcal{A}$ then $H^{\mathrm{dim}}_c(\Y_{\mathrm{C}_\U(Q)}^{\mathrm{N}_{\G}(Q)},\Lambda) \mathrm{br}_Q(e_s^{\Levi^F})$ extends to a $\mathrm{N}_{\G^F}(Q) \times \mathrm{N}_{\Levi^F}(Q)^{\mathrm{opp}} \Delta \mathrm{N}_{\mathcal{N}}(Q)$-module, where $\mathcal{N}:=\mathrm{N}_{\Gtilde^F \langle F_0 \rangle }(\Levi,e_s^{\Levi^F})$.
\end{lemma}

\begin{proof}
	The proof is the same as in Lemma \ref{mainlemma:local} with the necessary modifications made in Lemma \ref{mainlemma2}.
\end{proof}

\section{An equivariant Bonnafé--Dat--Rouquier equivalence}\label{sec 4}

In this section we give a partial answer to the question whether the Morita equivalence constructed by Bonnafé--Dat--Rouquier is automorphism-equivariant. We often use the following well known fact.

\begin{lemma}\label{norm}
	Let $\G$ be a connected reductive group and $\phi: \G\to \G$ a Frobenius endomorphism.
	The norm map $N_{\phi^r/\phi}: \G \to \G, \, x \mapsto \prod_{i=0}^{r-1} \phi^i(x),$ is surjective and maps $\G^{\phi^r}$ to $\G^\phi$.
\end{lemma}

\begin{proof}
	By Lang's theorem we can write $y \in \G$ as $y=a^{-1} \phi^r(a)$. Then for $x:=a^{-1} \phi(a)$ we have $N_{\phi^r/\phi}(x)=y$. The second statement is a straightforward computation.
\end{proof}

\begin{theorem}\label{equiv}
 Let $\G$ be a simple, simply connected algebraic group of type $B_n$, $C_n$ or $E_7$ such that either $n > 2$ or $q$ is odd. Then there exists a Frobenius morphism $F_0: \G \to \G$ stabilizing $\Levi$ such that the image of $\Gtilde^F \rtimes \langle F_0 \rangle$ in the outer automorphism group of $\G^F$ is $\mathrm{Out}(\G^F)_{e_s^{\G^F}}$. There exists a Morita equivalence between $\Lambda \Nb^F e_s^{\Levi^F}$ and $\Lambda \G^F e_s^{\G^F}$ which lifts to a Morita equivalence between $\Lambda \mathrm{N}_{\tilde{\G}^F \langle F_0 \rangle }(\Levi,e_s^{\Levi^F}) e_s^{\Levi^F}$ and $\Lambda \tilde{\G}^F \langle F_0 \rangle e_s^{\G^F}$.
\end{theorem}

\begin{proof}
	As in the proof of \cite[Corollary 4.2]{Jordan} we see that there exists a field endomorphism $\phi: \G\to \G$ such that $\phi^r=F$ for some integer $r$ and such that $\phi$ together with $\Gtilde^F$ generates $\mathrm{Out}(\G^F)_{e_s^{\G^F}}$.
	The Levi subgroup $\phi(\Levi)$ is $\G^F$-conjugate to $\Levi$. Thus, there exists some $y \in \G^F$ such that $\Levi$ is $y \phi$-stable. Since $e_s^{\G^F}$ is $\phi$-stable there exists some $x \in \mathrm{N}_{\G^F}(\Levi)$ such that $e_s^{\Levi^F}$ is $xy \varphi$-stable. Denoting $z:=xy$ we have $( z \phi)^r=N_{F/\phi}(z) F$ and therefore $N_{F/\phi}(z) \in \mathrm{N}_{\G^{\phi}}(\Levi,e_s^{\Levi^F})$.
	
	Suppose first that $z_0:=N_{F/ \phi }(z) \in \Levi^F$. The Levi subgroup $\Levi$ is $z \phi$-stable and we have $z_0 \in \Levi^{z \phi}$. Therefore, by Lemma \ref{norm} there exists $l \in \Levi^{z_0 F}$ such that $z_0^{-1}=N_{z_0F/z\phi}(l)$. We define $F_0:=l z_0 \phi$. Then we have
	$$F_0^r=(l z_0 \phi)^r=(l (z_0 \phi) )^r=N_{z_0 F/ z\phi }(l) z_0 F=F.$$
	Since $F_0=l z_0 \phi$ commutes with $F$ we have $F(l z_0) (l z_0)^{-1}=F(l)l^{-1} \in \mathrm{Z}(\G)$. Thus, $\phi$ and $F_0$ do not necessarily have the same image in $\mathrm{Out}(\G^F)$ but differ by the diagonal automorphism induced by $\mathrm{ad}(l)$. Since $e_s^{\Levi^F}$ is $z_0 \phi$-stable and stable under diagonal automorphisms we observe that $e_s^{\Levi^F}$ is $F_0$-stable. From this we conclude that $F_0$ together with $\Gtilde^F$ generates the stabilizer of $e_s^{\G^F}$ in $\mathrm{Out}(\G^F)$.
	%
	Note that $(\Nb/\Levi)^{F_0}=\Nb/\Levi$ since by the remarks following \cite[Lemma 3]{Lucas} the group $\Nb/\Levi$ is isomorphic to a subgroup of $\mathrm{Z}(\G)$ which has at most order $2$.
	We can now apply Lemma \ref{mainlemma} and conclude that there exists a $\Lambda[\G^F \times (\mathbf{N}^F)^{\operatorname{opp}} \Delta (\langle F_0 \rangle) ]$-module $M$ extending $H_c^{\mathrm{dim}}(\Y_\U,\Lambda) e_s^{\Levi^F}$ which induces a Morita equivalence between $\Lambda \mathbf{N}^F e_s^{\Levi^F}$ and $\Lambda \G^F e_s^{\G^F}$. In particular, the bimodule $\tilde{M}:=\mathrm{Ind}_{\G^F \times (\mathbf{N}^F)^{\operatorname{opp}} \Delta \tilde{\mathbf{N}}^F \langle F_0 \rangle }^{\tilde{\G}^F \langle F_0 \rangle \times (\tilde{\mathbf{N}}^F \langle F_0 \rangle)^{\operatorname{opp}}} (M)$ induces a Morita equivalence between $\Lambda \tilde{\Nb}^F \langle F_0 \rangle e_s^{\Levi^F}$ and $\Lambda \tilde{\G}^F \langle F_0 \rangle  e_s^{\G^F}$.

	
	Assume now that that $N_{F/ \phi }(z) \notin \Levi^F$. Since $\Nb/\Levi$ is isomorphic to a subgroup of $\mathrm{Z}(\G)$ we conclude that $N_{F/\phi}(z)$ generates the quotient group $\Nb/\Levi$. Thus, in this case the bijective morphism $F_0:=\phi$ has the property that the element $z F_0 \in G \langle F_0 \rangle$ generates the quotient group $\mathrm{N}_{G \langle F_0 \rangle}(\Levi,e_s^{\Levi^F}) / \Levi^F$ which is in particular cyclic. We conclude that in this case the $\Lambda(\G^F \times (\mathbf{L}^F)^{\operatorname{opp}} \Delta (\tilde{\Levi}^F) )$-module $H_c^{\mathrm{dim}}(\Y_\U,\Lambda) e_s^{\Levi^F}$ extends to a $\G^F \times (\mathbf{L}^F)^{\operatorname{opp}} \Delta (\N_{\Gtilde^F \langle F_0 \rangle}(\Levi,e_s^L)) $-module. The claim is now a consequence of \cite[Theorem 3.4]{Marcus} (see also \cite[Theorem 1.7]{Jordan}).
\end{proof}

\section{Reduction to isolated blocks: type $A$}\label{sec 5}

\subsection{Strictly quasi-isolated blocks}

The following notion will become important when dealing with actual blocks of groups of Lie type.

\begin{definition}
	Let $\G$ be a connected reductive group. We say that a semisimple element $s \in (\G^\ast)^{F^\ast}$ is \textit{strictly quasi-isolated in $(\G^\ast,F^\ast)$} if $\mathrm{C}_{\G^\ast}(s)^{F^\ast} \mathrm{C}_{\G^\ast}^\circ(s)$ is not contained in a proper Levi subgroup of $\G^\ast$.
\end{definition}

We denote by $A(s)$ the component group $\mathrm{C}_{\G^\ast}(s)/ \mathrm{C}^\circ_{\G^\ast}(s)$. The proof of the following lemma is similar to the proof of \cite[Corollary 2.9]{Bonnafe}.


\begin{lemma}\label{strict}
	Assume that $\G$ is simple, simply connected of type $A$. If $s \in (\G^\ast)^{F^\ast}$ is a semisimple element which is strictly quasi-isolated, then we have $A(s)^{F^\ast}= A(s)$.
\end{lemma}

\begin{proof}
	Let $\iota^\ast: \Gtilde^\ast \to \G^\ast$ be the map dual to the map $\iota: \G \hookrightarrow \Gtilde$ and let $ \tilde s\in (\Gtilde^\ast)^{F^\ast}$ such that $\iota^\ast(\tilde{s})=s$.
	Consider the injective morphism $\omega_s: A(s) \to \mathrm{Z}(\Gtilde^\ast)$ as in \cite[Corollary 2.8]{Bonnafe}. Let $e$ denote the exponent of the subgroup $A(s)^{F^\ast}$ of $A(s)$. Let $g \in \mathrm{C}^\circ_{\G^\ast}(s) \mathrm{C}_{(\G^\ast)^{F^\ast}}(s)$. Then $1=\omega_s(g)^e=\omega_{s^e}(g)$ and therefore $g \in \mathrm{Ker}(\omega_s^e)=\iota^\ast(\mathrm{C}_{\Gtilde^\ast}(\tilde{s}^e))$. Consequently, $\iota^\ast(\mathrm{C}_{\Gtilde^\ast}(\tilde{s}^e))=\mathrm{C}^\circ_{\G^\ast}(s^e)$ is a Levi subgroup of $\G^\ast$ containing $\mathrm{C}^\circ_{\G^\ast}(s) \mathrm{C}_{(\G^\ast)^{F^\ast}}(s)$. This is a proper Levi subgroup of $\G^\ast$ unless $s^e=1$. By the classification of quasi-isolated elements of $\G^\ast=\mathrm{PGL}_{n+1}(\overline{\mathbb{F}_p})$ in \cite[Proposition 5.2]{Bonnafe}, we have $\mathrm{o}(s)=|A(s)|$ which implies that $e=|A(s)|$. Since $e$ is the exponent of $A(s)^{F^\ast}$ we must necessarily have $A(s)=A(s)^{F^\ast}$.
\end{proof}

\subsection{Duality for groups of type $A$}\label{5.2}

From now on until Section \ref{sec 11} we assume that $\G=\mathrm{SL}_{n+1}(\overline{\mathbb{F}_p})$ is of type $A_n$. For an integer $q=p^f$ we let $F_q: \G \to \G$ be the Frobenius endomorphism which raises every matrix entry to its $q$th power. For $\varepsilon \in \{ \pm 1 \}$ we define $F'=F_q (\gamma')^{\frac{1-\varepsilon}{2}}$. Here $\gamma': \G \to \G, \, g \mapsto g^{-\mathrm{tr}}$ is the graph automorphism given by transpose inversion. We denote by $\T_0$ the torus of diagonal matrices and by $\B_0$ the Borel subgroup of upper triangular matrices. We denote by $W:= \N_{\G}(\T_0)/\T_0$ the Weyl group of $\G$ and identify $W$ with the symmetric group $S_{n+1}$.

Note that with our choices of $F'$ the torus $\T_0$ is not maximally split in the twisted case. This is because the graph automorphism $\gamma'$ doesn't stabilize the Borel subgroup $\B_0$. Therefore, we define $\gamma:=\ad(n_0) \gamma'$, where $n_0 \in \N_\G(\T_0)$ is the matrix with entry $(-1)^{l+1}$ at position $(l,n+1-l)$ with $1 \leq l \leq n$ and $0$ elsewhere. Observe that its image $w_0 \in W$ is the longest element of $W$. For $\varepsilon \in \{ \pm 1\}$ and fixed $q$ we consider the Frobenius endomorphism $F: \G \to \G$ given by $F:=F_q \gamma^{\frac{1-\varepsilon}{2}}$.

To efficiently compute with dual groups we will use some simplifications specific to the situation in type $A$. Since $\G^\ast$ is adjoint of type $A$ there exists a surjective morphism $\pi: \G \to \G^\ast$. Suppose that $f: \G \to \G$ is a bijective morphism such that $f(\T_0)= \T_0$. Then we denote by $f^\ast: \G^\ast \to \G^\ast$ the unique morphism satisfying $\pi \circ f= f^\ast \circ \pi$.

If $f$ is a Frobenius endomorphism then the pair $(\G^\ast,f^\ast)$ is dual to $(\G,f)$. For a Levi subgroup $\Levi$ of $\G$ we define $\Levi^\ast:=\pi(\Levi)$. We observe that $W^\ast:= \N_{\G^\ast}(\T_0^\ast)/\T_0^\ast$ is the Weyl group of $\G^\ast$ and we obtain a natural isomorphism $\pi: W \to W^\ast$.

 If $\Levi$ is $\mathrm{ad}(g)f$-stable, for some $g\in \G^F$, then one can check the pair $(\Levi,\mathrm{ad}(g)f)$ is again in duality to $(\Levi^\ast,\mathrm{ad}(\pi(g))f^\ast)$.


\begin{remark}\label{type}
We consider $\G$ with the Frobenius $F'$ as defined above and we let $s \in (\G^\ast)^{F'^\ast}$ be a strictly quasi-isolated element. The aim of this remark is to compute the $F'$-type of the Levi subgroup $\Levi^\ast:= \mathrm{C}^\circ_{\G^\ast}(s)$. For an integer $m$ dividing $n+1$ we fix a primitive $m$th root of unity in $\overline{\mathbb{F}_p}$ and define
$$t_m:=t:=(1,\zeta,\dots,\zeta^{m-1}) \otimes I_e,$$
where $me= n+1$. According to \cite[Proposition 5.2]{Bonnafe} there exists $g \in \G^\ast$ and $m$ dividing $n+1$ such that $t_m={}^g s \in \T_0^\ast$. Since $s$ is $(F')^\ast$-stable we have $F'^\ast(t)={}^{F'^\ast(g)} s={}^{F'^\ast(g) g^{-1}} t$. Hence, $t$ is $w^\ast F'^\ast$-stable, where $w^\ast$ is the image of $g^{-1} F'^\ast(g)$ in $W_{\G^\ast}(\T_0^\ast)$. Since $s$ is assumed to be strictly quasi-isolated Lemma \ref{strict} yields $A(s)^{F'^\ast}=A(s)$. This implies that $A(t)^{w^\ast F'^\ast} =A(t)$. By \cite[Proposition 5.2]{Bonnafe} the component group $A(t)$ has order $m$. On the other hand $A(s)$ is isomorphic to a subgroup of $\mathrm{Z}(\G^F)$ which has order $(n,q -\varepsilon)$. Therefore $m$ divides $(n,q - \varepsilon )$. From this we deduce that the element $t$ is $(F')^\ast$-stable. From this it follows that $w^\ast \in W(t):= \{ w \in W^\ast \mid {}^w t =t \}$. Let $v_m \in W^\ast$ be the permutation defined by
$$v_m(i) \equiv i + e \, \mathrm{mod} \, (n+1)$$
such that $A(t)$ is generated by $v_m$ and has order $m$.
Denote by $W^\circ(t)$ the Weyl group of $ \Levi_m^\ast:=\mathrm{C}^\circ_{\G^\ast}(t)$ relative to the maximal torus $\T_0^\ast$ such that we have $W(t) = W^\circ(t) \rtimes A(t)$, see \cite[Proposition 1.3(c)]{Bonnafe}. To determine the type of $\Levi^\ast$ we can change $w$ by an element of $W^\circ(t)$ and we therefore conclude that $\Levi$ is of $F'$-type $wW^\circ(t)$, where $w \in \langle v_m \rangle$.

If we consider $\G$ with the twisted Frobenius endomorphism $F$ instead of $F'$ it follows that $\Levi$ is of $F$-type $w W^\circ(t)$ with $w \in w_0 \langle v_m \rangle$.


\end{remark}
%

%
%
%

\begin{corollary}\label{1split}
Let $s$ be a strictly quasi-isolated element. Then the Levi subgroup $\Levi^\ast= \mathrm{C}^\circ_{\G^\ast}(s)$ is $d$-split for some integer $d$ dividing $2|A(s)|$.
\end{corollary}

\begin{proof} We keep the notations of Remark \ref{type}. Assume first that $\varepsilon=1$.
Let $w \in \langle v_m \rangle$ as in Remark \ref{type} such that $\Levi^\ast$ is of type $W^\circ(t) w$. Then $w$ has order $d$ with $d$ dividing $m=|A(s)|$ and we can decompose $w=w_1 \cdots w_r$ into disjoint cycles each $w_i$ of order $d$ with $d \mid m$ and $n+1=r d$. Note that $\Levi^\ast$ is $d$-split if and only if $\tilde{\Levi}^\ast$ is $d$-split, see \cite[Proposition 13.2]{MarcBook}.
Furthermore, $\tilde{\Levi}^{F'} \cong \tilde{\Levi}_m^{w F'} \cong \mathrm{GL}_r({\mathbb{F}_{q^d}}) \times \dots \mathrm{GL}_r({\mathbb{F}_{q^d}})$. Hence $\tilde{\Levi}$ and therefore also $\Levi$ is $d$-split, see \cite[Example 13.4]{MarcBook}. By Ennola-duality we obtain that if $ \varepsilon=-1$ then $\Levi$ is $d'$-split where
\begin{equation*}
d'=
\begin{cases}
2d & \text{if } d \text{ is odd},\\
d/2         & \text{if } d \equiv 2 \, \mathrm{mod} \, 4,\\
d  &\text{if } d \equiv 0 \, \mathrm{mod} \, 4.
\end{cases}
\end{equation*}
This also shows the result in the twisted case.
\end{proof}
\begin{remark}
Corollary \ref{1split} offers potentially an approach to an equivariant Bonnaf\'e--Dat--Rouquier type bijection by using the results of Brough--Späth, in particular \cite[Theorem 4.3]{Brough}.
\end{remark}

\section{Equivariant Bonnaf\'e--Dat--Rouquier equivalence}\label{sec 6}

 Until Section \ref{sec 11}, $s$ denotes an arbitrary, but fixed semisimple strictly quasi-isolated $\ell'$-element in $(\G^\ast)^{F^\ast}$.
As before, we denote by $\Levi$ the Levi subgroup of $\G$ dual to $\Levi^\ast=\mathrm{C}^\circ_{\G^\ast}(s)$ defined as its preimage under the projection map $\pi: \G \to \G^\ast$.

Assume that we are given a different Levi subgroup $\Levi'$ of $\G$ containing $\Levi$ such that $\mathrm{C}_{(\G^\ast)^{F^\ast}}(s) (\Levi')^\ast=(\Levi')^\ast \mathrm{C}_{(\G^\ast)^{F^\ast}}(s)$. Denote by $N'$ the common stabilizer of $\Levi'$ and $e_s^{L'}$ in $\G^F$. By the results of Bonnafé--Dat--Rouquier, see Theorem \ref{Boro}, there exists a Morita equivalence between $\Lambda N' e_s^{L'}$ and $\Lambda G e_s^G$. We first make the following observation:

\begin{lemma}
	With the notation as above,
	$N'/L'$ is naturally isomorphic to a subquotient of $N/L$.
\end{lemma}

\begin{proof}
 We have $\mathbf{N}^\ast/\Levi^\ast=\mathrm{C}_{\G^\ast}(s)/ \mathrm{C}^\circ_{\G^\ast}(s)$ and on the other hand,  $\mathbf{N}'^\ast/(\Levi'^\ast)=(\Levi')^\ast \mathrm{C}_{\G^\ast}(s)/(\Levi')^\ast \cong \mathrm{C}_{\G^\ast}(s)/ \mathrm{C}_{\Levi'^\ast}(s)$. Since $\mathrm{C}^\circ_{\G^\ast}(s) \subset \mathrm{C}_{\Levi'^\ast}(s)$ we obtain a surjection $\mathbf{N}^\ast/(\Levi^\ast) \to \mathbf{N}'^\ast/(\Levi'^\ast)$. The statement follows by taking $F^\ast$-fixed points and using the duality between $(\G,F)$ and $(\G^\ast,F^\ast)$.
\end{proof}

%

Recall from Remark \ref{type} that $s$ is $\G^\ast$-conjugate to the element $t_m$, where $m$ is some positive integer dividing $n+1$. Hence, the conjugacy classes of (strictly) quasi-isolated elements are parametrized in a uniform way. In this section we use the uniformity of this description to our advantage.

\begin{lemma}\label{prime order}
	There exists a proper $F$-stable Levi subgroup $\Levi'$ of $\G$ containing $\Levi$ such that $N'/L'$ is cyclic of prime order.
	%
	%
	%
\end{lemma}

\begin{proof}
	Recall that	for a given $m$ with $m \mid (n+1)$ we denote $t_m:=(1,\zeta,\dots,\zeta^{m-1}) \otimes I_e$,
	where $\zeta \in \mathbb{F}_q^\times$ is a fixed primitive $m$th-root of unity with $me=n+1$. Consider the Levi subgroup $\Levi^\ast_m:=\mathrm{C}^\circ_{\G^\ast}(t_m)$ of $\G^\ast$ and let $\Levi_m:=\pi^{-1}(\Levi_m^\ast)$ be the standard Levi subgroup of $\G$ in duality with $\Levi_m^\ast$.

	Assume that $s$ is $\G^\ast$-conjugate to $t_m$, i.e. there exists $g \in \G^\ast$ such that $t={}^g s$.
	We let $m'$ be a prime divisor of $m$. We note that $\mathrm{C}_{\G^\ast}(t_m)$ is contained in $\mathrm{C}_{\G^\ast}(t_{m'})$.
 	
 	Denote by $W^\circ(t_m)$ the Weyl group of $\mathrm{C}^\circ_{\G^\ast}(t_m)$ relative to the maximal torus $\T_0^\ast$. By \cite[Proposition 5.2]{Bonnafe} we obtain that $W^\circ(t_{m}) \cong (S_e)^m$ where $v_m$ transitively permutes the $m$ copies of $S_e$. Since $W^\circ(t_{m'})\cong (S_{e'})^{m'}$ and $m' \mid m$ it follows that $v_m$ normalizes $W^\circ(t_{m'})$. We define $\Levi'^\ast:={}^g \Levi_{m'}^\ast$. Recall from Remark \ref{type} that $\Levi^\ast$ is of $F'^\ast$-type $w W^\circ(t_{m})$ with $w \in \langle v_m \rangle$.
	 Since $w$ normalizes $W^\circ(t_{m'})$ and $F'^\ast$ acts trivially on $W^\ast$ we have ${}^{w F'} W^\circ(t_{m'})=W^\circ(t_{m'})$ and so $\Levi'^\ast$ is $F'^\ast$-stable of type $w W^\circ(t_{m'})$. It follows that $\Levi'^\ast$ is $F'^\ast$-stable. Furthermore, $A(t_{m'})^{wF'^\ast} \cong A(t_{m'}) \cong C_{m'}$ and so we observe that $N'/L'$ is cyclic of $m'$-order.
	%
	%
\end{proof}

%

The upcoming sections will provide the necessary knowledge on strictly quasi-isolated elements in type $A$.

\subsection{Some computations in the Weyl group}

Recall that $\gamma: \G\to \G$ denotes the graph automorphism stabilizing $(\T_0,\B_0)$. In this section we collect some properties of the Levi subgroup $\Levi_m$ which was defined in the proof of Lemma \ref{prime order}.

\begin{lemma}\label{connectedgraph}
The Levi subgroup $\Levi_m$ is a $\gamma$-stable Levi subgroup and contained in a $\gamma$-stable parabolic subgroup $\Para_m$ of $\G$. In particular, $\Levi_m^\gamma$ is a Levi subgroup in the connected reductive group $\G^\gamma$ with parabolic subgroup $\Para_m^\gamma$.
\end{lemma}

\begin{proof}
	Observe that $\gamma$ is a quasi-central morphism in the sense of \cite[Definition-Theorem 1.15]{DM2}. According to \cite[Remark 1.30]{DM2} the group $\G^\gamma$ is connected. The Levi subgroup $\Levi_m$ is a standard Levi subgroup of $\G$ relative to the $\gamma$-stable pair $(\T_0,\B_0)$ whose associated set of simple roots is $\gamma$-stable. Thus, $\Levi_m$ is a $\gamma$-stable Levi subgroup contained in a $\gamma$-stable parabolic subgroup $\Para_m$ of $\G$. From \cite[Proposition 1.11]{DM2} it therefore follows that $\Levi_m^\gamma$ is a Levi subgroup in the connected reductive group $\G^\gamma$ with parabolic subgroup $\Para_m^\gamma$.
\end{proof}

The subgroup $V$ constructed in the proof of the following lemma is sometimes also referred to as the extended Weyl group.

\begin{lemma}\label{preimageweyl}
	There exists a subgroup $V \subseteq \mathrm{N}_{\G}(\T_0)$ and a surjective group homorphism $\rho:V \to W$ such that the following holds.
	Every $\gamma$-stable element $w \in \mathrm{N}_\G(\T_0)/\T_0$ has a $\langle F_p,\gamma \rangle$-stable preimage in $\mathrm{N}_\G(\T_0)$.
\end{lemma}

\begin{proof}
 Define $V:=\langle n_{(i,i+1)}, \mid i=1,\dots,n \rangle$, where $n_{(i,i+1)}$ is the matrix obtained by taking the permutation matrix corresponding to $(i,i+1)$ and multiplying its $i$th row with $-1$. We have a group epimorphism $V \to W,\, n_{(i,i+1)} \mapsto (i,i+1)$ whose kernel is isomorphic to $(C_2)^{n+1}$. By construction, we have $\gamma( n_{(i,i+1)})=n_{\gamma((i,i+1))}$ and every $n_{(i,i+1)}$ is $F_p$-stable.
\end{proof}


%

%

\begin{lemma}\label{preimage}
	Every $\langle F, \gamma \rangle$-stable element of $\mathrm{N}_{\G}(\Levi_m)/\Levi_m$ has a $\langle F, \gamma \rangle$-stable preimage in $\mathrm{N}_{\G}(\Levi_m)$.
\end{lemma}

\begin{proof}
%
%
	Let $W_{\Levi_m}$ be the Weyl group of $\Levi_m$ with respect to the maximal torus $\T_0$. By duality, $W_{\Levi_m}$ is (anti-)isomorphic to the Weyl group $W^\circ(t_m)$ of $\Levi_m^\ast$. Since $W^\circ(t_m) \cong (S_e)^m$ by \cite[Proposition 5.2]{Bonnafe} a computation in $W \cong S_{n+1}$ shows that $\mathrm{N}_{W}(W_{\Levi_m}) \cong S_e \wr S_m$.  
	Using \cite[Corollary 12.11]{MT} we thus obtain $\mathrm{N}_\G(\Levi_m)/\Levi_m\cong \mathrm{N}_{W}(W_{\Levi_m})/W_{\Levi_m} \cong S_m$. We first show that every $\gamma$-stable element $\sigma$ of $\mathrm{N}_{W}(W_{\Levi_m})/W_{\Levi_m}$ has a $\gamma$-stable preimage in $W$. For this we write $\sigma \in S_m$ as a product $\sigma=\sigma_1 \cdots \sigma_r$ of disjoint cycles. Every such cycle $\sigma_i$ corresponds to a permutation of the $m$ components of $W_{\Levi_m} \cong (S_e)^m$. If $\gamma(\sigma_i)=\sigma_j$ for $i \neq j$ then we let $w_i \in W$ be a preimage of $\sigma_i$ such that the support of $w_i$ is on the integers which correspond to the blocks which support $\sigma_i$. Furthermore, we define $w_j:=\gamma(w_i)$. If $\gamma(\sigma_i)=\sigma_i$ then $\sigma_i$ is necessarily a transposition of the form $\sigma_i=(a,m-a+1)$, i.e. $\sigma_i$ swaps the $a$th block (corresponding to the integers $\{(a-1)e+1, \dots,ae \}$ ) and the $(m-a+1)$th block (corresponding to the integers $\{(m-a)e+1,\dots,(m-a+1)e\}$). Then we define $w_i$ to be the following product of $e$ transpositions:
	$$w_i:= \prod_{l=1}^e \big( (a-1)e+l,(m-a+1)e-l \big).$$
	By construction it is clear that $w_i$ is a $\gamma$-stable preimage of $\sigma_i$. By construction the permutations $w_i$ are only supported on the integers which correspond to the blocks which are supported by $\sigma_i$. Therefore, for $i \neq j$ the permutations $w_i$ and $w_j$ have disjoint support, in particular they commute. From this it follows that $w_\sigma:=w_1 \cdots w_r \in W$ is a preimage of $\sigma$ which is $\gamma$-stable. Then Lemma \ref{preimageweyl} shows the existence of an $\langle F , \gamma \rangle$-stable preimage in $\mathrm{N}_{\G}(\T_0)$.
\end{proof}

Denote by $\mathbf{N}_m/\Levi_m$ the subgroup of $\mathrm{N}_\G(\Levi_m)/\Levi_m$ generated by $v_m$.

\begin{lemma}\label{symmetric group}
	Let $n \in \mathrm{N}_{\G^F}(\Levi_m)$ such that $\gamma(n)n^{-1} \in \mathbf{N}_m$. Then there exists some $y \in \mathbf{N}_m$ such that $yn$ is $\langle F, \gamma \rangle$-stable.
\end{lemma}

\begin{proof}
	We abbreviate $v:=v_m$.	According to Lemma \ref{preimage} it is enough to show that if $w \in \N_{W}(W_{\Levi_m})$ with $\gamma(w)w^{-1} \in \langle v \rangle W_{\Levi_m}$ then there exists some $y \in \langle v \rangle$ such that $\gamma(yw) (yw)^{-1} \in W_{\Levi_m}$.
	
	Denote by $\bar{\,}:\N_{W}(W_{L_m}) \to \N_{W}(W_{L_m})/W_{L_m}$ the projection map. In what follows, we identify $\mathrm{N}_\G(\Levi_m)/\Levi_m$ with $S_m$. Note that the element $\overline v$ corresponds to an $m$-cycle of $S_m$. Assume first that $\overline{v}$ has odd order. Then the map $\langle \overline{v} \rangle \to \langle \overline{v} \rangle, \overline{v} \mapsto \gamma(\overline{v})\overline{v}^{-1}=\overline{v}^{-2}$ is surjective. Thus, there exists $\overline{x} \in \langle \overline{v} \rangle$ such that $\gamma(\overline{w}) \overline{w}^{-1}=\gamma(\overline{x}^{-1}) \overline{x}$. In other words, $\overline{xw}$ is $\gamma$-stable.
	
	Assume now that $\overline{v}$ has even order. We denote by $\mathrm{sgn}$ the sign map on the symmetric group $\mathrm{N}_\G(\Levi_m)/\Levi_m \cong S_m$ and by $A_m$ its kernel. Since $\overline{v}$ has even order we have $\mathrm{sgn}(\overline{v})=-1$. On the other hand, $\mathrm{sgn}(\gamma(\overline{w}) \overline{w}^{-1})=1$ and thus $\gamma(\overline{w}) \overline{w}^{-1} \in \langle \overline{v}^2 \rangle= \langle \overline{v} \rangle \cap A_m$. However, $\gamma(\overline{v}) \overline{v}^{-1}=\overline{v}^2$. Thus, there exists $\overline{x} \in \langle \overline{v} \rangle$ such that $\gamma(\overline{w})\overline{w}^{-1}=\gamma(\overline{x}^{-1}) \overline{x}$. In other words, $\overline{xw}$ is $\gamma$-stable.
\end{proof}

\subsection{Stabilizers of blocks}


Recall from the remarks before  Remark \ref{type} that we defined $f^\ast: \G^\ast \to \G^\ast$ as the unique morphism satisfying $f^\ast \circ \pi=\pi \circ f$, where $\pi: \G \to \G^\ast$ is the natural surjective map. 
The following observation will be used throughout this section:

\begin{lemma}\label{duality}
The surjective map $\pi: \G \to \G^\ast$ induces a natural isomorphism between the quotient groups $\mathrm{N}_{G \langle f \rangle}(\Levi',e_s^{L'})/L'$ and $\mathrm{C}_{(\G^\ast)^{F^\ast} \langle f^\ast \rangle}(s) ((\Levi')^\ast)^{F^\ast}/((\Levi')^\ast)^{F^\ast}$ given by sending the class of $g f$ to $\pi(g) f^\ast$.
\end{lemma}

\begin{proof}
By definition of $f$ and $f^\ast$ it follows that for every element $g \in \G^F$ the morphisms $\mathrm{ad}(g) f$ and $\mathrm{ad}(\pi(g)) f^\ast$ are in duality. Now assume that $ g\in \G^F$ is such that $g f$ stabilizes the pair $(\Levi',e_s^{L'})$. It follows that $\pi(g) f^\ast$ stabilizes $(\Levi')^\ast$. Since $(\mathrm{ad}(\pi(g)) f^\ast)|_{(\Levi')^\ast}$ is in duality with $(\mathrm{ad}(g) f)|_{\Levi'}$ we deduce that $\pi(g) f^\ast$ stabilizes the $((\Levi')^\ast)^{F^\ast}$-conjugacy class of $s$. Since $\pi$ maps $(\Levi')^F$ to $((\Levi')^\ast)^{F^\ast}$ it follows that the map in the statement of the lemma is well-defined and injective. On the other hand, assume that $x f^\ast$ stabilizes the $((\Levi')^\ast)^{F^\ast}$-conjugacy class of $s$. Let $g' \in \G$ such that $\pi(g')=x$. Hence, $F(g')(g')^{-1} \in \mathrm{Z}(\G) \subseteq \Levi'$ and by applying Lang's theorem to $F: \Levi' \to \Levi'$ we deduce that there exists some $l' \in (\Levi')^\ast$ such that $g:=g'l'$ is $F$-stable. We have $\pi(g)=x \pi(l')$ with $\pi(l') \in ((\Levi')^\ast)^{F^\ast}$ and so indeed $g f$ is a preimage of $x f^\ast$.
\end{proof}

The most striking property of strictly quasi-isolated blocks (concerning the action of group automorphisms) is the following:

\begin{lemma}\label{graph}
	Assume that $e_s^{\G^F}$ is $\gamma$-stable. Then there exists some $m \in \G^F$ and a parabolic subgroup $\Para$ with Levi complement $\Levi$ such that $y \gamma$ stabilizes $(\Levi,\Para)$ and the idempotent $e_s^{\Levi^F}$.
\end{lemma}

\begin{proof}
	Let $\gamma^\ast: \G^\ast \to \G^\ast$ be the dual of the graph automorphism $\gamma$. There exists some $g \in \G^\ast$ such that ${}^g s =t$. Consider the element $$\tilde{t}=\tilde{t}_m:=(1,\zeta,\dots,\zeta^{m-1}) \otimes I_e\in \Gtilde^\ast=\mathrm{GL}_{n+1}(\overline{\mathbb{F}_p}),$$ a preimage of $t$ under the natural map $\mathrm{GL}_{n+1}(\overline{\mathbb{F}_p}) \to \mathrm{PGL}_{n+1}(\overline{\mathbb{F}_p})$. This element satisfies $\iota^\ast(\tilde{t})=t$ and we observe that $\gamma^\ast(\tilde{t})=tz$, where $z:=\zeta I_{n+1} \in \mathrm{Z}(\Gtilde)$. So, $t$ is $\gamma^\ast$-stable and thus $\gamma^\ast(s)={}^{\gamma^\ast(g^{-1}) g} s$.
	
	Since ${}^{\gamma^\ast} t = t$ we conclude that ${}^{g^{-1} \gamma^\ast(g) \gamma^\ast} s={}^{g^{-1} \gamma^\ast g} s=s$. Consider the element $x:=g^{-1} \gamma^\ast(g)$. By assumption there exists some $y \in (\G^\ast)^{F^\ast}$ such that $s$ is $y \gamma^\ast$-stable. Therefore $x y^{-1} \in \mathrm{C}_{\G^\ast}(s)$. Since $A(s)^{F^\ast}=A(s)$ by Lemma \ref{strict} we deduce that 
	$$F^\ast(xy^{-1})  (xy^{-1})^{-1}=F^\ast(x)x^{-1} \in \mathrm{C}^\circ_{\G^\ast}(s)=\Levi^\ast.$$
	Thus, by Lang's theorem there exists some $l \in \mathrm{C}_{\G^\ast}^\circ(s)$ such that $lx$ is $F^\ast$-stable. We have
	$(g^{-1} \gamma^\ast g)^2=1$ which implies that $(lx \gamma^\ast)^2 \in \mathrm{C}^\circ_{\G^\ast}(s)$.
	Recall that $\Levi^\ast={}^g \Levi_m^\ast$ and that $W^\circ(t_m)$ denotes the Weyl group of $\mathrm{C}^\circ_{\G^\ast}(t_m)$ relative to the maximal torus $\T_0^\ast$. Let $\Para_m^\ast$ be the standard parabolic subgroup associated with the parabolic subsystem $W^\circ(t_m)$ of $W^\ast$ with Levi complement $\Levi_m^\ast$. We define $\Para^\ast:={}^g \Para_m^\ast$ and we observe that the pair $(\Levi_m^\ast,\Para_m^\ast)$ is $\gamma^\ast$-stable. Consequently, $x \gamma^\ast$ stabilizes the pair $(\Levi^\ast,\Para^\ast)$ and since $l \in \Levi^\ast$ we observe that $lx \gamma^\ast$ also stabilizes this pair. The statement follows now by using duality, see Lemma \ref{duality}.
\end{proof}

We observe that the conclusion of Lemma \ref{graph} remains valid for the Levi subgroup $\Levi'$ constructed in Lemma \ref{prime order}.

\begin{corollary}\label{graph2}
	With the assumption of Lemma \ref{graph}, there exists a parabolic subgroup $\Para'$ whose Levi complement is the Levi subgroup $\Levi'$ from Lemma \ref{prime order} such that $y \gamma$ stabilizes $(\Levi',\Para')$ and the idempotent $e_s^{L'}$.
\end{corollary}

\begin{proof}
	By construction, we have $((\Levi')^\ast,(\Para')^\ast)={}^g (\Levi_{m'},\Para_{m'})$ where $g$ is as in the proof of Lemma \ref{graph}. The graph automorphism $\gamma$ stabilizes $((\Levi_{m'})^\ast,(\Para_{m'})^\ast)$ and we have $\Levi^\ast \subseteq (\Levi')^\ast$. It follows that the element $l g^{-1} \gamma^\ast(g) \gamma^\ast$, where $l \in (\Levi^\ast)^{F^\ast}$ is as in the proof of Lemma \ref{graph}, stabilizes $((\Levi')^\ast,(\Para')^\ast)$ and the $((\Levi')^\ast)^{F^\ast}$-conjugacy class of $s$. This yields the claim.
\end{proof}

\subsection{Untwisted groups of type $A$}

%

In the following section we assume that $(\G,F)$ is untwisted of type $A$. Consider the subgroup $\mathcal{B}$ of $\mathrm{Aut}(\Gtilde^F)$ generated by the field automorphism $F_p$ and the graph automorphism $\gamma$. Denote by $\mathcal{B}_{e_s^{\G^F}}$ the stabilizer of $e_s^{\G^F}$ in $\mathcal{B}$. We have $\mathcal{B}_{e_s^G}=\langle \gamma_0, \phi \rangle,$ where $\phi: \G \to
 \G$ is a Frobenius endomorphism (a power of $F_p$ or $F_p \gamma$) which satisfies $\phi^r=F$ for some $r$ and $\gamma_0 \in \{ \mathrm{id}_\G, \gamma\}$ is a (possibly trivial) graph automorphism. As in the proof of Lemma \ref{graph} we use some explicit properties of the element $s$ to say something about their structure.

\begin{lemma}\label{Frobcyclic}
The quotient group $\mathrm{N}_{G \langle \phi \rangle}(\Levi',e_s^{L'})/L'$ is cyclic unless $\Levi'$ is $1$-split.
\end{lemma}

\begin{proof}
%
%
%
 We have $t={}^g s$ for some $g\in \G^\ast$. By the construction in Lemma \ref{prime order} the quotient group $\mathrm{C}_{\G^\ast}(s)(\Levi')^\ast/(\Levi')^\ast$ is cyclic of prime order. In particular, $n':=g^{-1} F^\ast(g) \in \mathrm{C}_{\G^\ast}(s)$ and so $n'$ either generates $\mathrm{C}_{\G^\ast}(s)(\Levi')^\ast/(\Levi')^\ast$ or $n' \in (\Levi')^\ast$.

Recall that $\phi$ is a power of $F_p$ or $F_p \gamma$ and $\gamma$ fixes $t$, see proof of Lemma \ref{graph}. Consequently, $\phi$ acts as a permutation on the different eigenvalues of $t$. Therefore, there exists some $z \in W^\ast$ of order dividing $r$ such that ${}^{z \phi} t = t$. We let $z_0 \in \mathrm{N}_{\G^\ast}(\T_0^\ast)$ be the permutation matrix corresponding to $z$.

   We deduce that $s$ is $x \phi^\ast$-stable, where $x:=g^{-1}z_0 \phi^\ast(g)$. By assumption there exists some $y \in (\G^\ast)^{F^\ast}$ such that $s$ is $y \phi^\ast$-stable. Therefore $x y^{-1} \in \mathrm{C}_{\G^\ast}(s)$. Since $A(s)^{F^\ast}=A(s)$ by Lemma \ref{strict} we deduce that $F^\ast(xy^{-1})  (xy^{-1})^{-1}=F^\ast(x)x^{-1} \in \mathrm{C}^\circ_{\G^\ast}(s)$. Thus, by Lang's theorem there exists some $l \in \mathrm{C}_{\G^\ast}^\circ(s)$ such that $lx$ is $F^\ast$-stable. We have 
	$$(x\phi^\ast)^r=(g^{-1} z_0 \phi^\ast g)^r=g^{-1} (z_0 \phi^\ast)^r g=g^{-1} F^\ast(g) F^\ast=n' F^\ast.$$
	If $n'$ generates $(\mathbf{N}')^\ast/(\Levi')^\ast= \mathrm{C}_{(\G^\ast)^{F^\ast}}(s)(\Levi')^\ast/ (\Levi')^\ast$ then we can conclude by duality that $\mathrm{N}_{G \langle \phi \rangle }(L',e_s^{L'})/L'$ is cyclic. On the other hand, if $g^{-1} F^\ast(g) \in (\Levi')^\ast$ then there exists $l' \in (\Levi')^\ast$ such that $gl'$ is $F^\ast$-stable. We conclude that $(\Levi')^\ast={}^{gl}(\Levi_{m'})^\ast$ is $(\G^\ast)^{F^\ast}$-conjugate to $(\Levi_{m'})^\ast$ and therefore maximally split.
\end{proof}


\begin{lemma}\label{quotient=2}
If $\Levi'$ is not $1$-split, then the quotient group $\mathrm{N}_{G \mathcal{B}}(\Levi',e_s^{L'})/L'$ is abelian.	
\end{lemma}

\begin{proof}
	We keep the notation of Lemma \ref{Frobcyclic}.
	As in Lemma \ref{Frobcyclic} there exists some $z_0 \in \N_{\G^\ast}(\T_0)$ of order dividing $r$ such that ${}^{z_0 \phi^\ast} t = t$. The element $\gamma_0^\ast(z_0)z_0^{-1}$ stabilizes $t$ and so $\gamma_0^\ast(z_0)z_0^{-1} \in \Nb_{m'}^\ast$. We want to show that $z_0$ can be chosen to be $\gamma_0^\ast$-stable.	By Lemma \ref{symmetric group} we first see that there exists some $y \in \Nb^\ast_{m'}$ such that $yz_0$ is $\langle F^\ast, \gamma^\ast \rangle$-stable. Observe that we still have ${}^{yz_0 \phi^\ast} t=t$ and $(yz_0)^r \in \mathbf{N}^\ast_{m'}$. Note that $ v_{m'}$ corresponds to an $m'$-cycle under the isomorphism $\mathrm{N}_{\G^\ast}(\Levi^\ast_{m'})/\Levi^\ast_{m'} \cong S_{m'}$. Since $m'$ is a prime number it follows that if some power of $\sigma \in S_{m'}$ is an $m'$-cycle then $\sigma$ itself must be an $m'$-cycle. Since $(yz_0)^r \in \mathbf{N}^\ast_{m'}$ this argument shows that we must either have $yz_0 \in \mathbf{N}^\ast_{m'}$ or $(yz_0)^r \in \Levi^\ast_{m'}$. In this first case we must already have $z_0 \in \mathbf{N}_{m'}^\ast$ and so $t$ is $\phi^\ast$-stable. In this case we can simply replace $z_0$ by $1$. In the second case we can replace $z_0$ by $yz_0$. Thus, we can assume that $z_0$ is $\gamma^\ast$-stable.
	
	We denote $\phi'^\ast:=z_0 \phi^\ast$. Consider $x_{\phi'}:=g^{-1} \phi'^\ast(g)$ and $x_{\gamma_0}:=g^{-1} \gamma_0^\ast(g)$. Since $\phi'^\ast=z_0 \phi^\ast$ and $\gamma^\ast$ commute it follows that $$[x_{\phi'} \phi'^\ast,x_\gamma \gamma_0^\ast]=[g \phi'^\ast g^{-1},g \gamma_0^\ast g^{-1}]=1.$$
	The statement of the lemma follows now from Lemma \ref{duality}.
\end{proof}

\begin{lemma}\label{Frobroot}
	Assume that $\Levi'$ is $1$-split. By possibly replacing $s$ by a $(\G^\ast)^{F^\ast}$-conjugate the following hold:
	\begin{enumerate}[label={\alph*)}]
		\item  The Levi subgroups $(\Levi',\Para')$ and $e_s^{L'}$ are $\gamma_0$-stable.
		\item There exists a Frobenius endomorphism $F_0$ of $\G$ such that $F_0^r=F$ which commutes with $\gamma_0$ such that $F_0$ stabilizes $\Levi'$ and $e_s^{L'}$. Moreover, $\Gtilde^F \langle F_0, \gamma_0 \rangle$ is the stabilizer of $e_s^{\G^F}$ in $\mathrm{Out}(\G^F)$.
	\end{enumerate}
\end{lemma}

\begin{proof}
	As in Lemma \ref{Frobcyclic} let $g \in \G^\ast$ such that ${}^g s=t$. Our assumption implies that $g^{-1} F^\ast(g) \in (\Levi')^\ast$. In particular there exists some $l' \in (\Levi')^\ast$ such that $s$ and ${}^{l'} t$ are $(\G^\ast)^{F^\ast}$-conjugate. Therefore, we can replace $s$ by ${}^{l'} t$. By doing so we can assume that there exists some $ g\in (\Levi_{m'})^\ast$ such that ${}^g s=t$. Thus, we may assume that $\Levi'=\Levi_{m'}$.
	
	Corollary \ref{graph2} implies the existence of $m \in G$ such that $m\gamma_0$ stabilizes $(\Levi',\Para')$ and $e_s^{L'}$. Since $(\Levi',\Para')$ is $\gamma_0$-stable it follows that $m \in \N_G(\Levi',\Para')$. Since parabolic subgroups are self-normalizing (see e.g. \cite[Exercise 20.3]{MT}) we obtain $\N_{\G}(\Levi',\Para')=\Levi' \N_{\U'}(\Levi')$. By the proof of \cite[Corollary 1.18]{DM} we obtain $\N_{\U'}(\Levi')=1$ and so $m \in \N_G(\Levi',\Para')=L'$. This yields the first claim.
	
	There exists some $z \in \N_G(\Levi')$ such that $z \phi$ stabilizes $e_s^{L'}$. Thus, $\gamma_0(z) z^{-1} \in N'$. By Lemma \ref{symmetric group} there exists $y \in \mathbf{N'}$ such that $yz \in N'$ is $\gamma$-stable. Hence, we may assume that $z$ is $\gamma$-stable. As in the proof of Theorem \ref{equiv} we see that $z_0:=N_{F/\phi}(z) \in \mathrm{N}_{\G^\phi}(\Levi', e_s^{L'})$. By the proof of Lemma \ref{Frobcyclic} we can assume that $z_0\in (L')^{\gamma_0}$.
By Lemma \ref{connectedgraph} $(\Levi')^{\gamma_0}$ is connected reductive. Consequently, by Lemma \ref{norm} there exists $l \in ((\Levi')^{\gamma_0})^{z_0 F}$ such that $z_0^{-1}=N_{z_0 F/z \phi}(l)$. From this we deduce that $F_0:= \mathrm{ad}(l z_0) \phi$ commutes with $\gamma_0$ and satisfies $F_0^r=F$.  
\end{proof}

The previous lemma as well as Theorem \ref{equiv} suggest that we should distinguish the cases whether $\Levi'$ is $1$-split or not.

\subsection{Twisted groups}

We will now assume that $(\G,F)$ is twisted of type $A_n$. We consider again the subgroup $\mathcal{B}$ of $\mathrm{Aut}(\Gtilde^F)$ generated by the field automorphism $F_p$ and the graph automorphism $\gamma$. Since $F=F_q \gamma$ we observe that $\mathcal{B}$ is generated by $F_p$. In particular, we have $\mathcal{B}_{e_s^G}=\langle  \phi \rangle,$ where $\phi: \G \to \G$ is a Frobenius endomorphism (a power of $F_p$) which satisfies $\phi^r=F \gamma_0$ and $\gamma_0$ is a possibly trivial graph automorphism.

We mention that the proof of the following lemma is very similar to Lemma \ref{Frobcyclic} and Lemma \ref{quotient=2} however the presence of a non-trivial graph automorphism causes some additional difficulties.

\begin{lemma}\label{Frobcyclic2}
	The quotient group $\mathrm{N}_{G \langle \phi \rangle}(\Levi',e_s^{L'})/L'$ is abelian unless $\Levi'$ is $1$-split.
\end{lemma}

\begin{proof}
	%
	%
	%
	We have $t={}^g s$ for some $g\in \G^\ast$. In particular, $n':=g^{-1} F^\ast(g) \in \mathrm{C}_{\G^\ast}(s)$ and so $n'$ either generates $\mathrm{C}_{(\G^\ast)^{F^\ast}}(s)(\Levi')^\ast/(\Levi')^\ast$ or $n' \in (\Levi')^\ast$.
	
	Recall that $\phi$ is a power of $F_p$ or $F_p \gamma$ and $\gamma^\ast$ stabilizes $t$, see proof of Lemma \ref{graph}. Consequently, $\phi^\ast$ acts as a permutation on the different eigenvalues of $t$. Therefore, there exists some $z \in W^\ast$ of order dividing $r$ such that ${}^{z \phi^\ast} t = t$.
	We deduce that $s$ is $x_\phi \phi$-stable, where $x_\phi:=g^{-1}z \phi(g)$. By assumption there exists some $y \in (\G^\ast)^F$ such that $s$ is $y \phi$-stable. Therefore, $x_\phi y^{-1} \in \mathrm{C}_{\G^\ast}(s)$. Since $A(s)^{F^\ast}=A(s)$ by Lemma \ref{strict} we deduce that $F^\ast(x_\phi y^{-1})  (x_\phi y^{-1})^{-1}=F^\ast(x_\phi )x_\phi^{-1} \in \mathrm{C}^\circ_{\G^\ast}(s)$. Thus, there exists some $l_\phi \in \mathrm{C}_{\G^\ast}^\circ(s)$ such that $x_\phi':=l_\phi x_\phi$ is $F^\ast$-stable.
	
	Note that $g^{-1} \gamma_0^\ast(g) \gamma_0^\ast$ stabilizes $s$ by the proof of Lemma \ref{graph}. By the same reasoning as above for $x_{\gamma_0}:=g^{-1} \gamma_0^\ast(g)$ there exists some $l_{\gamma_0}$ such that $x_{\gamma_0}':=l_{\gamma_0} x_{\gamma_0}$ is $F^\ast$-stable. 
	Furthermore, we have 
	$$(x_\phi \phi^\ast)^r x_{\gamma_0} \gamma_0^\ast=g^{-1} F_q^\ast(g) F_q^\ast g^{-1} \gamma_0^\ast(g) \gamma_0^\ast= g^{-1} F_q^\ast(\gamma_0^\ast(g)) F_q^\ast \gamma_0^\ast=g^{-1} F^\ast(g) F^\ast.$$
	We conclude that if $g^{-1} F^\ast(g) \notin (\Levi')^\ast$ then the quotient group is generated by $x_\phi \phi^\ast$ and $x_{\gamma_0} \gamma_0^\ast$. 	As in Lemma \ref{quotient=2} one now shows that the quotient group is abelian. On the other hand, if $g^{-1} F^\ast(g) \in (\Levi')^\ast$ then there exists $l' \in (\Levi')^\ast$ such that $gl'$ is $F^\ast$-stable. We conclude that $(\Levi')^\ast={}^{gl}(\Levi'_m)^\ast$ is $1$-split.
\end{proof}

\begin{lemma}\label{Frobroot2}
	Assume that $\Levi'$ is $1$-split. By possibly replacing $s$ by a $(\G^\ast)^{F^\ast}$-conjugate the following hold:
		\begin{enumerate}[label={\alph*)}]
		\item 	The graph automorphism $\gamma_0$ stabilizes $(\Levi',\Para')$ and the idempotent $e_s^{L'}$.
		\item 	There exists a Frobenius endomorphism $F_0$ such that $F_0^r=F \gamma_0$ which commutes with $\gamma_0$ and stabilizes $\Levi'$ and $e_s^{L'}$. Moreover, $\Gtilde^F \langle F_0 \rangle$ is the stabilizer of $e_s^{\G^F}$ in $\mathrm{Out}(\G^F)$.
	\end{enumerate}
\end{lemma}

\begin{proof}
	By Lemma \ref{Frobcyclic2} we can assume that $\Levi'=\Levi_{m'}$. Corollary \ref{graph2} implies the existence of $m \in G$ such that $m\gamma$ stabilizes $(\Levi',\Para')$ and $e_s^{L'}$. Since $(\Levi',\Para')$ is $\gamma$-stable it follows that $m \in \N_G(\Levi',\Para')=L$. This yields the first claim.

There exists some $z \in \N_G(\Levi')$ such that $z \phi$ stabilizes $e_s^{L'}$. Thus, $\gamma(z) z^{-1} \in N'$. By Lemma \ref{symmetric group} there exists $y \in \mathbf{N'}$ such that $yz \in N'$ is $\gamma$-stable. Hence we may assume that $z$ is $\gamma$-stable. Since $e_s^{L'}$ is $\gamma$-stable we deduce that $z_0:=N_{F_q/\phi}(z) \in \mathrm{N}_{\G^\phi}(\Levi', e_s^{L'})$. The proof of Lemma \ref{Frobcyclic2} shows that we must have $z_0 \in L^\gamma$. By Lemma \ref{connectedgraph} $(\Levi')^\gamma$ is connected reductive. Consequently, by Lemma \ref{norm} there exists $l \in ((\Levi')^\gamma)^{z_0 F}$ such that $z_0^{-1}=N_{z_0 F/z \phi}(l)$. From this we deduce that $F_0:= \mathrm{ad}(l z_0) \phi$ commutes with $\gamma$ and satisfies $F_0^r=F_q$.  
\end{proof}
%

\section{Global equivalences}\label{sec 7}

\subsection{Extending modules}

Let $X$ be a normal subgroup of a finite group $Y$. If $M$ is a $\Lambda X$-module then in practice it is often quite hard to decide whether $M$ extends to a $\Lambda Y$-module if $\ell \mid |Y:X|$ unless $M$ is simple or projective. We will therefore in the following often consider the following weaker notion:

\begin{definition}\label{almost}
	Let $X$ be a normal subgroup of a finite group $Y$ such that the quotient group $Y/X$ is solvable. If $M$ is a $\Lambda X$-module then we say that it \textit{almost extends to $Y$} if $M$ extends to a $\Lambda H$-module $M'$, where $H/X$ is a Hall $\ell'$-subgroup of $Y/X$, such that the extension $M'$ is $Y$-stable.
\end{definition}

We will check almost extendibility with the following remark.

\begin{remark}\label{normal hall}
	In the situation of the previous definition assume that $M$ is $Y$-stable and $E:=\mathrm{End}_{\Lambda N}(M)/ J(\mathrm{End}_{\Lambda N}(M))$ is abelian. If the Hall $\ell'$-subgroup $H/X$ is normal in $Y/X$ and $M$ extends to $H$ then by \cite[Corollary 2.6]{Thevenaz} the module $M$ automatically almost extends to $Y$.
	\end{remark}

We will later use the following result about almost extendibility.

\begin{lemma}\label{abelian factor}
	Let $X$ be a normal subgroup of a finite group $Y$ such that $Y/X$ is abelian. Assume that $M$ is a $\Lambda X$-module such that $E:=\mathrm{End}_{\Lambda X}(M)/ J(\mathrm{End}_{\Lambda X}(M))$ is abelian and $M$ almost extends to $Y$. If $M'$ is an extension of $M$ to a subgroup $X'$ of $Y$ such that $\ell \nmid |X':X|$ then $M'$ almost extends to $Y$ as well.
\end{lemma}

\begin{proof}
Let $H/X$ be the Hall $\ell'$-subgroup of $Y/X$. By assumption, $H/X'$ is the Hall $\ell'$-subgroup of $Y/X'$. By Remark \ref{normal hall} it is therefore enough to show that $M'$ extends to $H$.

Assume that $U$ is maximal among the subgroups of $H$ containing $X'$ with the property that $M'$ extends to $U$. Denote by $M''$ an extension of $M'$ to $U$. If $M''$ is stable under some element $x \in H \setminus U$ then by \cite[Lemma 10.2.13]{Rouquier3} the module $M''$ extends to $\langle U,x\rangle$. Therefore, the stabilizer of $M''$ in $H$ is $U$.
	
	Denote by $\bar{\,}: Y\to Y/X$ the projection map.
	 By the remark following \cite[Theorem 2.5]{Thevenaz} the set of $H$-invariant extension of $M$ to $U$ is in bijection with $H^1(\overline{U}, E^\times)^{\overline{H}/ \overline{U}}$. Since $Y/X$ is abelian, the action of ${\overline{H}/ \overline{U}}$ on this cohomology group is trivial. On the other hand, by \cite[Proposition 2.4]{Thevenaz} the set of extensions of $M$ to $U$ is in bijection with $H^1(\overline{U}, E^\times)$. From this it follows that every extension of $M$ to $U$ is $H$-stable. In particular, $M'$ is $H$-stable and we conclude that $U=H$.
\end{proof}

\subsection{Some auxiliary results}

Our aim in this section is to show that the cohomology module $H_c^{\mathrm{dim}}(\Y_{\U'}^\G,\Lambda) e_s^{L'}$ almost extends to a $\Lambda (G \times (L')^{\mathrm{opp}} \Delta \mathrm{N}_{G\mathcal{A}}(\Levi',e_s^{L'}))$-module, where $\mathcal{A}$ (defined below in Definition \ref{defA}) is a subgroup of $\mathrm{Aut}(\tilde{\G}^F)$ whose image in the outer automorphism group is $\mathrm{Out}(\G^F)_{e_s^{\G^F}}$. Thanks to Lemma \ref{graph} many situations can be handled by the following lemma.

\begin{lemma}\label{extension}

	Assume that there exist $m_{\phi},m_\gamma \in \G^F$ such that  $\N_{ \tilde{G} \mathcal{B}}(\Levi',e_s^{L'})=\tilde L' \langle m_{\phi} \phi,m_\gamma \gamma \rangle$. If $m_\gamma \gamma$ stabilizes $\Para'$ and $\N_{ \tilde{G} \mathcal{B}}(\Levi',e_s^{L'})/\tilde{L}'$ is abelian then $H_c^{\mathrm{dim}}(\Y_{\U'}^\G,\Lambda) e_s^{L'}$ almost extends to $G \times (L')^{\mathrm{opp}} \Delta( \N_{ \tilde{G} \mathcal{B}}(\Levi',e_s^{L'}) )$.
\end{lemma}

\begin{proof}
Recall that for disconnected reductive groups we define parabolic subgroups and Levi subgroups as in \cite[Section 2.1]{Jordan}.
Consider the disconnected reductive group $\hat{\G}:=\tilde\G \rtimes \langle \gamma \rangle$. Since $\phi$ and $F$ commute with $\gamma$ there exist unique extensions of $\phi$ and $F$ to $\hat{\G}$ with $\phi(\gamma)=\gamma$ and $F(\gamma)=\gamma$ respectively. We observe that $\hat{\Para}'=\mathrm{N}_{\hat{\G}}(\tilde \Para')$ is a parabolic subgroup of $\hat{\G}$ with Levi subgroup $\hat{\Levi}'=\mathrm{N}_{\hat{\G}}(\tilde \Levi',\tilde \Para)=\mathrm{N}_{\hat \Para'}(\tilde \Levi')$. We abbreviate $(\hat{\Levi}')^F:=\hat{L}'$. The Deligne--Lusztig variety $\Y_{\U'}^\G$ has a natural $G\times (L')^{\mathrm{opp}} \Delta( \hat{L}')$-action. Observe that $\hat{\Levi}'$ is $m_\phi \phi$-stable since $[m_\gamma \gamma, m_\phi \phi] \in \Levi'$. Consider the morphism $\phi':=\mathrm{ad}(m_{\phi}) \phi$ of $\hat{\G}$. It follows that $\phi': \Y_{\U'}^\G \to \Y_{\phi'(\U')}^\G$ is a bijective morphism of $G\times (L')^{\mathrm{opp}} \Delta (\hat{L}')$-varieties. Since $m_\gamma \gamma$ stabilizes $\Para'$ by assumption it follows that $m_\gamma \gamma \in \hat{\Levi'}$ and thus the quotient group $\hat{\Levi'}/\Levi'$ is generated by $m_\gamma \gamma$. We therefore obtain that $e_s^{L'}$ is $\hat{L'}$-stable. We have 
	$H_c^{\mathrm{dim}}(\Y_{\U'}^\G,\Lambda) e_s^{L'} \cong H_c^{\mathrm{dim}}(\Y_{\phi'(\U')}^\G,\Lambda) e_s^{L'}$
	as $G\times (L')^{\mathrm{opp}} \Delta \hat{L'}$-modules. From this we conclude that the cohomology module $H^{\mathrm{dim}}_c(\Y_{\U'}^\G,\Lambda) e_s^{L'}$ is $\Delta \langle m_{\phi} \phi \rangle$-stable. Hence it almost extends to $G \times (L')^{\mathrm{opp}} \Delta( \hat{L}' \langle m_\phi \phi \rangle)$.
\end{proof}

The following corollary is the local version of Lemma \ref{extension}.

\begin{corollary}\label{extension local}

Let $Q$ be an $\ell$-subgroup of $L'$. Under the assumptions of Lemma \ref{extension} the module $H_c^{\mathrm{dim}}(\Y_{\C_{\U'}(Q) }^{\N_{\G}(Q)},\Lambda) \br_Q(e_s^{L'})$ almost extends to a $\Lambda(\N_G(Q) \times (\N_{L'}(Q))^{\mathrm{opp}} \Delta( \N_{ \tilde{G} \mathcal{B}}(\Levi',Q,e_s^{L'}) )$-module.
\end{corollary}

\begin{proof}
	If the quotient group $\N_{ \tilde{G} \mathcal{B}}(L',e_s^{L'})/\tilde{L}'$ is cyclic then the statement is trivial. We can therefore assume that $\N_{ \tilde{G} \mathcal{B}}(L',e_s^{L'})/\tilde{L}'$ is non-cyclic. From this we deduce that there exist $l_\phi, l_\gamma \in \tilde{L}$ such that $\N_{ \tilde{G} \mathcal{B}}(L',Q,e_s^{L'})=\N_{\tilde{L}}(Q) \langle l_\phi (m_\phi \phi)^i,l_\gamma m_\gamma \gamma \rangle \rangle$ for some $i$.
The Deligne--Lusztig variety $\Y_{\C_{\U'}(Q)}^{\N_\G(Q)}$ has a natural $\N_{G}(Q)\times (\N_{L'}(Q))^{\mathrm{opp}} \Delta( \N_{\hat{\Levi}'^F}(Q))$-action. Consider the morphism $\phi'$ induced by the action of $l_\phi (m_\phi \phi)^i$ on $\hat{\G}$. It follows that $\phi': \Y_{\C_{\phi'(\U')}(Q)}^{\N_\G(Q)}\to \Y_{\C_{\U'}(Q)}^{\N_\G(Q)}$ is a bijective morphism of $\N_{G}(Q)\times (\N_{L'}(Q))^{\mathrm{opp}} \Delta( \N_{\hat{\Levi}'^F}(Q))$-varieties. Since $m_\gamma \gamma \in \hat{\Levi'}$ we deduce that the quotient group $\N_{\hat{\Levi'}}(Q)/\N_{\Levi'}(Q)$ is generated by $l_\gamma m_\gamma \gamma$. We have 
$H_c^{\mathrm{dim}}(\Y_{\C_{\U'}(Q)}^{\N_\G(Q)},\Lambda) \br_Q(e_s^{L'}) \cong H_c^{\mathrm{dim}}(\Y_{\C_{\phi'(\U')}(Q)}^{\N_\G(Q)},\Lambda) \br_Q(e_s^{L'})$
as $\N_{G}(Q)\times (\N_{L'}(Q))^{\mathrm{opp}} \Delta( \N_{\hat{\Levi}'^F}(Q))$-modules. From this we conclude that the cohomology module $H_c^{\mathrm{dim}}(\Y_{\C_{\U'}(Q)}^{\N_\G(Q)},\Lambda) $ is $\Delta \langle \phi' \rangle$-stable as $\N_{G}(Q)\times (\N_{L'}(Q))^{\mathrm{opp}} \Delta( \N_{\hat{\Levi}'^F}(Q))$-module. It therefore almost extends to a $\Lambda(\N_G(Q) \times (\N_{L'}(Q))^{\mathrm{opp}} \Delta( \N_{ \tilde{G} \mathcal{B}}(\Levi',Q,e_s^{L'}) )$-module.
\end{proof}

\subsection{Global equivalences}

In the following we prove an analog of Theorem \ref{equiv} for strictly quasi-isolated blocks in groups of type $A$. To formulate this theorem we need the following definition.

\begin{definition}\label{defA}
	If $\Levi'$ is not $1$-split, then we define $\mathcal{A}:= \mathcal{B}_{e_s^{\G^F}}  \subseteq \mathrm{Aut}(\Gtilde^F)$.
	Otherwise, we define $\mathcal{A}:=\langle \gamma_0, F_0 \rangle  \subseteq \mathrm{Aut}(\Gtilde^F)$, where $F_0$ is as defined in the proof of Lemma \ref{Frobroot} and Lemma \ref{Frobroot2} respectively.
	In both cases we abbreviate $\mathcal{N}:=\mathrm{N}_{\Gtilde^F \mathcal{A}}(\Levi',e_s^{L'})$. 
\end{definition}


Firstly, observe that $\mathcal{A}$ is always abelian. If $\Levi'$ is not $1$-split then $\mathcal{A}$ is abelian as a subgroup of the abelian group $\mathcal{B}$ and if $\Levi'$ is $1$-split this follows from Lemma \ref{Frobroot} and Lemma \ref{Frobroot2} respectively.
Notice that the Hall $\ell'$-subgroup of $\mathcal{N}/\tilde L'$ is always normal. This follows from the fact that $N'/L'$ is of $\ell'$-order (see {\cite[Corollary 2.9]{Bonnafe}) and $\N_{G \mathcal{A}}(\Levi',e_s^{L'})/\tilde N' \cong \mathcal{A}$ is abelian. We denote by $\mathcal{N}_{\ell'}/\tilde L'$ the unique Hall $\ell'$-subgroup of $\mathcal{N}/\tilde L'$.

\begin{theorem}\label{mainthm}
The $\Lambda (G \times (L')^{\mathrm{opp}} \Delta (\tilde{L}'))$-module $H_c^{\mathrm{dim}}(\Y_{\U'}^\G,\Lambda) e_s^{L'}$ almost extends to a $\Lambda (G \times (L')^{\mathrm{opp}} \Delta \mathrm{N}_{G\mathcal{A}}(\Levi',e_s^{L'}))$-module.
\end{theorem}

\begin{proof}
	Assume first that $\Levi'$ is not  $1$-split. Thanks to the proof of Lemma \ref{quotient=2} and Lemma \ref{Frobcyclic2} we know that the quotient group $\mathcal{N}/\tilde L'$ is generated by $m_{\gamma_0} \gamma_0,m_\phi \phi$ which satisfy $[m_{\gamma_0} \gamma_0,m_\phi \phi] \in L'$. We can therefore use Lemma \ref{extension} to conclude that $H_c^{\mathrm{dim}}(\Y_{\U'}^\G,\Lambda) e_s^{L'}$ almost extends.
	
Assume therefore now that $\Levi'$ is $1$-split. We let $X:=G\times (L')^{\mathrm{opp}} \Delta \tilde{L}'$ and $M:=H^{\mathrm{dim}}_c(\Y_{\U'}^{\G},\Lambda) e_s^{L'}$. According to \cite[Theorem 7.5]{Dat} $\mathrm{Ind}_X^{\Gtilde^F \times (\tilde{\Levi}'^F)^{\mathrm{opp}}} (M')$ is multiplicity free, so $M$ is multiplicity free as well. Furthermore, the $\Lambda X$-module $M$ is $\Delta(\N_{G\mathcal{A}}(\Levi',e_s^{L'}))$-stable. Therefore, \cite[Proposition 1.13]{Thevenaz} and Remark \ref{normal hall} are applicable and we conclude that it is enough to show that for every prime $b$, $b \neq \ell$, the module $M$ extends to $X \Delta(P_b)$ where $P_b/ \tilde L'$ is a Sylow $b$-group of $\mathcal{N}/ \tilde L'$. If such a Sylow $b$-subgroup is cyclic then the claim follows from \cite[Lemma 10.2.13]{Rouquier3} and hence we can concentrate on all $b$ such that the Sylow $b$-subgroup is non-cyclic. Therefore, we can assume that $b \in \{2, m'\}$.

Let us consider the case where $b=m'$ and $m' \neq 2$. In this case, Lemma \ref{Frobroot} shows that $P_{m'}=N'\langle F_0' \rangle$, where $F_0'=F_0^i$ for some $i \mid r$. Since $N'/L' \cong C_{m'}$ and $F_0'$ has $m'$-power order we obtain that $F_0'$ centralizes $N'/L'$. By Lemma \ref{mainlemma}, we conclude that $M'$ extends to $X \Delta(P_{m'})$.

It remains to consider the case $b=2$. Assume first that $m' \neq 2$. Then we have $P_2=L' \langle F_0^i, \gamma_0 \rangle$. We can therefore use Lemma \ref{extension} to conclude that $M'$ extends to $X \Delta(P_2)$.
%
%
%
Finally, assume that $m'=2$. This implies that the quotient group $N'/L'$ is $\mathcal{A}'$-stable. We obtain $P_2=N' \langle \gamma_0, F_0' \rangle$, where $F_0'=F_0^i$ for some $i$ dividing $r$. In particular there exists $k$ such that $(F_0')^k=F$ or $(F_0')^k=\gamma F$. In the first case we can use Lemma \ref{mainlemma} to conclude that $M$ extends to a $X \Delta(P_2)$-module. In the latter case we observe that $P_2=N' \langle  F_0' \rangle$ and we can apply Lemma \ref{mainlemma2} to conclude that $M$ extends to a $X \Delta(P_2)$-module.
\end{proof}

We let $\hat{M}$ be an extension of $H_c^{\mathrm{dim}}(\Y_{\U'}^\G,\Lambda) e_s^{L'}$ to $\Lambda (G \times (L')^{\mathrm{opp}} \Delta \mathcal{N}_{\ell'})$ which is $\Delta \mathcal{N}$-stable. Furthermore, we denote by $M'$ the restriction of $\hat{M}$ to $G \times (N')^{\mathrm{opp}} \Delta \tilde{N}'$.

\begin{corollary}\label{typeA}
	Suppose the assumptions and notation as above. Then the bimodule $M'$ induces a Morita equivalence between $\Lambda N' e_s^{L'}$ and $\Lambda G e_s^G$ which lifts to a Morita equivalence between $\Lambda \mathcal{N}_{\ell'} e_s^{L'}$ and $\Lambda \tilde{G}^F \mathcal{A}_{\ell'} e_s^{G}$ given by $\Ind_{G \times (L')^{\opp} \Delta \mathcal{N}_{\ell'}}^{\tilde{G} \mathcal{A}_{\ell'} \times (\mathcal{N}_{\ell'})^{\opp}}(\hat{M})$.
\end{corollary}

\begin{proof}
	The first statement is a consequence of (the proof of) \cite[Theorem 7.5]{Dat}. The second statement then follows from \cite[Theorem 3.4]{Marcus}. 
\end{proof}

\section{Local equivalences}\label{sec 8}

The following simple group-theoretic lemma will be used in the proof of Lemma \ref{ext loc}

\begin{lemma}\label{group lemma}
	Let $X$ be a subgroup of a finite group $Y$ and $P$ a Sylow $b$-subgroup of $Y$. If $[Y,Y] \subseteq X$ then $P \cap X$ is a Sylow $b$-subgroup of $X$. 
\end{lemma}

\begin{proof}
	The hypothesis implies that $[P,X] \subseteq X$. From this we deduce that $PX=XP$ and so $PX$ is a subgroup of $Y$ containing $P$. It follows that $|P|_b=|PX|_b=\frac{|P|_b |X|_b}{|P \cap X|_b}$ and so $|P\cap X|_b=|X|_b$.
\end{proof}


\begin{lemma}\label{ext loc}
	Let $Q$ be an $\ell$-subgroup of $L'$. Suppose that either $\N_{\tilde N'}(Q) \tilde L'=\tilde N'$ or $\mathcal{N}/\tilde L'$ is abelian.
	The bimodule $H^{\mathrm{dim}}_c(\Y_{\mathrm{C}_{\U'}(Q)}^{\mathrm{N}_{\G}(Q)}, \Lambda) \br_Q(e_s^{L'})$ almost extends to an $\N_G(Q) \times \N_{N'}(Q)^{\opp} \Delta \N_{\mathcal{N}}(Q)$-module.
\end{lemma}

\begin{proof}
Let $X_Q:=\mathrm{N}_{G}(Q) \times \mathrm{N}_{L'}(Q)^{\mathrm{opp}} \Delta \mathrm{N}_{\tilde{L}'}(Q)$ and $M_Q:=H^{\mathrm{dim}}_c(\Y_{\mathrm{C}_{\U'}(Q)}^{\mathrm{N}_{\G}(Q)}, \Lambda) \br_Q(e_s^{L'})$. 
By Lemma \ref{multiplicity free}, the module $M_Q$ is multiplicity free. We conclude by Lemma \ref{normal hall} and \cite[Proposition 1.13]{Thevenaz} that it is enough to show that for every prime $b$ the module $M_Q$ almost extends to $X_Q \Delta(P_b)$ where $P_b/\mathrm{N}_{L'}(Q)$ is a Sylow $b$-group of $\mathrm{N}_{\mathcal{N}}(Q)/\mathrm{N}_{L'}(Q)$. Observe that $\mathrm{N}_{\mathcal{N}}(Q)/\mathrm{N}_{\tilde L'}(Q)$ embeds as a subgroup of $\mathcal{N}/\tilde L'$. Arguing as in the proof of Theorem \ref{mainthm} we can thus assume that $b \in \{2, m'\}$. Since $[\mathcal{N},\mathcal{N}] \subseteq \tilde{N}'$ it follows from Lemma \ref{group lemma} and our assumption on $Q$ that $(P_b \cap \mathrm{N}_{\mathcal N}(Q))/N_{\tilde L'}(Q)$ is a Sylow $b$-subgroup of $\mathrm{N}_{\mathcal N}(Q)/N_{\tilde L'}(Q)$. In the proof of Theorem \ref{mainthm} we have shown that  $H_c^{\mathrm{dim}}(\Y_{\U'}^\G,\Lambda) e_s^{L'}$ almost extends to a $\Lambda [G \times (L')^{\mathrm{opp}} \Delta \tilde{L}' P_b]$-module. The corresponding local result in Lemma \ref{extension local}, Lemma \ref{mainlemma:local} and Lemma \ref{mainlemma:local2} show that $H_c^{\mathrm{dim}}(\Y_{\mathrm{C}_{\U'}(Q)}^{\mathrm{N}_\G(Q)} ,\Lambda) \mathrm{br}_Q(e_s^{L'})$ almost extends to an $X_Q \Delta \mathrm{N}_{P_b'}(Q)$-module.
\end{proof}

Suppose now that $\ell \nmid |H^1(F,\mathrm{Z}(\G))|$.
As in Lemma \ref{diagonal action} let $\mathcal C'$ be a complex of $\Lambda (G \times (N')^{\mathrm{opp}} \Delta \tilde{N})$-modules such that $H^d(\mathcal C') \cong M'$ and $\mathcal C'$ induces a splendid Rickard equivalence between $\Lambda G e_s^{G}$ and $\Lambda N' e_s^{L'}$. According to the proof of \cite[Theorem 5.2]{Rickard} there exists a unique complex $\mathcal{C}_Q'$ of $\ell$-permutation $\Lambda(\mathrm{C}_G(Q) \times \mathrm{C}_{N'}(Q)^{\mathrm{opp}} \Delta \mathrm{N}_{\tilde N'}(Q))$-modules which lifts the complex $\mathrm{Br}_{\Delta Q}(\mathcal C')$ of $\ell$-permutation modules from $k$ to $\Lambda$. The following lemma strengthens Lemma \ref{ext loc}. For simplicity we denote

$$M_Q':=\Ind_{\mathrm{C}_G(Q) \times \mathrm{C}_{N'}(Q)^{\mathrm{opp}} \Delta \mathrm{N}_{\tilde N'}(Q)}^{\mathrm{N}_{G}(Q) \times \mathrm{N}_{N'}(Q)^{\mathrm{opp}} \Delta \mathrm{N}_{\tilde N'}(Q)}H^{d_Q}(\mathcal{C}_Q') \mathrm{br}_Q(e_s^{L'}).$$

\begin{lemma}\label{extending brauer image}
Suppose that $\ell \nmid |H^1(F,\mathrm{Z}(\G))|$. Then the bimodule $M_Q'$ almost extends to $\mathrm{N}_G(Q) \times \mathrm{N}_{N'}(Q)^{\mathrm{opp}} \Delta \mathrm{N}_{\mathcal{N}}(Q)$.
\end{lemma}

\begin{proof}
	We first assume that $\Levi'$ is $1$-split. Then $M \cong \Lambda[\G^F/\U'^F] e_s^{L'}$ is an $\ell$-permutation module and $\mathcal C \cong M$. It follows that $\mathcal C' \cong M'$ and thus $H^{0}(\mathrm{Br}_{\Delta Q}(\mathcal C'))  \cong \mathrm{Br}_{\Delta Q}(M') $. The bimodule $M'$ extends to a $G \times (L')^{\opp} \Delta \mathcal{N}_{\ell'}$-module $\hat{M}$, which is $\mathcal N$-stable. Thus, $\mathrm{Br}_{\Delta Q}(\hat{M})$ is a $k(\C_G(Q) \times (\C_{L'}(Q))^{\opp} \Delta \N_{\mathcal{N}_{\ell'}}(Q))$-module extending $\mathrm{Br}_{\Delta Q}(M')$. It follows that $\mathrm{Br}_{\Delta Q}(\hat{M})$ is an $\ell$-permutation module as well. Thus, there exists a unique $\ell$-permutation module $\hat{M}_Q$ which lifts $\mathrm{Br}_{\Delta Q}(\hat{M})$ to $\Lambda$, see \cite[Corollary 3.11.4]{Benson}. Since this lift is unique and $\mathrm{Br}_{\Delta Q}(\hat{M})$ is $\mathrm{N}_{\mathcal{N}}(Q)$-stable it follows that $\hat{M}_Q$ is $\mathrm{N}_{\mathcal{N}}(Q)$-stable as well. In other words, $H^{0}(\mathcal{C}_Q') \mathrm{br}_Q(e_s^{L'})$ almost extends to $\mathrm{C}_G(Q) \times \mathrm{C}_{N'}(Q)^{\mathrm{opp}} \Delta \mathrm{N}_{\mathcal{N}}(Q)$. The claim follows by applying the induction functor.
	
	Now let us assume that $\Levi'$ is not $1$-split. By Lemma \ref{ext loc}, the bimodule $H^{\mathrm{dim}}_c(\Y_{\mathrm{C}_{\U'}(Q)}^{\mathrm{N}_{\G}(Q)},\Lambda) \mathrm{br}_Q(e_s^{L'})$ almost extends to $\mathrm{N}_{G}(Q) \times \mathrm{N}_{L'}(Q)^{\mathrm{opp}} \Delta \mathrm{N}_{\mathcal{N}}(Q)$. On the other hand, the bimodule $M_Q'$ is an extension of $H^{\mathrm{dim}}_c(\Y_{\mathrm{C}_{\U'}(Q)}^{\mathrm{N}_{\G}(Q)},\Lambda) \mathrm{br}_Q(e_s^{L'})$. The quotient group $\mathcal{N}/\tilde{L}'$ is abelian by Lemma \ref{quotient=2} and Lemma \ref{Frobcyclic2}. Therefore, Lemma \ref{abelian factor} shows that $H^{d_Q}(\mathcal{C}_Q') \mathrm{br}_Q(e_s^{L'})$ almost extends as well to a $\Lambda (\mathrm{N}_{G}(Q) \times \mathrm{N}_{L'}(Q)^{\mathrm{opp}} \Delta \mathrm{N}_{\mathcal{N}}(Q))$-module.
\end{proof}

According to Lemma \ref{extending brauer image} there exists an extension $\hat M_Q$ of $ M_Q'$ to $\Lambda(\mathrm{N}_G(Q) \times \mathrm{N}_{N}(Q)^{\mathrm{opp}} \Delta \mathrm{N}_{\mathcal{N}_{\ell'}}(Q))$.
With this notation the following is quite immediate.

\begin{corollary}\label{typeAlocal}
	Suppose the assumptions and notation as above. Then the bimodule $M_Q' C_Q$ induces a Morita equivalence between $\Lambda \mathrm{N}_{N'}(Q) C_Q$ and $\Lambda \mathrm{N}_{G}(Q) B_Q$ which lifts to the Morita equivalence $\hat{M}_Q$ between $\Lambda \mathrm{N}_{\mathcal{N}_{\ell'}}(Q,C_Q) C_Q$ and $\Lambda \mathrm{N}_{G \mathcal{A}_{\ell'}}(Q,B_Q) B_Q$.
\end{corollary}

\begin{proof}
Note that $M_Q' \cong \Ind_{\mathrm{C}_G(Q) \times \mathrm{C}_{N'}(Q)^{\mathrm{opp}} \Delta \mathrm{N}_{N'}(Q)}^{\mathrm{N}_{G}(Q) \times \mathrm{N}_{N'}(Q)^{\mathrm{opp}} \Delta \mathrm{N}_{\tilde N'}(Q)}(H^{d_Q}(\mathcal{C}_Q')) \mathrm{br}_Q(e_s^{L'})$. From this we conclude that the bimodule $M_Q' C_Q$ induces a Morita equivalence between $\Lambda \mathrm{N}_{N'}(Q) C_Q$ and $\Lambda \mathrm{N}_{G}(Q) B_Q$. The second claim now follows from this and \cite[Theorem 3.4]{Marcus}.
\end{proof}

\section{The first reduction}\label{sec 9}

Let $b$ be a block of $\Lambda G e_s^{G}$, where $s \in (\G^\ast)^{F^\ast}$ is the fixed strictly-quasi isolated element from before. By \cite[Theorem 1.3]{Jordan} there exists a defect group $D$ contained in $N$ and since $\ell \nmid |N/L|$ we have $D \subset L$. In what follows, we let $Q$ be a fixed characteristic subgroup of $D$. For a given character $\chi$ we use $\mathrm{bl}(\chi)$ to denote the $\ell$-block $\chi$ belongs to. For the language of character triples and the definition of the order relation $\geq_b$ on character triples we refer the reader to \cite[Section 1.1]{Jordan2}.

\begin{theorem}\label{12}
	Let $ \chi \in \Irr(G,b)$ and $\chi' \in \Irr(\N_G(Q),B_Q)$ such that the following holds:
	\begin{enumerate}[label=(\roman*)]
		\item We have $(\tilde{G} \mathcal{B})_\chi = \tilde{G}_\chi \mathcal{B}_\chi$ and 
		$\chi$ extends to $(G \mathcal{B})_\chi$.
		\item
		We have $( \N_{\tilde{G}}(Q)  \mathrm{N}_{G \mathcal{B}}(Q) )_{\chi'}= \N_{\tilde{G}}(Q)_{\chi'} \mathrm{N}_{G \mathcal{B}}(Q)_{\chi'}$ and $\chi'$ extends to $ \mathrm{N}_{G \mathcal{B}}(Q)_{\chi'}$.
		\item $(\tilde{G} \mathcal{B})_\chi = G ( \N_{\tilde{G}}(Q)  \mathrm{N}_{G \mathcal{B}}(Q) )_{\chi'}$.
		\item There exists $\tilde{\chi} \in \Irr(\tilde{G} \mid \chi)$ and $\tilde{\chi}' \in \Irr(\N_{\tilde{G}}(Q) \mid \chi')$ such that the following holds:
		\begin{itemize}

			\item  For all $m \in {\mathrm{N}_{G \mathcal{B}}(Q)}_{\chi'}$ there exists $\nu \in \Irr(\tilde{G} /G)$ with $\tilde{\chi}^m= \nu \tilde{\chi}$ and $\tilde{\chi}'^m=\mathrm{Res}^{\tilde{G}}_{\N_{\tilde{G}}(Q)}(\nu) \tilde{\chi}'$.
			\item The characters $\tilde{\chi}$ and $\tilde{\chi}'$ cover the same underlying central character of $\mathrm{Z}(\tilde{G})$.
		\end{itemize}
		\item The Clifford correspondents $\tilde{\chi}_0 \in \Irr(\tilde{G}_\chi \mid \chi)$ and $\tilde{\chi}'_0 \in \Irr(\N_{\tilde{G}}(Q)_{\chi'} \mid \chi')$ of $\tilde{\chi}$ and $\tilde{\chi}'$ respectively satisfy $\mathrm{bl}(\tilde{\chi}_0)= \mathrm{bl}(\tilde{\chi}_0')^{\tilde{G}_\chi}$.
	\end{enumerate}
Let $\mathrm{Z}:=  \mathrm{Ker}(\chi) \cap \mathrm{Z}(G) $.
Then
$$(( \tilde{G} \mathcal{B})_\chi /Z, G/Z, \overline{\chi}) \geq_b ((\N_{\tilde{G}}(Q) \mathrm{N}_{G \mathcal{B}}(Q))_{\chi'} / Z,\N_G(Q)/Z, \overline{\chi'}),$$
where $\overline{\chi}$ and $\overline{\chi'}$ are the characters which inflate to $\chi$, respectively $\chi'$.
\end{theorem}

\begin{proof}
	This is a consequence of \cite[Theorem 2.1]{Jordan2} and \cite[Lemma 2.2]{Jordan2}. 
\end{proof}

Note that all conditions in Theorem \ref{12} except condition (v) only depend on the character theory of $G$ and $\tilde{G}$ (together with its associated groups).

Let $M'$ be the bimodule constructed in Corollary \ref{typeA} which induces a Morita equivalence between $\Lambda G e_s^{G}$ and $\Lambda N' e_s^{L'}$. We let $c$ be the block of $\Lambda N' e_s^{L'}$ corresponding to $b$ under this equivalence. Recall the group $\mathcal{A}$ from Definition \ref{defA} and that $\mathcal{N}=\mathrm{N}_{\tilde{G} \mathcal{A}}(\Levi',e_s^{L'})$. We denote by $\mathcal{N}_b$ the stabilizer of the block $b$ in $\mathcal{N}$.

With this notation we are now ready to state the following theorem which can be seen as an analog of \cite[Theorem 2.12]{Jordan2}.

\begin{theorem}\label{reduction}
	Assume that $\ell \nmid 2 |H^1(F,\mathrm{Z}(\G))|$ and suppose that the following hold.
	\begin{enumerate}[label=(\roman*)]

		\item There exists an $\Irr( \tilde{N}' / N') \rtimes \mathrm{N}_{\mathcal{N}}(Q)$-equivariant bijection $\tilde{\varphi}: \Irr(\tilde{N'} \mid \Irr_0(c) ) \to \Irr(\N_{\tilde{N}'}(Q) \mid \Irr_0(C_Q))$ such that it maps characters covering the character $\nu \in \Irr(\mathrm{Z}(\tilde{G}))$ to a character covering $\nu$.
		\item There exists an $\mathrm{N}_{ \mathcal{N}}(Q,C_Q)$-equivariant bijection 
		$\varphi:  \Irr_0(N' , c) \to \Irr_0(\N_{N'}(Q) , C_Q)$ which satisfies the following two conditions:
		\begin{itemize}
			\item 
			If $\chi \in \Irr_0(N',c)$ extends to a subgroup $H$ of $\mathcal{N}_b$ then $\varphi(\chi)$ extends to $\mathrm{N}_{H}(Q)$.
			\item $\tilde{\varphi}(\Irr(\tilde{N}' \mid \chi))= \Irr(\N_{\tilde{N}'}(Q) \mid \varphi(\chi))$ for all $\chi \in \Irr_0(c)$.
		\end{itemize}
		\item For every $\theta \in \Irr_0(c)$ and $\tilde{\theta} \in \Irr(\tilde{N}' \mid \theta)$ the following holds: If $\theta_0 \in \Irr(\tilde{N}'_\theta \mid \theta)$ is the Clifford correspondent of $\tilde{\theta} \in \Irr(\tilde{N'})$ then $\mathrm{bl}(\theta_0)=\mathrm{bl}(\theta_0')^{\tilde{N}'_\theta}$, where $\theta'_0 \in \Irr(\N_{\tilde{N}'}(Q)_{\varphi(\theta)} \mid \varphi(\theta))$ is the Clifford correspondent of $\tilde{\varphi}(\tilde{\theta})$.
		
	\end{enumerate}
	Then the block $b$ is AM-good.
\end{theorem}

\begin{proof}
According to Corollary \ref{typeA} the bimodule $M'$ induces an $\mathcal{N}$-equivariant bijection
	$$R: \Irr_0(N',e_s^{L'}) \to \Irr_0(G,e_s^G).$$
By Corollary \ref{typeAlocal} the bimodule $M_Q'$ induces an $\mathrm{N}_{\mathcal{N}}(Q,C_Q)$-equivariant bijection
 $$R_Q: \Irr_0(\mathrm{N}_{N'}(Q),C_Q) \to \Irr_0(\mathrm{N}_G(Q),B_Q).$$
	We define $\Psi:=R_Q \circ  \varphi \circ R^{-1}:\Irr_0(G,b) \to \Irr_0(\mathrm{N}_G(Q), B_Q)$, which is by construction $\mathrm{N}_{\tilde{G} \mathcal{A}}(Q,B_Q)$-equivariant.
	
	As in the proof of \cite[Theorem 2.12]{Jordan2} we can
	assume that the character $\chi$ satisfies condition (i) of Theorem \ref{12}. We denote $\chi' := \Psi(\chi)$ and show that the characters $\chi$ and $\chi'$ satisfy the conditions of Theorem \ref{12}.
	
	Since the bijection $\Psi:\Irr_0(G,b) \to \Irr_0(\N_G(Q),B_Q)$ is $\mathrm{N}_{\tilde{G} \mathcal{A}}(Q,B_Q)$-equivariant we deduce that condition (iii) in Theorem \ref{12} is satisfied and we have 
	$$(  \mathrm{N}_{\tilde G \mathcal{B}}(Q))_{\chi'}= \N_{\tilde{G}}(Q)_{\chi'} \mathrm{N}_{G \mathcal{B}}(Q)_{\chi'}.$$
	The following lemma finishes the verification of condition (ii) in Theorem \ref{12}
	
	\begin{lemma}
		The character $\chi'$ extends to its inertia group in $\N_{G \mathcal{B}}(Q)$.
	\end{lemma}

\begin{proof}
We have $\N_{G \mathcal{B}}(Q)_{\chi'}/\N_{G}(Q) \cong \mathcal{B}_\chi$. Since the Sylow $r$-subgroups of $\mathcal{B}_\chi$ for $r \neq 2$ are cyclic it suffices to show that $\chi'$ extends to $\N_{G \mathcal{B}_2}(Q)$, where $\mathcal{B}_2$ is the Sylow $2$-subgroup of $\mathcal{B}_\chi$.

We may assume that $\mathcal{B}_2=\langle \phi^i, \gamma \rangle$. There exists some $\tilde{g} \in \tilde{L}'$ such that $F_0^i =\mathrm{ad}(\tilde{g}) \phi^i$. Recall that the character $\chi$ extends to $G \mathcal{B}_2$. It follows that $\chi$ has a $\phi^i$-stable extension to $G \langle \gamma \rangle$. Consequently, this extension is $\tilde{g}^{-1} F_0^i$-stable. From this it follows that $\chi$ extends to $G \langle \tilde{g}^{-1} F_0^i \rangle$. Thus, by Corollary \ref{typeA} the character $R^{-1}(\chi)$ extends to $N' \langle \tilde{g}^{-1} F_0^i,\gamma \rangle$. By assumption (ii), the character $\varphi(R^{-1}(\chi))$ therefore extends to its inertia group in $\N_{G \langle \tilde{g}^{-1} F_0^i,\gamma \rangle}(Q)$. Hence, by Corollary \ref{typeAlocal} the character $\chi'$ has an extension to $\N_{G \langle \gamma \rangle}(Q)_{\chi'}$ which is $\N_{G \langle \phi,\gamma \rangle}(Q)_{\chi'}$-stable. It follows that $\chi'$ extends to $\N_{G\mathcal{B}_{2}}(Q)=\langle \phi^i, \gamma \rangle$.
\end{proof}


Fix a character $\tilde{\chi} \in \Irr(\tilde{G} \mid \chi)$. We define
	$\tilde{\chi}':=\tilde{R}_Q\circ \tilde{\varphi} \circ \tilde{R}^{-1} (\tilde{\chi})$, where $\tilde{R}$ and $\tilde{R}_Q$ are defined as in \ref{character bijection}. From \cite[Lemma 2.9]{Jordan2} and assumption (i) it follows that the first part of condition (iv) in Theorem \ref{12} is satisfied. Moreover, \cite[Lemma 2.10]{Jordan} and assumption (i) imply that the second part of condition (iv) in Theorem \ref{12} is satisfied.

We now verify condition (v) in Theorem \ref{12}. Let $\tilde{\chi}_0 \in \Irr(\tilde{G}_\chi \mid \chi)$ be the Clifford correspondent of $\tilde{\chi}$. Moreover, let $\tilde{\chi}_0' \in \Irr( \N_{\tilde{G}}(Q)_{\chi'} \mid \chi')$ be the Clifford correspondent of $\tilde{\chi}'$. Recall that there exists a complex $\mathcal C'$ of $\Lambda (G \times (N')^{\mathrm{opp}} \Delta \tilde{N})$-modules such that $H^d(\mathcal C') \cong M'$ and $\mathcal C'$ induces a splendid Rickard equivalence between $\Lambda G e_s^{G}$ and $\Lambda N' e_s^{L'}$, see Lemma \ref{diagonal action}.

\begin{lemma}\label{intermediatesubgroup}
	The characters $\tilde{\chi}_0 \in \Irr(\tilde{G}_\chi \mid \chi)$ and $\tilde{\chi}_0' \in \Irr(\N_{\tilde{G}}(Q)_{\chi'} \mid \chi')$ satisfy $\mathrm{bl}(\tilde{\chi}_0')^{\tilde{G}_\chi}=\mathrm{bl}(\tilde{\chi}_0)$.
\end{lemma}
\begin{proof}
	Define $J:=\tilde{G}_{\chi}$ and let $J_0$ be the subgroup of $\tilde{N}'/N'$ corresponding to $J$ under the natural isomorphism $\tilde{N}' / N' \cong \tilde{G} / G$.
	
Consider $\mathcal C:= G \Gamma_c(\Y_{\U'}^\G,\Lambda) e_s^{L'}$ as complex of $\Lambda( G \times (L')^{\opp} \Delta (\tilde{L}'))$-modules and define $\tilde{\mathcal{ C}}:=\Ind_{ G \times (L')^{\opp} \Delta (\tilde{L}')}^{ \tilde{G} \times (\tilde{L}')^{\opp}}(\mathcal{C})$. We have $\tilde{\mathcal{C}} \cong  G \Gamma_c(\Y_{\U'}^{\tilde{\G}},\Lambda) e_s^{L'}$ by \cite[Proposition 1.1]{Godement}.

As in \ref{character bijection} we let $e \in \mathrm{Z}(\Lambda \tilde L')$ be the central idempotent such that $\sum_{n \in N'/L'} {}^n e=e_s^{L'}$. It follows similar to arguments given in the proof of \cite[Theorem 7.5]{Dat} that the complex $\tilde{\mathcal C}':=\tilde{\mathcal C} e \otimes_{\Lambda \tilde{L}'} \Lambda \tilde{N}'$ induces a splendid Rickard equivalence between $\Lambda \tilde{N}' e_s^{L'}$ and $\Lambda \tilde{G} e_s^{G}$. Moreover, we have $\Ind_{G \times (L')^{\opp} \Delta (\tilde{L}')}^{\tilde{G} \times (\tilde{N}')^{\opp}}(\mathcal C') \cong \tilde{\mathcal C}'$. The cohomology of $\mathcal C'$ is concentrated in degree $d:= \mathrm{dim}(\Y_\U^\G)$ and $H^d(\tilde{\mathcal C}_0') \cong  \mathrm{Ind}_{( G \times (L')^{\mathrm{opp}}) \Delta J_0}^{J \times J_0^{\mathrm{opp}}} H^d(\mathcal C')$.  By \cite[Theorem 3.4]{Marcus} the bimodule $H^d(\tilde{\mathcal C}_0')$ induces a Morita equivalence between $\mathcal{O} J_0 c$ and $\mathcal{O} J b$. We denote by 
$$R_0:\Irr(J_0 , c) \to \Irr(J , b)$$
the associated bijection between irreducible characters. Using \cite[Lemma 1.9]{Jordan} we see that the complex $\tilde{\mathcal C}_0':=\mathrm{Ind}_{( G \times N^{\mathrm{opp}}) \Delta(J_0)}^{J \times J_0^{\mathrm{opp}}} (\mathcal C') c$ induces a splendid Rickard equivalence between $\mathcal{O} J b$ and $\mathcal{O} J_0 c$.  The bimodule
 $$(M'_Q)_0:=\Ind^{\N_{J}(Q) \times (\N_{J_0}(Q))^{\opp} }_{\N_{G}(Q) \times (\N_{L}(Q))^{\opp} \Delta \N_{J_0}(Q) } (\Res_{\N_{G}(Q) \times (\N_{L'}(Q))^{\opp} \Delta \N_{J_0}(Q) }^{\N_{G}(Q) \times (\N_{L}(Q))^{\opp} \Delta \N_{ \tilde{N}'}(Q)} (M_Q'))$$ induces a Morita equivalence between $\mathcal{O} \mathrm{N}_{J_0}(Q) C_Q$ and $\mathcal{O} \mathrm{N}_{J}(Q)B_Q$. We denote the associated character bijection by
	$$(R_0)_Q:\Irr(\mathrm{N}_{J_0}(Q) , C_Q) \to \Irr(\mathrm{N}_J(Q), B_Q).$$

%
%
%

	By construction, $\tilde{\chi}' \in \Irr( \mathrm{N}_{\tilde{G}}(Q) \mid \chi')$. Let $\theta:=R^{-1}(\chi)$ and $\tilde{\theta}:=\tilde{R}^{-1}(\tilde{\chi})$. 
We obtain that $(R_0)^{-1}_Q(\tilde{\chi}_0')\in \Irr(\mathrm{N}_{J_0}(Q) \mid \varphi(\theta))$ is the Clifford correspondent of $\tilde{\varphi}(\tilde{\theta})$. Consequently, we have 
	$$\mathrm{bl}((R_0)_Q^{-1}(\tilde{\chi'_0}))^{J_0}=\mathrm{bl}(R^{-1}_0(\tilde{\chi}_0)) $$
	by assumption (iii).
By the definition of $M_Q'$ after Lemma \ref{extending brauer image} we have $$M_Q' \otimes_\Lambda k \cong \Ind_{\mathrm{C}_G(Q) \times \mathrm{C}_{N}(Q)^{\mathrm{opp}} \Delta \mathrm{N}_{N'}(Q)}^{\mathrm{N}_{G}(Q) \times \mathrm{N}_{N'}(Q)^{\mathrm{opp}} \Delta \mathrm{N}_{\tilde N'}(Q)} H^{d_Q}(\mathrm{Br}_{\Delta Q}(C')) \mathrm{br}_Q(e_s^{L'})$$
and therefore,
$$(M_Q')_0 \otimes_\Lambda k \cong \Ind_{\mathrm{C}_G(Q) \times \mathrm{C}_{N'}(Q)^{\mathrm{opp}} \Delta \mathrm{N}_{J_{0}}(Q)}^{\mathrm{N}_{G}(Q) \times \mathrm{N}_{N'}(Q)^{\mathrm{opp}} \Delta \mathrm{N}_{\tilde J_{0}}(Q)} H^{d_Q}(\mathrm{Br}_{\Delta Q}(\mathcal C_0')) \mathrm{br}_Q(e_s^{L'}).$$
Therefore, by \cite[Remark 1.21]{Jordan} the equality $\mathrm{bl}((R_0)_Q^{-1}(\tilde{\chi'_0}))^{J_0}=\mathrm{bl}(R^{-1}_0(\tilde{\chi}_0))$ implies that $\mathrm{bl}(\tilde{\chi}'_0)^J=\mathrm{bl}(\tilde{\chi}_0)$.
\end{proof}

Theorem \ref{12} applies and we obtain
	$$(( \tilde{G} \mathcal{B})_\chi /Z, G/Z, \overline{\chi}) \geq_b ( \mathrm{N}_{\tilde G \mathcal{B}}(Q)_{\chi'} / Z,\mathrm{N}_G(Q), \overline{\chi'}),$$
where $\mathrm{Z}:= \mathrm{Z}(G) \cap \mathrm{Ker}(\chi)$. 
Using \cite[Theorem 1.10]{Jordan} we deduce that the bijection $\Psi:\Irr_0(G,b) \to \Irr_0(\N_G(Q),B_Q)$ is a strong iAM-bijection in the sense of \cite[Definition 1.9]{Jordan2}, which shows that $b$ is AM-good, see \cite[Definition 1.11]{Jordan2}.
\end{proof}

\begin{remark}\label{reductionBC}
	The statement of the previous theorem remains true for the groups occurring in Theorem \ref{equiv} if we take $\Levi'^\ast$ to be the minimal Levi subgroup containing $\C^\circ_{\G^\ast}(s)$ and replace $\mathcal{A}$ by the group $\langle F_0\rangle$. To see this, just replace Corollary \ref{typeA} and Corollary \ref{typeAlocal} by Theorem \ref{equiv} (and its local version which can be proved as in Corollary \ref{typeAlocal}).
\end{remark}

\begin{remark}\label{relative}
In the proof of Theorem \ref{reduction} we have shown that there exist a strong iAM-bijection $\Psi: \Irr_0(G,b) \to \Irr_0(\N_G(Q),B_Q)$ for the characteristic subgroup $Q$ of a defect group $D$ of the block $b$. In Section \ref{sec 11} it will be important to keep track of the subgroup $Q$. If such a strong iAM-bijection exists then we will say that the block $b$ is AM-good relative to the subgroup $Q$.
\end{remark}

\section{Results on $\ell$-blocks when $\ell \mid (q-\varepsilon)$}


\subsection{Action of diagonal automorphisms on $2$-blocks}

Let $Y$ be a finite group and $X$ be a normal subgroup of $Y$. For an $\ell$-block $e$ of $X$ we denote by $Y[e]$ the \textit{Dade ramification group of $e$ in $Y$}, see for instance \cite{Murai}.

\begin{lemma}\label{2block}
	Let $\G$ be a simple, simply connected group of type $A$ and let $b$ be a (non-necessarily quasi-isolated) $\ell$-block of $G$. If $\ell \mid (q-\varepsilon)$ then block induction yields a bijective map $\mathrm{Bl}(\tilde{G}[b] \mid b) \to \mathrm{Bl}(\tilde{G} \mid b)$.
\end{lemma}

\begin{proof}
	Suppose that $b$ is a block of $\Lambda G e_s^{G}$ for a semisimple element $s \in (\G^\ast)^{F^\ast}$ of $\ell'$-order. Consider a regular embedding $\iota: \G \hookrightarrow \Gtilde$. Let $\tilde{\chi} \in \mathcal{E}(\tilde{G}, \tilde{s})$ be the unique semisimple character in its Lusztig series, see \cite[Theorem 15.10]{Bonnafe2}. We fix a constituent $\chi \in \Irr(G\mid \tilde{\chi})$ which lies in the block $b$. We have $\Irr(\tilde{G} \mid \chi)=\{ \hat{z} \otimes \tilde{\chi} \mid z \in \Z(\tilde{G}) \}$ and such a character $\hat{z} \otimes \tilde{\chi}$ lies in the Lusztig series $\mathcal{E}(\tilde{G},\tilde{s} z)$. Note that $\mathcal{E}_{\ell}(\tilde{G},\tilde{s})=\Irr(\tilde{G},\tilde{b})$, where $\tilde{b}$ is the block of $\tilde{\chi}$. Hence, $\hat{z} \otimes \tilde{\chi}$ lies in $\tilde{b}$ if and only if the $\ell'$-part $(\tilde{s} z)_{\ell'}$ of $\tilde{s} z$ is $\tilde{\G}^\ast$-conjugate to $\tilde{s}$. This is equivalent to $\tilde{s} z_{\ell'}$ being $\tilde{\G}^\ast$-conjugate to $\tilde{s}$. In this case we obtain that $\tilde\chi=\hat{z}_{\ell'} \otimes \tilde\chi$ since $\tilde{\chi}$ is the unique semisimple character in $\mathcal{E}(\tilde{G}, \tilde{s})$. Thus, $\tilde{b} \otimes \hat{z}=\tilde{b}$ if and only if $\chi=\hat{z}_{\ell'} \otimes \chi$. From this we deduce by Clifford theory that $|\mathrm{Bl}(\tilde{G} \mid b)|=|\tilde{G}_\chi: G|_{\ell'}$.
	
	On the other hand, by \cite[Proposition 3.9]{Murai} we deduce that $|\mathrm{Bl}(\tilde{G}[b] \mid b)|=|\tilde{G}[b]:G|_{\ell'}$. According to \cite[Theorem 3.5]{Murai} and the Fong--Reynolds reduction block induction yields a surjective map $\mathrm{Bl}(\tilde{G}[b] \mid b) \to \mathrm{Bl}(\tilde{G} \mid b)$. This shows that $|\tilde{G}[b]: G|_{\ell'} \geq |\tilde{G}_\chi: G|_{\ell'}$. On the other hand, by \cite[Theorem 4.1]{Murai} we deduce that $|\tilde{G}[b]: G|_{\ell'} \leq |\tilde{G}_\chi: G|_{\ell'}$. This shows the statement.
\end{proof}

As in  Section \ref{sec 9} we let $b$ be an $\ell$-block of $\Lambda G e_s^{G}$, where $s \in (\G^\ast)^{F^\ast}$ is a strictly quasi-isolated element of $\ell'$-order. Recall the Levi subgroup $\Levi'$ of $\G$ defined in Lemma \ref{prime order}.

\begin{lemma}\label{unique}
	There exists a block $c_1$ of $\Lambda L' e_s^{L'}$, unique up to $N'$-conjugation, such that $b H_c^{\mathrm{dim}}(\Y_{\U'}^\G,\Lambda) c_1 \neq 0$.
\end{lemma}

\begin{proof}	
	By Theorem \ref{Boro} there exists a bimodule $M_0$ inducing a Morita equivalence between $\Lambda G e_s^{G}$ and $\Lambda N' e_s^{L'}$. We let $c$ be the block of $\Lambda N' e_s^{L'}$ corresponding to $b$ under the Morita equivalence induced by $M_0$. The $\Lambda(G \times N'^{\opp})$-bimodule $b M_0 c \cong b M_0$ is thus indecomposable. We have $\mathrm{Res}^{G \times N'^{\opp}}_{G \times (L')^{\opp}}(M_0) \cong b H_c^{\mathrm{dim}}(\Y_{\U'}^\G,\Lambda)$. Observe that $b M_0$ is a right $\Lambda N' c$-module. Thus, if $c_1$ is a block of $L'$ such that $b H_c^{\mathrm{dim}}(\Y_{\U'}^\G,\Lambda) c_1 \neq 0$ then $c_1$ has to lie below $c$. This determines $c_1$ up to $N'$-conjugation.
\end{proof}

\begin{definition}\label{Tgroup}
If $\ell \nmid (q-\varepsilon)$ denote $T:=N'$ and $\hat{c}:=c$. Otherwise, we let $T$ be the largest subgroup of $N'$ containing $L'$ such that there exists a unique block $\hat{c}$ covering $c_1$. Additionally, we set $\tilde{T}:=T \tilde{L}'$.
\end{definition}

Let us now assume that $\ell \mid (q-\varepsilon)$. The block $b$ of $\Lambda G e_s^{G}$ is Morita equivalent to a block of $\Lambda N e_1^{L^F}$. This block of $N$ is covered by the principal block $\Lambda L e_1 ^{L}$ of $L$. Thus, any Sylow $\ell$-subgroup $Q$ of $L$ is a defect group of $b$. A Frattini argument shows that $\N_{\tilde{N}'}(Q) \tilde L'=\tilde{N}'$. Thus, the subgroup $Q$ of $L$ satisfies the group-theoretic assumptions of Lemma \ref{ext loc}. We denote by $\hat{C}_Q$ the Brauer correspondent of $\hat{c}$ in $\N_T(Q)$.

\begin{theorem}\label{reduction2}
	Assume that $\ell \mid (q-\varepsilon)$ and suppose that the following holds.
	\begin{enumerate}[label=(\roman*)]

		\item There exists an $\Irr( \tilde{T} / T) \rtimes \mathrm{N}_{\mathcal{N}}(Q)$-equivariant bijection $\tilde{\varphi}: \Irr(\tilde{T} \mid \Irr_0(\hat{c}) ) \to \Irr(\N_{\tilde{T}}(Q) \mid \Irr_0(\hat{C}_Q))$ such that it maps characters covering the character $\nu \in \Irr(\mathrm{Z}(\tilde{G}))$ to a character covering $\nu$.
		\item There exists an $\mathrm{N}_{ \mathcal{N}}(Q,\hat C_Q)$-equivariant bijection 
		$\varphi:  \Irr_0(T , \hat c) \to \Irr_0(\N_{T}(Q) , \hat C_Q)$ which satisfies the following two conditions:
		\begin{itemize}
			\item 
			If $\chi \in \Irr_0(T,\hat c)$ extends to a subgroup $H$ of $\mathcal{N}_b$ then $\varphi(\chi)$ extends to $\mathrm{N}_{H}(Q)$.
			\item $\tilde{\varphi}(\Irr(\tilde T\mid \chi))= \Irr(\N_{\tilde{T}}(Q) \mid \varphi(\chi))$ for all $\chi \in \Irr_0(\hat c)$.
		\end{itemize}
		\item For every $\theta \in \Irr_0(\hat c)$ and $\tilde{\theta} \in \Irr(\tilde{T} \mid \theta)$ the following holds: If $\theta_0 \in \Irr(\tilde{T}'_\theta \mid \theta)$ is the Clifford correspondent of $\tilde{\theta} \in \Irr(\tilde{T})$ then $\mathrm{bl}(\theta_0)=\mathrm{bl}(\theta_0')^{\tilde{T}_\theta}$, where $\theta'_0 \in \Irr(\N_{\tilde{T}}(Q)_{\varphi(\theta)} \mid \varphi(\theta))$ is the Clifford correspondent of $\tilde{\varphi}(\tilde{\theta})$.
		
	\end{enumerate}
	Then the block $b$ is AM-good.
\end{theorem}

\begin{proof}
	According to the proof of Theorem \ref{mainthm} there exists a $K[G \times (L')^{\mathrm{opp}} \Delta \mathcal{N}]$-module extending $H_c^{\mathrm{dim}}(\Y_{\U'}^\G,K) e_s^{L'}$. Denote by $M'$ the restriction of $\hat{M}$ to $G \times (N')^{\mathrm{opp}} \Delta \tilde{N}'$.
	
	We claim that the bimodule $b \Res_{G \times T^{\mathrm{opp}}}^{G \times (N')^{\mathrm{opp}} \Delta \tilde{N}'}(M') \hat{c}$ induces an $\mathcal{N}_b$-equivariant bijection $R:\Irr(T,\hat{c}) \to \Irr(G,b)$. For $T=N'$ we observe that $\hat{c}$ is the unique block of $N'$ covering the block $c_1$ of $L'$. By Lemma \ref{unique} we deduce that $bM' =b M' \hat{c}$. Since $N'/L'$ is cyclic, Lemma \ref{multiplicity free} implies that $b M' \hat{c}$ is multiplicity free as $K(G \times (N')^{\mathrm{opp}})$-module. Now Lemma \ref{unique} together with Theorem \ref{Boro} imply that $|\Irr(G,b)|=|\Irr(N',\hat{c})|$ and thus the bimodule $b M' \hat{c}$ necessarily induces a bijection between these two sets.
	
	Assume therefore now that $T=L'$. Then $c_1$ is necessarily $N'$-stable and there are exactly $|N'/L'|$ different blocks covering $c_1$. Consequently, by \cite[Proposition 3.9]{Murai} and \cite[Theorem 4.1]{Murai} in this case the block $\Lambda N' c$ is isomorphic to $\Lambda L' c_0$ via restriction. We have $b \Res_{G \times (L')^{\mathrm{opp}}}^{G \times (N')^{\mathrm{opp}} \Delta \tilde{N}'}(M') \hat{c} \cong  H^{\mathrm{dim}}(\Y_{\U'}^\G,K ) \hat c$. Hence, Theorem \ref{Boro} show that $H^{\mathrm{dim}}(\Y_{\U'}^\G,\Lambda ) \hat c$ induces a Morita equivalence between $\Lambda L' \hat c$ and $\Lambda G b$. The claim thus also follows in this case.
	
	As explained before the statement of this theorem, the subgroup $Q$ satisfies the group theoretic requirement in Lemma \ref{ext loc}. The proof of said lemma therefore shows that the bimodule $H^{\mathrm{dim}}_c(\Y_{\mathrm{N}_{\U'}(Q)}^{\mathrm{N}_{\G}(Q)}, K) \br_Q(e_s^{L'})$ extends to an $\N_G(Q) \times \N_{N'}(Q)^{\opp} \Delta \N_{\mathcal{N}}(Q)$-module. Let $\hat{M}_Q$ be such an extension and denote by $M_Q'$ its restriction to $\N_G(Q) \times \N_{N'}(Q)^{\opp} \Delta \N_{\tilde{N}'}(Q)$.
	
	Arguing as in the global case we obtain that the bimodule $B_Q \Res_{\N_G(Q) \times \N_T(Q)^{\opp}}^{\N_G(Q) \times \N_{N'}(Q)^{\opp} \Delta \N_{\tilde{N}'}(Q)}(M_Q') \hat{C}_Q$ induces a bijection
	$R_Q: \Irr_0(\N_{T}(Q), \hat{C}_Q) \to \Irr_0(\N_G(Q), B_Q)$.
	As in Theorem \ref{reduction} we define $\Psi:=R_Q \circ  \varphi \circ R^{-1}:\Irr_0(G,b) \to \Irr_0(\mathrm{N}_G(Q), B_Q)$, which is by construction $\mathrm{N}_{\tilde{G} \mathcal{A}}(Q,B_Q)$-equivariant.  If $T= L'$ then the arguments in the proof of \cite[Theorem 3.12]{Jordan2} show the result. We can therefore from now on assume that $T=N'$.
	
	We fix a character $\chi \in \Irr_0(G, b)$ and denote $\chi' := \Psi(\chi)$. As in Theorem \ref{reduction} we want to show that the characters $\chi$ and $\chi'$ satisfy the conditions in Theorem \ref{12}. Moreover, as in the proof of Theorem \ref{reduction} we fix a character $\tilde{\chi} \in \Irr(\tilde{G} \mid \chi)$ and define $\tilde{\chi}':=\tilde{R}_Q\circ \tilde{\varphi} \circ \tilde{R}^{-1} (\tilde{\chi}).$ One easily sees that conditions (i)-(iv) of the theorem can now be verified exactly in the same way as in Theorem \ref{reduction}. However, to show condition (v) we crucially used Lemma \ref{extending brauer image}. Nevertheless according to Lemma \ref{character theoretic BDR} we see that $\mathrm{bl}(\tilde{\chi}')^{\tilde{G}}=\mathrm{bl}(\tilde{\chi})$.
	According to Lemma \ref{2block} the Clifford correspondents $\tilde{\chi}_0 \in \Irr(\tilde{G}_\chi \mid \chi)$ and $\tilde{\chi}'_0 \in \Irr(\N_{\tilde{G}}(Q)_{\chi'} \mid \chi')$ of $\tilde{\chi}$ and $\tilde{\chi}'$ respectively therefore satisfy $\mathrm{bl}(\tilde{\chi}_0)= \mathrm{bl}(\tilde{\chi}_0')^{\tilde{G}_\chi}$. In other words also condition (v) is satisfied.
\end{proof}

%
%
%


\section{Results on defect groups of blocks of groups of Lie type}\label{sec 10}

\subsection{Centralizers of defect groups}\label{centralizer}

Following the terminology in \cite[Section 3.4]{Cabanesgroup} we say that an $\ell$-group $D$ is \textit{Cabanes} if it has a unique maximal abelian normal subgroup $Q$. Moreover, in this case we say that $Q$ is the \textit{Cabanes subgroup} of $D$. We recall \cite[Lemma 4.16]{Marc}:

\begin{theorem}
Let $\mathbf{H}$ be a connected reductive group defined over $\overline{\mathbb{F}}_p$ such that $\ell \geq 5$ and $\ell \geq 7$ if $\mathbf{H}$ has a component of type $E_8$. Then the defect group of any $\ell$-block of $\mathbf{H}^F$ is Cabanes.
\end{theorem}
%


 We keep the notation from the previous section. In particular, $\G$ is simple, simply connected of type $A$ and $b$ denotes a strictly quasi-isolated block of $\Lambda G e_s^{G}$ and we let $\Levi'$ be the proper Levi subgroup of $\G$ constructed before. Moreover, $c$ denotes the block of $\Lambda N' e_s^{L'}$ which corresponds to $b$ under the Morita equivalence given by the bimodule $M'$ defined after the proof of Theorem \ref{mainthm}.
 
 We fix a defect group $D$ of the block $c$. Moreover if $\ell \geq 5$ we define $Q$ to be the Cabanes subgroup of $D$ and if $\ell <5$ we let $Q:=D$.
Let $\tilde{b}$ be a block of $\tilde{G}$ covering $b$ with defect group $\tilde{D}$ satisfying $\tilde{D} \cap G=D$. If $\ell \geq 5$ let $\tilde{Q}$ be the Cabanes subgroup of $\tilde{D}$ otherwise we define $\tilde{Q}:=\tilde{D}$. 


\begin{lemma}\label{linear prime}
	Suppose that $\ell \neq 2$ and that $D$ is abelian if $\ell=2$.
	With the notation as above we have $\C_{G}(Q)=\C_{G}(\tilde Q)$ and $\N_{G}(Q)=\N_{G}(\tilde Q)$.
\end{lemma}

\begin{proof}
	We may assume that $\ell \mid (n+1,q-\varepsilon)=|\Gtilde^F: \Z(\tilde \G^F)\G^F|$ since otherwise we have $\tilde{Q}=Q \mathrm{Z}(\Gtilde^F)_\ell$ and the statement follows. 
	
	If $s$ has order $n+1$ then we have $\ell \nmid (n+1)$. As explained at the beginning of the proof this implies the statement in this case. Hence, by \cite[Lemma 5.2]{twoblocks} (see also Lemma \ref{quasi2} below) the statement of the lemma holds for $\ell=2$. We can therefore assume that $\ell \geq 5$.
	
		Suppose first that $s=1$.
 We consider again the Frobenius endomorphism $F'$ introduced in \ref{5.2} instead of $F$. 
		Let $\mathbf{S}$ be the diagonal torus in $\G$. The Cabanes subgroup of a Sylow $\ell$-subgroup of $\G^F$ is given by $\mathbf{S}_\ell^{F'}$. Since $\ell \neq 2$ we have $\C_{\G}(\mathbf{S}_\ell^{F'})=\mathbf{S}$, see \cite[Proposition 22.6]{MarcBook}. From this the claim of the lemma follows easily in this case.


Let $\Levi$ be a Levi subgroup of $\G$ dual to the Levi subgroup $\mathrm{C}^\circ_{\G^\ast}(s)$ of $\G^\ast$. Note that $\Lambda \Nb^F e_1^{\Levi^{F'}}$ is Morita equivalent to $\Lambda \G^F e_s^{\G^{F'}}$, see \cite[Example 7.10]{Dat}. Any block of $\Lambda \Nb^{F'} e_1^{\Levi^{F'}}$ has the same defect group as a block of $\Levi^{F'}$ which is covered by it. It is well known, see \cite[Theorem 13]{CabEng}, that since $\ell \mid (q-\varepsilon)$ the algebra $\Lambda \Levi^{F'} e_1^{\Levi^{F'}}$ consists of one block only, the principal block of $\Levi^{F'}$.

Therefore, $Q$ is the Cabanes subgroup of a Sylow $\ell$-subgroup of $\Levi^{F'}$. By Corollary \ref{1split} we have $\tilde{\Levi} \cong(\mathrm{GL}_e)^d\times \dots \times (\mathrm{GL}_e)^d$ with the Frobenius endomorphism $F'$ transitively permuting the $d$ copies of each $(\mathrm{GL}_e)^d$ such that
	$\tilde{\Levi}^{F'} \cong \mathrm{GL}_e( (\varepsilon q)^d) \times \dots \times \mathrm{GL}_e((\varepsilon q)^d)$
with $e > 1$. Denote $Q_0:=Q \cap [\Levi,\Levi]^{F'}$. We conclude $[\Levi,\Levi]^{F'} \cong \mathrm{SL}_e((\varepsilon q)^d) \times \dots \times \mathrm{SL}_e((\varepsilon q)^d)$. Here, $d$ divides $|A(s)|$ which is of order prime to $\ell$, see Corollary \ref{1split}. Let $\mathbf{S}$ be the diagonal torus in $\mathrm{SL}_e(\overline{\mathbb{F}}_q)$. The Cabanes subgroup of the Sylow $\ell$-subgroup of $\mathrm{SL}_e((\varepsilon q)^d)$ is given by $\mathbf{S}_\ell^{F'^d}$. Since $\ell \nmid d$ and $\ell \mid q-\varepsilon$ it follows that $((\varepsilon q)^d-1)_\ell=(\varepsilon q-1)_\ell$. We deduce that $|\mathbf{S}_\ell^{F'}|=|\mathbf{S}_\ell^{F'^d}|$. Denote by $\Delta_d:\GL_e \to (\GL_e)^d$ the $d$-fold diagonal embedding. By the proof of \cite[Lemma 4.16]{Marc} we can therefore assume that
	$$Q_0= \Delta_d(\mathbf{S}_\ell^{F'}) \times \dots \times \Delta_d(\mathbf{S}_\ell^{F'}).$$
	Hence $\C_{\tilde{\G}}(Q_0)\subseteq \GL_{ed} \times \dots \times \GL_{ed}$. Since $\ell \neq 2$ we have $\C_{\GL_e}(\mathbf{S}_\ell^{F'})=\mathbf{S}$, see \cite[Proposition 22.6]{MarcBook}. Therefore, $\C_{\Gtilde}(Q_0) \cong (\mathrm{GL}_d)^e \times \dots \times (\mathrm{GL}_d)^e$. A similar calculation shows that this coincides with $\C_{\Gtilde}(\tilde Q)$. This implies that $\C_{\G}(\tilde Q)=\C_{\G}(Q)=\C_{\G}(Q_0)$.
	
	Observe that $\mathbf{M}:=\mathrm{C}_{\Gtilde}(Q)$ is a Levi subgroup of $\Gtilde$. It follows that $\mathrm{N}_{\Gtilde}(Q)$ and $\mathrm{N}_{\Gtilde}(\tilde Q)$ are both contained in $\mathrm{N}_{\Gtilde}(\mathbf{M})$. On the other hand, $\mathrm{N}_{\Gtilde}(\mathbf{M})/\mathbf{M} \cong S_{(n+1/d)}$ given by permuting the components of $\mathbf{M}= (\mathrm{GL}_d)^e \times \dots \times (\mathrm{GL}_d)^e$. By the description of $Q_0$ is is clear that these automorphisms stabilize $Q_0$. This shows that $\mathrm{N}_{\Gtilde}(Q)=\mathrm{N}_{\Gtilde}(\tilde Q)=\mathrm{N}_{\Gtilde}(\mathbf{M})$.
\end{proof}

\begin{remark}
Suppose that $D$ is the Sylow $2$-subgroup of $\G$. Then the conclusion of the previous lemma holds unless $\G$ is of type $A_1$ and $q \equiv \pm 3 \, \mathrm{mod} \, 8$, see \cite[Theorem 1]{Kon05}. 
\end{remark}

We denote $L_0:=[\Levi',\Levi']^F$ and we fix a block $c_0$ of $L_0$ below $c$ such that $D_0:=D \cap L_0$ is a defect group of $c_0$. Additionally, we set $Q_0:=Q \cap L_0$.

The induction step in the proof of Theorem \ref{maintheoremA} below requires the following property of Cabanes subgroups.

\begin{lemma}\label{centralizer cabanes}
	Suppose that $\ell \neq 2$ and that $D/ \mathrm{Z}(G)_\ell$ is abelian if $\ell=2$. With the notation as introduced above we have $\C_{[\Levi',\Levi']^F}(Q_0)=\C_{[\Levi',\Levi']^F}(Q)$ and $\N_{[\Levi',\Levi']^F}(Q_0)=\N_{[\Levi',\Levi']^F}(Q)$.
\end{lemma}

\begin{proof}
	Let us first consider the special case, where $\ell=2$ and $D$ is non-abelian but $D/ \mathrm{Z}(G)_\ell$ is abelian. According to the proof of \cite[Lemma 5.2]{twoblocks} and \cite[Proposition 3.4(a)]{twoblocks} we observe that $D \subset [\Levi,\Levi]^F$. Since $\Levi \subset \Levi'$ by construction we have $D \subset [\Levi',\Levi']^F$ and so $D_0=D$. Hence, the statement holds trivially in this case.	We can therefore assume that the assumptions of Lemma \ref{linear prime} are satisfied.
	
	We only prove the first part of the statement since the same arguments apply when we replace the centralizer subgroups everywhere by their corresponding normalizer subgroups.
	Consider the composition $[\Levi',\Levi'] \hookrightarrow \Levi' \hookrightarrow \tilde{\Levi}'$. We let $\tilde{c}$ be a block of $\tilde{N}'$ covering $c$. By \cite[Theorem 9.26]{NavarroBook} exists a defect group $\tilde{D}$ of $\tilde{c}$ such that $\tilde{D} \cap N'=\tilde{D} \cap L'=D$. It follows that the Cabanes subgroup $\tilde Q$ of $\tilde{D}$ satisfies $\tilde{Q} \cap L'=Q$. We have $\C_{[\Levi',\Levi']^F}(\tilde{Q}) \subseteq \C_{[\Levi',\Levi']^F}(Q) \subseteq \C_{[\Levi',\Levi']^F}(Q_0)$ and so it's enough to show that $\C_{[\Levi',\Levi']^F}(\tilde{Q})= \C_{[\Levi',\Levi']^F}(Q_0)$. 
	 
As in the proof of \cite[Proposition 3.7]{Jordan2} it follows that 
	 $[\Levi', \Levi']=\mathbf{H}_1 \times \dots \times \mathbf{H}_r,$ where the $\mathbf{H}_i$ are simple algebraic groups of simply connected type. The action of the Frobenius endomorphism $F$ induces a permutation $\pi$ on the set of simple components of $[\Levi',\Levi']$. We let $\pi=\pi_1 \cdots \pi_ t$ be the decomposition of this permutation into disjoint cycles. For $i=1, \dots,t$ choose $x_i \in \Pi_i$ in the support $\Pi_i$ of the permutation $\pi_i$ and let $n_i=|\Pi_i|$ be the length of the cycle $\pi_i$. We then have $$L_0:=[\Levi', \Levi']^F \cong \mathbf{H}^{F^{n_1}}_{x_1} \times \dots \times \mathbf{H}^{F^{n_t}}_{x_t}.$$

	  Similarly, we can decompose $\tilde{\Levi}'$ as $\tilde{\Levi}'=\tilde{\mathbf{H}}_1 \times \dots \times  \tilde{\mathbf{H}}_r$, where $\tilde{\mathbf{H}}_i \cap [\Levi',\Levi']= \mathbf{H}_i$. Therefore, $\tilde{\Levi}'^F \cong \tilde{\mathbf{H}}^{F^{n_1}}_{x_1} \times \dots \times \tilde{\mathbf{H}}^{F^{n_t}}_{x_t}$.

	Denote $\tilde{H}_i:=\tilde{\mathbf{H}}^{F^{n_i}}_{x_i}$ and $H_i:=\mathbf{H}^{F^{n_i}}_{x_i}$.
	 The block $c_0$ is strictly quasi-isolated and decomposes as a direct product $c_0=c_1 \otimes \dots \otimes c_t$ of blocks which are strictly quasi-isolated in $H_i$, see the proof of \cite[Proposition 3.7]{Jordan2}. It follows that $Q_0=Q_1 \times \dots \times Q_t$. Similarly, let $\tilde{c}'$ be a block of $\tilde{L}'$ below $\tilde{c}$ which covers $c_0$. We have a decomposition $\tilde{c}'=\tilde{c}_1 \otimes \dots \otimes \tilde{c}_t$, where $\tilde{c}_i$ is a block of $\tilde{H}_i$ covering $c_i$ and we obtain a decomposition $\tilde{Q}=\tilde{Q}_1 \times  \dots \times \tilde{Q}_t$ with $\tilde{Q}_i \cap H_i= Q_i$. It is therefore enough to show that $\C_{H_i}(Q_i)=\C_{H_i}(\tilde Q_i)$. This follows now from Lemma \ref{linear prime}.
\end{proof}

\begin{corollary}\label{central}
Keep the assumptions of Lemma \ref{linear prime}. If $D_0 \leq \Z(L_0)$ then we have $D \leq \Z(L')$.
\end{corollary}

\begin{proof}
By Lemma \ref{centralizer cabanes} our assumption implies that $Q \leq \C_{L'}(L_0)$. The proof of \cite[Lemma 6.1]{Bonnafe2} shows that $\C_{\Levi'}(L_0)=\Z(\Levi')$ and so $Q \leq \Z(L')$. This shows the claim for $\ell=2$. For $\ell \geq 5$ observe that $Q$ is the maximal normal abelian subgroup of $D$ and it follows from this that $Q=D$.
\end{proof}

%

\section{Reduction to isolated blocks}\label{sec 11}

We first consider certain blocks of simple groups with non-exceptional Schur multiplier.

\begin{lemma}\label{cyclic defect}
	Let $\ell \geq 5$ and $S$ be a simple group of Lie type $A$, $B$ or $C$ defined over a field of characteristic $\neq \ell$ with exceptional Schur multiplier. Then the Sylow $\ell$-subgroups of the universal covering group of $S$ are cyclic. In particular, the inductive Alperin--McKay condition holds for all $\ell$-blocks of the universal covering group of $S$.
\end{lemma}

\begin{proof}
	An examination of the groups $S$ with exceptional Schur multiplier shows that $|S|_\ell \in \{1, \ell \}$ and $\ell \nmid |M(S)|$, where $M(S)$ is the Schur multiplier of $S$. Therefore the Sylow $\ell$-subgroups of the universal covering group of $S$ are cyclic. Therefore, the inductive Alperin--McKay condition holds for all $\ell$-blocks of $S$ by the work of Koshitani--Späth, see \cite{cyclic}. 
\end{proof}

%
%
%
%

We can now prove our main theorem. The proof is essentially the same as the one of \cite[Theorem 3.12]{Jordan2} using all the new ingredients proved up to here. Recall the notation introduced in Remark \ref{relative}.
%
%
%
%


\begin{theorem}\label{maintheoremA}
 Let $\ell \geq 5$ and assume that
	all isolated $\ell$-blocks of quasi-simple groups of type $A$ defined over a field of characteristic $\neq \ell$ are AM-good relative to the Cabanes subgroup of their defect group.
	Then all $\ell$-blocks of quasi-simple groups of type $A$ are AM-good.
\end{theorem}

\begin{proof}
Note that for $\ell \geq 5$ all blocks of simple groups of type $A$ with exceptional Schur multiplier are AM-good (with respect to the prime $\ell$) by Lemma \ref{cyclic defect}. By \cite[Theorem 3.12]{Jordan2} it suffices to show that the strictly quasi-isolated blocks of $\G^F$, where $\G$ is of type $A$, are AM-good relative to the Cabanes subgroup of their defect group. We show this statement by induction on the rank of $\G$. 

Assume that $b$ is a strictly quasi-isolated block of $\G^F$ which is not isolated. Recall that we have a decomposition $L_0:=H_1 \times \dots \times H_t$ where the finite groups $H_i$ are either quasi-simple or solvable. In the former case our induction hypothesis implies that the blocks $c_i$ are AM-good with respect to the Cabanes subgroup of their defect group. By the proof of \cite[Proposition 3.7]{Jordan2} we therefore obtain an iAM-bijection 
$\varphi_0:  \Irr_0(L_0 , c_0) \to \Irr_0(\mathrm{N}_{L_0}(Q_0) , (C_0)_{Q_0})$. Here, $Q_0$ is the Cabanes subgroup of the defect group $D_0$ of $c_0$ and $(C_0)_{Q_0}= \br_{Q_0}(c_0)$. We again use the following notation as introduced in \cite[Notation 3.9]{Jordan2}:
\begin{notation}\label{not}
\
	\begin{itemize}
		\item Assume that $D_0$ is central in $L_0$. Then we let $D$ be a defect group of $c$ satisfying $D \cap L_0=D_0$. We define $L_1:=L_0 D$ and we let $e$ be the unique block of $L_1$ covering $c_0$. In addition, we set $Q:=D$. 
		\item If $D_0$ is not central in $L_0$ then we set $L_1:=L_0$, $e:=c_0$ and we let $Q:=Q_0$.
	\end{itemize}

\end{notation}


Arguing as in the proof of \cite[Lemma 3.10]{Jordan2} we obtain a bijection
	$$\varphi_0:  \Irr_0(L_1 , e) \to \Irr_0(\mathrm{N}_{L_1}(Q) , E_{Q})$$
which satisfies
	$$(\mathcal{N}_\chi, L_1, \chi) \geq_b (\mathrm{N}_{\mathcal{N}}(Q)_{\varphi_0(\chi)},\mathrm{N}_{L_1}(Q), \varphi_0(\chi)).$$
	for every character $\chi \in \Irr_0(L_1,e)$. We show that $\mathrm{N}_{L_1}(Q)=\mathrm{N}_{L_1}(\hat{Q})$, where $\hat{Q}$ is the Cabanes subgroup of the defect group $D$. Assume first that we are in the case that $D_0$ is central in $L_0$. It follows by Corollary \ref{central} that $D$ is abelian and thus $\hat{Q}=D$ is the Cabanes subgroup of $D$.
	If $D_0$ is not central in $L_0$ then Lemma \ref{centralizer cabanes} shows that $\N_{L_1}(\hat Q)=\N_{L_1}(Q)$.

	 Recall the subgroup $T$ from Definition \ref{Tgroup}. We apply \cite[Proposition 1.12]{Jordan2} 
	 to obtain an $\mathrm{N}_{\mathcal{N}}(\hat{Q},\hat{C}_{\hat{Q}})$-equivariant bijection 
	$\varphi:  \Irr_0(T , \hat{c}) \to \Irr_0( \mathrm{N}_{T}(\hat{Q}) , \hat{C}_{\hat{Q}})$ such that
	$$(\mathcal{N}_\chi, T, \chi) \geq_b (\mathrm{N}_{\mathcal{N}}(\hat{Q})_{\varphi(\chi)}, \mathrm{N}_{T}(\hat{Q}), \varphi(\chi))$$
	holds for every character $\chi \in \Irr_0(T,\hat{c})$.
	
	
	 The proof of \cite[Lemma 3.11]{Jordan2} now shows the existence of a bijection $\tilde{\varphi}: \Irr(\tilde{T} \mid \Irr_0(\hat{c}) ) \to \Irr(\mathrm{N}_{T}(\hat{Q}) \mid \Irr_0( \hat{C}_{\hat{Q}}))$
	such that $\tilde{\varphi}$ together with the bijection $\varphi: \Irr_0(T , \hat{c}) \to \Irr_0( \mathrm{N}_{T}(\hat{Q}) , \hat{C}_{\hat{Q}})$ satisfies assumptions (i)--(iii) of Theorem \ref{reduction} and Theorem \ref{reduction2} respectively.

	We can therefore apply Theorem \ref{reduction} and obtain that the block $b$ is AM-good with respect to the Cabanes subgroup $\hat{Q}$.
\end{proof}

\begin{corollary}\label{corollaryA}
		The inductive Alperin--McKay condition holds for all $\ell$-blocks of quasi-simple groups of type $A$, whenever $\ell \geq 5$ is a non defining prime.
\end{corollary}

\begin{proof}
	By Theorem \ref{maintheoremA} it is enough to show that the isolated (that means unipotent) $\ell$-blocks of type $A$ are AM-good relative to the Cabanes subgroup of their defect group. Let $d$ denote the order of $q$ modulo $\ell$.
	
	Suppose first that $\ell \nmid (q-\varepsilon)$. Consider a unipotent block $b$ of $\G^F$. We fix a block $\tilde{b}$ of $\tilde\G^F$ covering $b$. We observe that $\tilde{b}$ has the same defect group as $b$. There exists a $d$-cuspidal pair $(\mathbf{K},\zeta)$ of $(\tilde \G,F)$ associated to $\tilde b$. By the proof of \cite[Theorem 22.9]{MarcBook} it follows that $Q:=\Z(\mathbf{K})^F_\ell$ is the Cabanes subgroup of a defect group $D$ of $b$ and we have $\mathbf{K}=\C_{\Gtilde}(Q)$. Since $Q$ is characteristic in $\mathbf{K}$ it follows that $\N_{\Gtilde}( \mathbf{K})=\N_{\Gtilde}(Q)$. Note that unipotent blocks satisfy the requirements of \cite[Corollary 6.1]{Brough}. Therefore, \cite[Corollary 6.1]{Brough} shows that the iAM-condition holds for the unipotent block $b$ relative to $\N_G(\mathbf{K} )=\N_G(Q)$.
	
Assume now that $\ell \mid (q-\varepsilon)$. It is well known, see \cite[Example 22.10]{MarcBook} and \cite[Remark 22.11]{MarcBook} that in this case $\Gtilde^F$ and therefore also $\G^F$ has only one unipotent $\ell$-block. This is the principal block of $\G^F$ and therefore has maximal defect group. In particular, this block is AM-good relative to $\N_G(\mathbf{S})$, where $\mathbf{S}$ is a Sylow $\Phi_d$-torus of $\G$, by the main theorem of \cite{CS14}. It is thus sufficient to show that $\N_{\G}(Q)=\N_{\G}(\mathbf{S})$. Since $\ell \mid (q-\varepsilon)$ we know that $\mathbf{S}$ is a maximally split torus of $\G$. Again by the proof of \cite[Theorem 22.9]{MarcBook} it follows that $Q=\mathbf{S}_\ell^F$ is the Cabanes subgroup of a defect group of $b$ and $\C_{\G}(Q)=\mathbf{S}=\C_{\G}(\mathbf{S})$. This implies $\N_{\G}(Q)=\N_{\G}(\mathbf{S})$, which proves the claim.
\end{proof}

%


\begin{theorem}\label{maintheoremBC}
	Let $X$ be one of the symbols $B$ or $C$ and $\ell \geq 5$. Assume that
	all isolated $\ell$-blocks of quasi-simple group of type $X$ and are AM-good relative to the Cabanes subgroup of the defect group;
	Then all $\ell$-blocks of quasi-simple groups of type $X$ are AM-good.
\end{theorem}

\begin{proof}
	Assume first that $n=2$ and $q$ is even. Then $\Z(\G)$ is trival and so the isolated $\ell$-blocks of $\G^F$ are precisely the quasi-isolated $\ell$-blocks. The statement is then a consequence of the main theorem of \cite{Jordan2}. Also we observe that for $\ell \geq 5$ all blocks of simple groups of type $X$ with non-exceptional Schur multiplier are AM-good by Lemma \ref{cyclic defect}. We can therefore assume that we are in none of these exceptional cases.
	
 Fix an $\ell$-block $b$ of $\Lambda \G^F e_s^{\G^F}$ (not necessarily quasi-isolated), where $\G$ is of type $X$ and $s \in (\G^\ast)^{F^\ast}$ is semisimple of $\ell'$-order.
 
 In contrast to the proof of Theorem \ref{maintheoremA} we don't need to argue by induction. We let $\Levi^\ast$ be the minimal Levi subgroup of $\G^\ast$ containing $\mathrm{C}^\circ_{\G^\ast}(s)$ and let $\Levi$ be the Levi subgroup dual to $\Levi^\ast$. We let $\mathcal{A}:= \langle F_0 \rangle$ as in Theorem \ref{equiv}. Let $c$ be the block of $\Lambda \mathbf{N}^F e_s^{\Levi^F}$ corresponding to $b$ under the Bonnafé--Dat--Rouquier equivalence from Theorem \ref{equiv} and let $c_0$ be a block of $L_0:=[\Levi,\Levi]^F$ below $c$. We conclude that the block $c_0$ is then an isolated block of $L_0$. As in the proof of Theorem \ref{maintheoremA} we obtain a decomposition $L_0=H_1 \times \dots \times H_t$ into groups which are either quasi-simple of type $A$ or $X$ or solvable and a corresponding decomposition $c_0=c_1 \otimes \dots \otimes c_t$ of $c_0$ into block $c_i$ which are isolated in $H_i$. Our assumption together with the proof of Corollary \ref{corollaryA} implies that these blocks are AM-good relative to the Cabanes subgroup of their defect group. Following the proof of Theorem \ref{maintheoremA} and using Remark \ref{reductionBC} instead of Theorem \ref{reduction} we deducee that the block $b$ is AM-good relative to the $\ell$-subgroup $Q$, where $Q$ is defined as in Notation \ref{not}. 
\end{proof}

\section{A variant of Späth's reduction theorem}

In her paper \cite{IAM} Späth shows that the Alperin--McKay conjecture holds for every finite group if and only if the inductive Alperin--McKay condition holds for every simple group. Our aim here is to modify her proof in order to get a similar statement involving preferably only blocks with abelian defect. Unfortunately in her proof, it is necessary to consider central extensions of groups and a block might have abelian defect group but a block of a central extension dominating it might not. Therefore, we need to consider blocks whose defect group lies in a slightly larger class of groups.

For the following definition recall that if $G$ is a finite group its upper central series is defined recursively as $Z_0(G):=1$ and $\Z_i(G)$ is the unique subgroup of $G$ containing $\Z_{i-1}(G)$ such that $\Z_i(G)/\Z_{i-1}(G)=\Z(G/\Z_{i-1}(G))$.

\begin{definition}\label{almost abelian}
	We say that a subgroup $D$ of a finite group $G$ is \textit{almost abelian in} $G$ if there exists an $i$ such that $D \Z_i(G)/\Z_i(G)$ is abelian.
	
\end{definition}

Observe that the property of $D$ being almost abelian can depend on the ambient group $G$. Moreover, if one considers the hypercenter $\Z_\infty(G)$ of $G$ (i.e. the union of all $\Z_i(G)$ for $i \geq 0$) then $D$ is almost abelian in $G$ if and only $D \Z_\infty(G)/\Z_\infty(G)$ is abelian.

\begin{lemma}\label{central extension}
Suppose that $D$ is almost abelian in $G$.
\begin{enumerate}
	\item[(a)] Let $H$ be a subgroup of $G$ and $E \leq H \cap D$. Then $E$ is almost abelian in $H$.

	\item[(b)] For $j=1,2$ the groups $D_j$ are almost abelian in $G_j$ if and only if $D_1 \times D_2$ is almost abelian in $G_1 \times G_2$.
	\item[(c)] A subgroup $E$ of $G$ is almost abelian in $G$ if and only if $E \Z(G)/\Z(G)$ is almost abelian in $G/ \Z(G)$.
\end{enumerate}
\end{lemma}

\begin{proof}
Let us first prove part (a). By assumption, there is an $i$ such that $D\Z_i(G)/\Z_i(G)$ is abelian. By induction one easily shows that $\Z_i(G) \cap H \lhd \Z_i(H)$. From this we deduce that $\frac{(D \cap H) \Z_i(H)}{\Z_i(H)}$ is abelian. Therefore, $D\cap H$ is almost abelian in $H$. Consequently, $E \leq D \cap H$ is almost abelian as well.

For part (b) we observe that $\Z_i(G_1 \times G_2)=\Z_i(G_1) \times \Z_i(G_2)$. Thus, $$\frac{(D_1 \times D_2) \Z_i(G_1 \times G_2)}{\Z_i(G_1 \times G_2)} \cong \frac{D_1 \Z_i(G)}{\Z_i(G)} \times \frac{D_2 \Z_i(G)}{\Z_i(G)}$$
and the claim follows from this.

For part (c) one first shows by induction that $\Z_{i-1}(G/ \Z(G)) =\Z_i(G)/Z(G)$ for all $i$. Hence, $D \Z(G)/ \Z(G)$ is almost abelian in $G/ \Z(G)$ if and only if there exists an $i$ such that $$ \frac{\Z_i(G/ \Z(G)) D \Z(G)/ \Z(G)} {\Z_i(G/ \Z(G))} \cong \frac{\Z_{i+1}(G) D}{ \Z_{i+1}(G)}$$ is abelian. From this is follows that $D \Z(G)/ \Z(G)$ is almost abelian in $G/ \Z(G)$ if and only if $D$ is almost abelian in $G$.
\end{proof}


The aim of this section is to prove the following variant of \cite[Theorem C]{IAM}.
We closely follow the proof of \cite[Proposition 2.5]{CS14}.

\begin{proposition}\label{modifiedSpaeth}
	Let $X$ be a finite group and $\ell$ a prime. Assume that for every non-abelian simple subquotient $S$ of $X$ with $\ell \mid |S|$ the following holds: Every $\ell$-block of the universal covering group $H$ of $S$ with almost abelian defect group satisfies the iAM-condition. Then the Alperin--McKay conjecture holds for any $\ell$-block of $X$ with almost abelian defect.
\end{proposition}

For a finite group $X$ let $F^\ast(X)$ be its generalized Fitting subgroup.

\begin{proposition}\label{Fitting}
	Let $X$ be a finite group and $b$ an $\ell$-block of $X$ with almost abelian defect. Suppose that the Alperin--McKay conjecture is true for any $\ell$-block with almost abelian defect group of any group $H$ with $|H: \Z(H)| < |X:\Z(X)|$ and such that $H$ is isomorphic to a subquotient of $X/\Z(X)$. Then one of the following holds:
		\begin{enumerate}
		\item[(i)] The Alperin--McKay conjecture holds for $b$.
		\item[(ii)] For any non-central normal subgroup $K$ of $X$ we have $X=K \N_X(D) F^\ast(X)$.
	\end{enumerate}
\end{proposition}

\begin{proof}
Let $K \lhd X$ be a non-central normal subgroup of $X$. Replacing $K$ by $K \Z(X)$ if necessary we can assume that $\Z(X) \leq K$.
 By assumption the Alperin--McKay conjecture is true for every block with almost abelian defect of any central extension of $X/K$. An analysis of the proof of \cite[Theorem 6]{MuraiAlperin} then shows that the proof of said theorem can be adapted and we obtain $|\Irr_0(B)|=|\Irr_0(b)|$ where $B \in \mathrm{Bl}(K\N_X(D))$ is the unique block with $B^X=b$. If $K \N_X(D)$ is a proper subgroup of $X$ then our assumption implies that the Alperin--McKay conjecture holds for the block $B$ of $K \N_X(D)$. Therefore, $|\Irr_0(B)|=|\Irr_0(B_D)|$, where $B_D$ is the Brauer correspondent of $b$. This would then imply that the Alperin--McKay conjecture holds for $b$. Since the generalized Fitting subgroup $F^\ast(X)$ is such a non-central normal subgroup of $X$ this argument shows in particular that $X=F^\ast(X) \N_{X}(D)$.
\end{proof}

\textbf{Proof of Proposition \ref{modifiedSpaeth}:} The proof of the statement is by induction on $|X:\Z(X)|$. We let $b$ be an $\ell$-block of $X$ with almost abelian defect group $D$. We let $B_D$ be its Brauer correspondent in $\N_X(D)$.
	
	According to Proposition \ref{Fitting} we may assume that $X$ and $b$ satisfy the statement in Proposition \ref{Fitting}(ii). Consequently, any normal $\ell$-subgroup of $X$ is central. As in \cite[Section 6 and 7]{IAM} we distinguish two cases.
	
	Assume that there exists a normal non-central subgroup $K \lhd X$ with $K \leq F^\ast(X)$. If $K \cap D \leq \Z(X)$ then $|\Irr_0(b)|=|\Irr_0(B_D)|$,  according to \cite[Proposition 6.6]{IAM}.
	
	Assume otherwise that the Fitting subgroup is central in $X$, and $E(X)$, the group of components of $X$, is non-central. Inside $E(X)$ we can take a normal subgroup $K \lhd X$ such that $K=[K,K]$ and $K/\Z(K) \cong S^r$ for a non-abelian simple group $S$ and an integer $r \geq 1$. The proof of \cite[Proposition 9(ii)]{MuraiAlperin} shows that the block $b$ covers a unique block $b_0$ of $K$.  Note that $D_0:=D \cap K$ is a defect group of $b_0$ which is almost abelian by Lemma \ref{central extension}(a). Let $\tilde{G}=G^r$ be the univeral covering group of $K/\Z(K) \cong S^r$, where $G$ is the universal covering group of $S$. Let $\pi: \tilde{G} \to S^r$ be the associated quotient map. Let $\tilde{b}_0$ be a block of $\tilde{G}$ dominating $b_0$. Again Lemma \ref{central extension} ensures that $\tilde{b}_0$ has an almost abelian defect group $\tilde{D}$. Following the proof of \cite[Theorem 7.9]{IAM} we obtain a bijection $\tilde{\Omega}:\Irr(\tilde{G},\tilde{b}_0) \to \Irr(\tilde{M},\tilde{B}_0)$. Here, $\tilde M$ is a suitably defined subgroup of $\tilde G$ containing $\N_{\tilde G}(\tilde D)$ and $\tilde{B}_0^{\tilde G}=\tilde{b}_0$. From this the proof of \cite[Theorem 7.9]{IAM} then yields a bijection $\hat{\Omega}:\Irr(K,b_0) \to \Irr(M \N_K(D_0), B_0)$ having the properties of the bijection in the statement of \cite[Theorem 7.9]{IAM}. Here, $M:= \pi(\tilde{M})$ and $B_0$ is the unique block of $M \N_K(D_0)$ with $B_0^K=b_0$. Using the counting argument in the proof of \cite[Theorem C]{IAM} we are then able to deduce that $|\Irr_0(X,b)|=|\Irr_0(M \N_X(D_0),B)|$.

	  Since $|M \N_X(D_0) / \Z(X)| <|X/\Z(X)|$ we can apply the induction hypothesis and we conclude that $|\Irr_0(M \N_X(D_0),B)|=|\Irr_0(\N_X(D),B_D)|$, where $B_D$ is the Brauer correspondent of $b$. \qed

\section{On the Alperin--McKay conjecture for $2$-blocks with abelian defect group}

\subsection{Classification of $2$-blocks with abelian defect group}

We start by recalling the following classification of quasi-isolated $2$-blocks of finite groups of Lie type, see  \cite[Lemma 5.2]{twoblocks}:
\begin{lemma}\label{quasi2}
	Assume that $p$ is odd and that $\G$ is simple and simply-connected. Let $b$ be a quasi-isolated $2$-block of $G$ with semi-simple label $s \in \G^\ast$. 
	\begin{enumerate}
		\item[(a)]

	Suppose that $b$ has abelian defect groups. Then one of the following holds.
	\begin{enumerate}[label=(\roman*)]
		\item $\G$ is of type $A_n$, $n$ is even, and $C^\circ_{\G^\ast}(s)$ is a torus.
		\item  $\G$ is of type $G_2$, $F_4$, $E_6$ or $E_8$, $s=1$ and $b$ is of defect $0$.
	\end{enumerate}
\item[(b)] Suppose that $b$ has non-abelian defect groups, but for some central $2$-subgroup $Z$ of $\G^F$, the image $\overline{b}$ in $G/Z$ has abelian defect group. Then $Z$ is cyclic of order $2$ and one of the following holds.
	\begin{enumerate}[label=(\roman*)]
	\item $\G$ is of type $A_n$, $n \equiv 1 \, \mathrm{mod} \, 4$ and the defect groups of $\overline{b}$ are $C_2 \times C_2$.
	\item  $\G$ is of type $E_7$ and the defect groups of $\overline{b}$ are $C_2 \times C_2$.
\end{enumerate}
\end{enumerate}

\end{lemma}

In the following remark we collect some additional information from the proof of \cite[Lemma 5.2]{twoblocks}.

\begin{remark}\label{unipotent almost abelian}
	\ 
	\begin{enumerate}
		\item[(a)]

	If we assume additionally that the blocks in part (b)(i) are isolated (i.e. unipotent) then the proof of \cite[Lemma 5.2]{twoblocks} shows that $\G^F \cong A_1(q)$ and $q \equiv  \pm 3  \, \mathrm{mod} \, 8$. Moreover, their defect groups are isomorphic to the quaternion group $Q_8$.
	\item[(b)]
	The blocks in (b)(ii) occur only if $4 \mid \mid (q-1)$. These blocks are unipotent and their defect groups are isomorphic to the dihedral group $D_8$. They correspond to lines 3 and 7 of the table on page 354 of \cite{Enguehard}.
\end{enumerate}
\end{remark}

%

%

%
%

Our aim is now to show the iAM-condition for all isolated blocks of positive defect occuring in the classification of Lemma \ref{quasi2}.


\subsection{On a certain $2$-block of $E_7(q)$}

In this subsection we consider the block $b$ of $E_7(q)$ occuring in part (b)(ii) of the classification of Lemma \ref{quasi2}. The author is very grateful to Gunter Malle for pointing out the proof of the following proposition to him.

\begin{proposition}\label{E7}
	Let $b$ be one of the blocks of $G=E_7(q)$ occuring in Lemma \ref{quasi2}(b)(ii). Then we have $|\Irr_0(b)|=4$ and $|\Irr(b)|=5$. Furthermore, the height zero characters of $b$ have $\mathrm{Z}(G)$ in their kernel and are as follows:
	\begin{enumerate}
		\item[(i)] Two unipotent characters $\chi_1,\chi_2$.
		\item[(ii)] Two non-unipotent characters $\chi_3,\chi_4$ which are conjugate under the diagonal automorphism of $G$.
	\end{enumerate}

\end{proposition}

\begin{proof}

	Let $q$ be a prime power such that $4||(q-1)$. Set $G=E_7(q)_{sc}$ and
	$S=G/Z(G)$, the simple group of type $E_7(q)$. We consider the unipotent
	$2$-block $b$ of $G$ parametrized by the 1-cuspidal pair $(E_6,E_6[\theta])$
	(and its Galois conjugate block $(E_6,E_6[\theta^2])$ for which all
	arguments apply similarly).
	
	Firstly, we note that according to the description in \cite[3.2]{Enguehard} the defect group of $b$ is isomorphic to the dihedral group $D_8$. By \cite[Theorem 8.1]{Sambale} we can therefore deduce that $|\Irr_0(b)|=4$ and $|\Irr(b)|=5$. Moreover, the unique block $\overline{b}$ of $S$ dominated by $b$ has defect group $C_2 \times C_2$. Hence, by \cite[Theorem 8.1]{Sambale} we know that $4=|\Irr(\overline{b})|=|\Irr_0(\overline{b})|$ and thus the unique character in $\Irr(b)$ with positive height is non-trivial on $\mathrm{Z}(G)$ and it is unique with this property among the irreducible characters of $b$. We will now describe the character of $\Irr(b)$ in more detail.
	
	According to \cite[Theorem B]{Enguehard}, the ordinary characters in $b$ are
	described as follows: let $t\in G^*=E_7(q)_{\ad}$ be a (semisimple)
	2-element such that $C^\circ_{\G^*}(t)$ has a Levi subgroup of type $E_6$.
	Then $b$ contains those elements of $\mathcal{E}(G,t)$ which under Jordan
	decomposition correspond to characters in the 1-Harish-Chandra series
	of $\C^\circ_{\G^*}(t)$ above $(E_6,E_6[\theta])$, and moreover, these
	are the only characters in $b$.
	
	In our particular case, the only centralizer in $\G^*$ that can possibly
	contain a Levi subgroup of type $E_6$ is, apart from $\G^*$ itself, also
	of type $E_6$. Now the centralizer of $E_6(q)$ in $E_7(q)$ is
	$K:=E_6(q)(q-1)$, and so its center has order $(q-1)$. Given that 
	$4||(q-1)$,
	there are hence only four 2-elements $t$ in $Z(K)$, the identity, an 
	involution
	and two elements of order~4. Clearly the centralizer of the identity is 
	all of
	$G^*$. Now $N(K)=K.2$ (by a calculation in the Weyl group), so for the involution
	$s\in \Z(K)$ we have $\C_{G^*}(s)=K.2$ (that is, the centralizer of $s$ in $\G^\ast$ is disconnected). Thus there are two characters, say $\chi_3$ and $\chi_4$, of $b$ 
	in that
	geometric Lusztig series $\tilde{\mathcal{E}}(G,s)$ fused by the diagonal 
	automorphism of
	$G$. On the other hand, $N(K)$ acts non-trivially on $Z(K)$ (again by a
	calculation in the Weyl group) so the two elements $t$ of order~4 in 
	$Z(K)$ are
	$G^*$-conjugate and moreover $C_{G^*}(t)=K$. So here $\mathcal{E}(G,t)$ contains one
	character $\chi_5$ in $b$, which is left invariant under the diagonal automorphism.
	
	From these calculations it necessarily follows from $|\Irr(b)|=5$ that $\mathcal{E}(G,1)$ contains two characters $\chi_1$ and $\chi_2$ of $b$ (in fact these are the irreducible characters lying in the Harish-Chandra series of $E_6(\theta)$). These characters have $\mathrm{Z}(G)$ in their kernel. The characters $\chi_3$ and $\chi_4$ are $\tilde G$-conjugate and thus have the same underlying character of the center. From this we deduce that $\chi_5$ has to be the unique character of $b$ which is nontrival on $\mathrm{Z}(G)$.  From this our arguments above show that $\Irr_0(G,b)=\{\chi_1,\chi_2,\chi_3,\chi_4\}$.
	
	An entirely similar argument applies when $4||(q+1)$; here all groups
	$E_6(q)$ have to be replaced by ${}^2 E_6(q)$ and all $(q-1)$ by $(q+1)$.
\end{proof}


To prove the iAM-condition for the block $b$ we also need to compute the action of group automorphisms on the height zero characters of its Brauer correspondent. The information given in \cite{Enguehard} does not seem sufficient for this. Instead we will try to obtain all the necessary local information from the invariants of the block $b$.  

The following proposition is a consequence of \cite[Proposition 10.26]{Sambale}.

\begin{proposition}\label{sambale}
	Let $H$ be a group isomorphic $C_2 \times C_2$ or $D_8$ and for $n \geq 1$ consider a group extension
	$$1 \xrightarrow{} H \xrightarrow{} D  \xrightarrow{} C_{2^n} \xrightarrow{} 1$$ 
Then the invariants for every block of a finite group with defect group $D$ are known.
\end{proposition}

\begin{proof}
We must show that the exceptions listed in \cite[Proposition 10.26]{Sambale} do not occur if $H$ is as in the statement of our proposition. The first assumption on the coupling $\omega$ in \cite[Proposition 10.26]{Sambale} can be easily verified by examining the automorphism group structure of $C_2 \times C_2$ and $D_8$ respectively. We are therefore left to show that $D$ is not isomorphic to $C_{2^m} \wr C_2$ for all $m \geq 3$. For $H \cong C_2 \times C_2$ this follows from the proof of \cite[Proposition 10.26]{Sambale}.

Assume therefore that $H \cong D_8$ and $H \lhd (C_{2^m} \times C_{2^m}) \rtimes C_2$. There exists an element $a \in H$ which is not contained in the base group $C_{2^m} \times C_{2^m}$. This element must act on the base group by interchanging both components. Let $g \in C_{2^m}$ be a generator. Then it follows that $a^{(g,1)}=(g,g^{-1}) a$ and so $(g,g^{-1}) \in D_8$. This implies that $m \leq 2$.
\end{proof}

\begin{lemma}\label{spaeth}
	Let $b$ be a block of a finite group $G$ with defect group $D$. Suppose that $B$ is a block of a subgroup $M$ of $G$ containing $\N_G(D)$ with $B^G=b$. Let $A \subseteq \mathrm{Aut}(G)$ be a finite cyclic subgroup stabilizing $M$. Assume that the Alperin--McKay conjecture holds for every block of $G \rtimes A$ and $M \rtimes A$ covering $b$ and $B$ respectively.
	
	\begin{enumerate}
		\item[(a)] If $A$ is a simple cyclic group or an $\ell$-group then the number of $A$-invariant characters in $\Irr_0(b)$ is equal to the number of $A$-invariant characters of $\Irr_0(B)$.
		\item[(b)] We have $\Irr_0(b)=\Irr_0(b)^A$ if and only if $\Irr_0(B)=\Irr_0(B)^A$.
	\end{enumerate}

\end{lemma}

\begin{proof}
This follows from the proof of \cite[Lemma 8.1]{IAM}.
\end{proof}

\begin{proposition}\label{block E7}
	Let $b$ be one of the blocks of $G=E_7(q)$ occuring in Lemma \ref{quasi2}(b)(ii). Then $b$ is AM-good relative to its defect group.
\end{proposition}

\begin{proof}
		Let $\G_{\mathrm{ad}}$ be the adjoint quotient of $\G$. There exist a Frobenius endomorphism $F$ on $\G_{\mathrm{ad}}$ which commutes with the quotient map $\pi: \G \to \G_{\mathrm{ad}}$. Then $\pi$ induces an injective map $\pi: S \to G_{\mathrm{ad}}$, where $G_{\mathrm{ad}}:=\G_{\mathrm{ad}}^F$ and $G_{\mathrm{ad}}$ induces all diagonal automorphisms on $S$. More precisely, we have $\tilde{G}/\mathrm{Z}(\tilde G) \cong G_{\mathrm{ad}}$. 
	
	Let $D$ be a defect group of $b$ and $B$ be its Brauer correspondent in $\N_G(D)$. We first construct an $\mathrm{Aut}(G)_{b,D}$-equivariant bijection $\Irr_0(G,b) \to \Irr_0(\N_G(D),B)$. Let $\delta: G \to G$ be a diagonal automorphism induced by the action of an element of $G_{\ad}$ and let $F_0: G \to G$ be a generator of the group of field automorphisms of $G$ which together generate $\mathrm{Out}(G)$. We first observe that the characters $\chi_1,\chi_2$ are $\langle \delta, F_0 \rangle$-stable by \cite[Theorem 2.5]{MalleUnip}. Moreover, Proposition \ref{E7} together with \cite[Theorem 2.11]{Jordan2} implies that $\langle \delta,F_0 \rangle_{\chi_i}=\langle F_0 \rangle$ for $i=3,4$. In particular, the block $b$ is $\langle \delta, F_0 \rangle$-stable and thus we can assume (by possibly replacing these automorphisms by a $G$-conjugate) that $D$ is $\langle \delta, F_0 \rangle$-stable.
	
	We denote by $c$ either the block $b$ or $B$. Moreover, we let $\overline{c}$ be the unique block of $S:=G/Z(G)$ respectively $\N_G(D)/\Z(G)$ which is dominated by $c$. As in the proof of Proposition \ref{E7} we find that $|\Irr_0(c)|=4$ and $|\Irr(\overline{c})|=|\Irr_0(\overline{c})|=4$. Therefore, the quotient map induces a bijection $\Irr_0(c) \to \Irr(\overline{c})$. Note that the defect group $\overline{D}:=D/\Z(G)$ of the block $c$ has a defect group isomorphic to $C_2 \times C_2$. According to Proposition \ref{sambale} we can therefore apply Lemma \ref{spaeth} to the automorphisms $\delta$ and $F_0$. Using Proposition \ref{E7} we find that every character of $\Irr_0(B)$ is $F_0$-stable. Moreover there exist two characters $\psi_1,\psi_2 \in \Irr_0(B)$ which are $\delta$-stable and the other two characters $\psi_3,\psi_4 \in \Irr_0(B)$ are $\delta$-conjugate. We therefore obtain that the bijection $\Irr_0(b) \to \Irr(B), \chi_i \mapsto \psi_i,$ is $\langle \delta,F_0 \rangle$-equivariant.

	We now show that the characters of $\Irr_0(\overline{b})$ and $\Irr_0(\overline{B})$ extend to their inertia groups in $G_{\ad} \langle F_0 \rangle$ and $\N_{G_{\ad} \langle F_0 \rangle}(\overline{D})$ respectively. Since $\chi_1,\chi_2$ are unipotent characters it follows that they extend to $G_{\ad} \langle F_0\rangle$, see for example \cite[Theorem 2.4]{MalleUnip}. For $i=3,4$ the stabilizer quotient $(G_{\ad} \langle F_0 \rangle)_{\chi_i}/S$ of $\chi_i$ is cyclic and so $\chi_i$ extends to its inertia group in $G_{\ad}\langle F_0 \rangle$. Similarly, for $i=3,4$ the local character $\psi_i$ extends to its inertia group in $\N_{G_{\ad} \langle F_0 \rangle}(\overline{D})$ as well. It is therefore left to show that $\psi_1$ and $\psi_2$ extend to $G_{\ad} \langle F_0 \rangle$ as well. We let $b_2$ be a block of $G_{\ad}\langle F_0 \rangle$ covering $\overline{b}$. Let $D_2$ be a defect group of $b_2$ such that $D_2 \cap S=\overline{D}$. We observe that $D_1:=D_2 \cap G_{\ad}$ is a defect group of the unique block $b_1$ of $G_{\ad}$ covering $\overline b$. Since Brauer's height zero conjecture holds for blocks with defect group of order $8$ (see e.g. \cite[Theorem 13.1]{Sambale} and \cite[Theorem 8.1]{Sambale})it follows that $D_1$ is nonabelian. Since $|\Irr(b_1)|=5$ we must have $D_1 \cong D_8$, see for instance \cite[Theorem 8.1]{Sambale}. From this it follows that $D_2$ is a cyclic extension of $D_8$. Let $B_1$ be the Harris--Knörr correspondent of $b_1$ in $\N_{G_{\ad}}(\overline{D})$. Again, Proposition \ref{sambale} ensures that Lemma \ref{spaeth} is applicable to the automorphism $F_0$. Since every character of $\Irr_0(b_1)$ is $F_0$-stable the same is true for every character of $\Irr_0(B_1)$. Since $\Irr_0(B_1)$ consists of the four extensions of $\psi_1$ and $\psi_2$ to $G_{\ad}$ it follows that $\psi_1$ and $\psi_2$ must necessarily extend to $G_{\ad} \langle F_0 \rangle$.

It is now easy to show that the above information is enough to verifiy the inductive conditions. For this we observe that we can choose extensions $\hat{\chi}_i \in \Irr((G_{\ad} \langle F_0 \rangle)_{\chi_i})$ of $\chi_i$ and extensions $\hat{\psi}_i \in \Irr((\N_{G_{\ad} \langle F_0 \rangle}( \overline{D}) )_{\psi_i}))$ of $\psi_i$ lying in Harris--Knörr corresponding blocks. It follows from \cite[Lemma 2.15]{LocalRep} and \cite[Proposition 4.4]{LocalRep} that the inductive Alperin--McKay condition (in the version of \cite[Definition 4.2]{LocalRep}) is satisfied for the block $b$ of $G$.
\end{proof}

\subsection{The inductive Alperin--McKay condition for $\mathrm{SL}_2(\mathbb{F}_q)$}

\begin{proposition}\label{block A1}
	The iAM-condition holds for the principal $2$-block of $\G^F=\mathrm{SL}_2(\mathbb{F}_q)$ relative to its defect group.
\end{proposition}

\begin{proof}
By \cite{Breuer} we can assume that $\G^F$ has non-exceptional Schur multiplier.
We follow the proof (and notation) of Proposition \ref{block E7}. Using \cite[Theorem 21.14]{MarcBook} it's easy to see that the principal block $b$ is the unique $2$-block of maximal defect. Hence, $\Irr_0(b)=\Irr_{p'}(G)$ and $\Irr_0(B)=\Irr_{p'}(\N_G(D))$, where $B$ is the Brauer correspondent of $b$. 

 We place ourselves in the situation of \cite[Section 15]{IMN}, where it was shown that $S=G/\mathrm{Z}(G)$ is McKay-good. More precisely, it was shown that there exists an intermediate subgroup $\N_G(D) \subset H$ and an $\mathrm{Aut}(G)_H$-equivariant bijection $\Irr_{p'}(G) \to \Irr_{p'}(H)$. Modifying their bijection we obtain an $\mathrm{Aut}(G)_D$-equivariant bijection $\Irr_0(G,b) \to \Irr_0(\N_G(D),B)$. This bijection preserves central characters since every $2'$-character of $G$ respectively $\N_G(D)$ has the $2$-group $\Z(G)$ in its kernel. One then checks that all characters of $\Irr_0(S,\overline{b})$ extend to their inertia group in $G_{\mathrm{ad}} \langle F_0 \rangle$ (the stabilizer is either cyclic or the characters are unipotent). Similarly, one also checks that the characters of $\Irr_0(\N_S(\overline D),\overline{B})$ extend to their inertia group: If $D$ is self-normalizing then every character in $\Irr_0(\N_S(\overline{D}),\overline{B})$ is linear and one can check the claim by explicit computations. Otherwise, both characters of $\Irr_0(\mathrm{N}_G(D),B)$ with non-cyclic inertia group in $\mathrm{Out}(G)$ are unipotent characters of $\N_G(D) \cong \mathrm{SL}_2(3)$ and their extension to the inertia group follows again from explicit computations. Now using similar arguments as in Proposition \ref{block E7} shows that the block $b$ is AM-good.
\end{proof}

\subsection{Groups with exceptional Schur multiplier}

After having checked the iAM-condition for all quasi-isolated $2$-blocks of groups of Lie type with almost abelian defect group in non-defining characteristic we are left to show the iAM-condition for the remaining simple groups and blocks under consideration. 

\begin{proposition}\label{exceptional}
	Suppose that $S$ is a simple group  not of Lie type in odd characteristic with exceptional Schur multiplier. Then the inductive Alperin--McKay condition holds for all $2$-blocks with almost abelian defect group of the universal covering group $X$ of $S$.
\end{proposition}

\begin{proof}
	Observe first that the iAM-condition holds for all simple sporadic groups by \cite{Breuer}. Assume that $S$ is a simple group of Lie type in characteristic different from $2$ with exceptional Schur multiplier and let $X$ be its universal covering group. Then $S$ is one of the following groups: $A_1(9)$, ${}^2 A_3(3)$, $B_3(3)$ or $G_2(3)$. According to \cite{Breuer} both $A_1(9) \cong A_6$ and $G_2(3)$ are AM-good. Observe that the blocks of the universal covering groups of ${}^2 A_3(3)$ and $B_3(3)$ with maximal defect have a defect group which is not almost abelian. Recall that every defect group of $X$ contains $\mathrm{Z}(X)_\ell$. An inspection of the blocks of the universal covering group of ${}^2 A_3(3)$ using \cite{Breuer2} now shows that all $2$-blocks with almost abelian defect groups have central defect. Therefore, there is nothing to check in this case. For the universal covering group of $B_3(3)$ the same arguments easily rule out all blocks except the blocks denoted by $2$, $3$, $8$ and $9$ in \cite{Breuer2}.
	
	We claim that the blocks $2$, $8$ and $9$ of $X$ have a defect group which is not almost abelian. For this let $b$  be any of these blocks. Consider a central extension $X \to X'$ whose kernel is of order $2$. According to \cite[Theorem 9.10]{Navarro} we obtain a bijection $b \to \overline{b}$ between blocks of $X$ and $X'$. The block of $X'$ corresponding to $b$ under this bijection has according to \cite[Theorem 9.10]{Navarro} a defect group of order $8$ and still two Brauer characters. The groups $C_8$ and $C_2 \times C_4$ admit no automorphisms of odd order. By the remarks following \cite[Theorem 1.30]{Sambale} we therefore deduce that these groups arise as defect groups of nilpotent blocks only. Moreover, any block with defect group $C_2 \times C_2 \times C_2$ can't have exactly two Brauer characters, see \cite[Theorem 13.1]{Sambale}. Since the block $\overline{b}$ is not nilpotent (it has $2$ Brauer characters) its defect group is therefore not abelian. Since $X'$ is a $3$-cover of $S$, we deduce that the block $b$ cannot have an almost abelian defect group.
	
	It therefore remains to consider the block labeled $3$ in \cite{Breuer2}. This block can however be consider as a block of the $2$-cover of $S$, i.e. as a block of $\G^F$, and can therefore be treated as a block with non-exceptional covering group. Using the reduction theorem in \cite{Jordan2} we can thus conclude that this block is AM-good. 
	
	
	
	For alternating groups the inductive AM condition is known to hold by the main result of \cite{Denoncin} and \cite[Corollary 8.3]{IAM}. By \cite{MalleAlperin} the iAM-condition holds for Suzuki and Ree groups.
	
	Let us finally assume that $S$ is a simple group of Lie type defined over a field of characteristic $2$. Let $\G^F$ be such that $\G^F/\mathrm{Z}(\G^F) \cong S$. As before, let $X$ be the universal covering group of $S$. Then there exists a surjective homorphism $X \to \G^F$ whose kernel is the Sylow $2$-group $Z$ of $\mathrm{Z}(\hat{G})$. According to \cite[Theorem 9.10]{Navarro} we obtain a bijection $B \to \overline{B}$ between blocks of $\hat{G}$ and $\G^F$ which maps blocks with almost abelian defect to each other.
	
	 It is known (see \cite[Theorem 6.18]{MarcBook}) that the only $2$-blocks of $\G^F$ are blocks of maximal defect and blocks of height zero. We must therefore consider the cases where a Sylow $2$-subgroup of $S$ is abelian. This is however precisely the case when $S \cong A_1(\pm 2^f)$. Again by the work of Breuer \cite{Breuer} we know that $S$ is AM-good whenver $S$ has an exceptional Schur multiplier. Using the properties of the bijection constructed in \cite{S12} in conjunction with \cite[Theorem 4.1]{CS14} shows that the principal block of $S$ is AM-good.
\end{proof}

\begin{theorem}\label{2blocks}
	Suppose that $S$ is a simple group of Lie type in characteristic unequal to $2$ with non-exceptional Schur multiplier. Then the inductive Alperin--McKay condition holds for all $2$-blocks with almost abelian defect group of the universal covering group $X$ of $S$.
\end{theorem}

\begin{proof}
	Let $\G$ be a simple, simply connected algebraic group and $F: \G \to \G$ a Frobenius endomorphism such that $S=\G^F/ \mathrm{Z}(\G^F)$. We fix a $2$-block $b$ of $\G^F$ with almost abelian defect group. We want to show that $b$ is AM-good relative to its defect group.
	
	Using the proof of \cite[Theorem 3.12]{Jordan2} we see that the statement of \cite[Theorem 3.12]{Jordan2} can be adapted to our situation as follows: If \cite[Hypothesis 3.3]{Jordan2} holds for all blocks with almost abelian defect groups then the block $b$ is AM-good. In other words, we can assume that the $2$-block $b$ of $\G^F$ is strictly quasi-isolated.  According to Lemma \ref{quasi2} all these blocks are unipotent unless $\G$ is of type $A$. Suppose therefore now that $\G$ is of type $A$ and $s \neq 1$. The assumptions of Lemma \ref{linear prime} applies. We can therefore use the proof of Theorem \ref{maintheoremA} shows that we can also assume in this case that the block is unipotent. It therefore suffices to check that the iAM-condition holds for the unipotent blocks occuring in Lemma \ref{quasi2}. Using Remark \ref{unipotent almost abelian} we see that this was checked in Proposition \ref{block E7} and Proposition \ref{block A1}.
\end{proof}

\begin{theorem}\label{Alperin McKay}
	The Alperin--McKay conjecture holds for all $2$-blocks of almost abelian defect.
\end{theorem}

\begin{proof}
	According to Proposition \ref{modifiedSpaeth} to show the theorem it is sufficent to prove the iAM-condition for all $\ell$-blocks with almost abelian defect of the universal covering group of a simple finite group. The iAM-condition in these cases has been verified in Theorem \ref{2blocks} and Proposition \ref{exceptional}.
\end{proof}


\begin{thebibliography}{10}
	
	\bibitem{Benson}
	David~J. Benson.
	\newblock {\em Representations and cohomology. {I}}, volume~30 of {\em
		Cambridge Studies in Advanced Mathematics}.
	\newblock Cambridge University Press, Cambridge, second edition, 1998.
	
	\bibitem{Bonnafe}
	C\'edric Bonnaf\'e.
	\newblock Quasi-isolated elements in reductive groups.
	\newblock {\em Comm. Algebra}, 33(7):2315--2337, 2005.
	
	\bibitem{Bonnafe2}
	C\'{e}dric Bonnaf\'{e}.
	\newblock Sur les caract\`eres des groupes r\'{e}ductifs finis \`a centre non
	connexe: applications aux groupes sp\'{e}ciaux lin\'{e}aires et unitaires.
	\newblock {\em Ast\'{e}risque}, (306), 2006.
	
	\bibitem{Dat}
	C\'edric Bonnaf\'e, Jean-Fran\c{c}ois Dat, and Rapha\"el Rouquier.
	\newblock Derived categories and {D}eligne--{L}usztig varieties {II}.
	\newblock {\em Ann. of Math. (2)}, 185(2):609--670, 2017.
	
	\bibitem{Godement}
	C\'{e}dric Bonnaf\'{e} and Rapha\"{e}l Rouquier.
	\newblock Coxeter orbits and modular representations.
	\newblock {\em Nagoya Math. J.}, 183:1--34, 2006.
	
	\bibitem{Breuer}
	Thomas Breuer.
	\newblock Computations for some simple groups.
	\newblock http://www.math.rwth-aachen.de/
	Thomas.Breuer/ctblocks/doc/overview.html.
	
	\bibitem{Breuer2}
	Thomas Breuer.
	\newblock Decomposition matrices.
	\newblock http://www.math.rwth-aachen.de/~moc/decomposition/.
	
	\bibitem{Brough}
	Julian {Brough} and Britta {Sp{\"a}th}.
	\newblock {On the Alperin-McKay conjecture for simple groups of type
		$\mathrm{A}$}.
	\newblock {\em arXiv e-prints}, page arXiv:1901.04591, Jan 2019.
	
	\bibitem{CabEng}
	Marc Cabanes and Michel Enguehard.
	\newblock Unipotent blocks of finite reductive groups of a given type.
	\newblock {\em Math. Z.}, 213(3):479--490, 1993.
	
	\bibitem{Marc}
	Marc Cabanes and Michel Enguehard.
	\newblock On blocks of finite reductive groups and twisted induction.
	\newblock {\em Adv. Math.}, 145(2):189--229, 1999.
	
	\bibitem{MarcBook}
	Marc Cabanes and Michel Enguehard.
	\newblock {\em Representation theory of finite reductive groups}, volume~1 of
	{\em New Mathematical Monographs}.
	\newblock Cambridge University Press, Cambridge, 2004.
	
	\bibitem{CS14}
	Marc Cabanes and Britta Sp\"{a}th.
	\newblock On the inductive {A}lperin-{M}c{K}ay condition for simple groups of
	type A.
	\newblock {\em J. Algebra}, 442:104--123, 2015.
	
	\bibitem{Denoncin}
	David Denoncin.
	\newblock Inductive {AM} condition for the alternating groups in characteristic
	2.
	\newblock {\em J. Algebra}, 404:1--17, 2014.
	
	\bibitem{DM}
	Fran\c{c}ois Digne and Jean Michel.
	\newblock {\em Representations of finite groups of {L}ie type}, volume~21 of
	{\em London Mathematical Society Student Texts}.
	\newblock Cambridge University Press, Cambridge, 1991.
	
	\bibitem{DM2}
	Fran\c{c}ois Digne and Jean Michel.
	\newblock Groupes r\'{e}ductifs non connexes.
	\newblock {\em Ann. Sci. \'{E}cole Norm. Sup. (4)}, 27(3):345--406, 1994.
	
	\bibitem{twoblocks}
	Charles~W. Eaton, Radha Kessar, Burkhard K\"{u}lshammer, and Benjamin Sambale.
	\newblock 2-blocks with abelian defect groups.
	\newblock {\em Adv. Math.}, 254:706--735, 2014.
	
	\bibitem{Enguehard}
	Michel Enguehard.
	\newblock Sur les {$l$}-blocs unipotents des groupes r\'{e}ductifs finis quand
	{$l$} est mauvais.
	\newblock {\em J. Algebra}, 230(2):334--377, 2000.
	
	\bibitem{IMN}
	I.~Martin Isaacs, Gunter Malle, and Gabriel Navarro.
	\newblock A reduction theorem for the {M}c{K}ay conjecture.
	\newblock {\em Invent. Math.}, 170(1):33--101, 2007.
	
	\bibitem{KessarMalle}
	Radha Kessar and Gunter Malle.
	\newblock Quasi-isolated blocks and {B}rauer's height zero conjecture.
	\newblock {\em Ann. of Math. (2)}, 178(1):321--384, 2013.
	
	\bibitem{Cabanesgroup}
	Radha Kessar and Gunter Malle.
	\newblock Lusztig induction and {$\ell$}-blocks of finite reductive groups.
	\newblock {\em Pacific J. Math.}, 279(1-2):269--298, 2015.
	
	\bibitem{Robinson}
	Reinhard Kn\"{o}rr and Geoffrey~R. Robinson.
	\newblock Some remarks on a conjecture of {A}lperin.
	\newblock {\em J. London Math. Soc. (2)}, 39(1):48--60, 1989.
	
	\bibitem{Kon05}
	A.~S. Kondrat'~ev.
	\newblock Normalizers of {S}ylow 2-subgroups in finite simple groups.
	\newblock {\em Mat. Zametki}, 78(3):368--376, 2005.
	
	\bibitem{cyclic}
	Shigeo Koshitani and Britta Sp\"{a}th.
	\newblock The inductive {A}lperin-{M}c{K}ay and blockwise {A}lperin weight
	conditions for blocks with cyclic defect groups and odd primes.
	\newblock {\em J. Group Theory}, 19(5):777--813, 2016.
	
	\bibitem{MalleUnip}
	Gunter Malle.
	\newblock Extensions of unipotent characters and the inductive {M}c{K}ay
	condition.
	\newblock {\em J. Algebra}, 320(7):2963--2980, 2008.
	
	\bibitem{MalleAlperin}
	Gunter Malle.
	\newblock On the inductive {A}lperin-{M}c{K}ay and {A}lperin weight conjecture
	for groups with abelian {S}ylow subgroups.
	\newblock {\em J. Algebra}, 397:190--208, 2014.
	
	\bibitem{MT}
	Gunter Malle and Donna Testerman.
	\newblock {\em Linear algebraic groups and finite groups of {L}ie type}, volume
	133 of {\em Cambridge Studies in Advanced Mathematics}.
	\newblock Cambridge University Press, Cambridge, 2011.
	
	\bibitem{Marcus}
	Andrei Marcus.
	\newblock On equivalences between blocks of group algebras: reduction to the
	simple components.
	\newblock {\em J. Algebra}, 184(2):372--396, 1996.
	
	\bibitem{MuraiAlperin}
	Masafumi Murai.
	\newblock A remark on the {A}lperin-{M}ckay conjecture.
	\newblock {\em J. Math. Kyoto Univ.}, 44(2):245--254, 2004.
	
	\bibitem{Murai}
	Masafumi Murai.
	\newblock On blocks of normal subgroups of finite groups.
	\newblock {\em Osaka J. Math.}, 50(4):1007--1020, 2013.
	
	\bibitem{NavarroBook}
	Gabriel Navarro.
	\newblock {\em Characters and blocks of finite groups}, volume 250 of {\em
		London Mathematical Society Lecture Note Series}.
	\newblock Cambridge University Press, Cambridge, 1998.
	
	\bibitem{Navarro}
	Gabriel Navarro, Pham~Huu Tiep, and Alexandre Turull.
	\newblock Brauer characters with cyclotomic field of values.
	\newblock {\em J. Pure Appl. Algebra}, 212(3):628--635, 2008.
	
	\bibitem{Rickard}
	Jeremy Rickard.
	\newblock Splendid equivalences: derived categories and permutation modules.
	\newblock {\em Proc. London Math. Soc. (3)}, 72(2):331--358, 1996.
	
	\bibitem{Rouquier3}
	Rapha\"el Rouquier.
	\newblock The derived category of blocks with cyclic defect groups.
	\newblock In {\em Derived equivalences for group rings}, volume 1685 of {\em
		Lecture Notes in Math.}, pages 199--220. Springer, Berlin, 1998.
	
	\bibitem{Lucas}
	Lucas {Ruhstorfer}.
	\newblock {On the Bonnaf{\'e}--Dat--Rouquier Morita equivalence}.
	\newblock {\em arXiv e-prints}, page arXiv:1812.07354, Dec 2018.
	
	\bibitem{Jordan}
	Lucas {Ruhstorfer}.
	\newblock {Derived equivalences and equivariant Jordan decomposition}.
	\newblock {\em arXiv e-prints}, page arXiv:2010.04468, October 2020.
	
	\bibitem{Jordan2}
	Lucas {Ruhstorfer}.
	\newblock {Jordan Decomposition for the Alperin-McKay Conjecture}.
	\newblock {\em arXiv e-prints}, page arXiv:2010.04499, October 2020.
	
	\bibitem{Sambale}
	Benjamin Sambale.
	\newblock {\em Blocks of finite groups and their invariants}, volume 2127 of
	{\em Lecture Notes in Mathematics}.
	\newblock Springer, Cham, 2014.
	
	\bibitem{S12}
	Britta Sp\"{a}th.
	\newblock Inductive {M}c{K}ay condition in defining characteristic.
	\newblock {\em Bull. Lond. Math. Soc.}, 44(3):426--438, 2012.
	
	\bibitem{IAM}
	Britta Sp\"{a}th.
	\newblock A reduction theorem for the {A}lperin-{M}c{K}ay conjecture.
	\newblock {\em J. Reine Angew. Math.}, 680:153--189, 2013.
	
	\bibitem{LocalRep}
	Britta Sp\"{a}th.
	\newblock Reduction theorems for some global-local conjectures.
	\newblock In {\em Local representation theory and simple groups}, EMS Ser.
	Lect. Math., pages 23--61. Eur. Math. Soc., Z\"{u}rich, 2018.
	
	\bibitem{Taylor2}
	Jay Taylor.
	\newblock The structure of root data and smooth regular embeddings of reductive
	groups.
	\newblock {\em Proc. Edinb. Math. Soc. (2)}, 62(2):523--552, 2019.
	
	\bibitem{Thevenaz}
	Jacques Th\'{e}venaz.
	\newblock Extensions of group representations from a normal subgroup.
	\newblock {\em Comm. Algebra}, 11(4):391--425, 1983.
	
\end{thebibliography}

\end{document}